\documentclass{amsart}
\usepackage{amssymb,euscript,amsmath, mathrsfs}
\usepackage{enumerate} 
\usepackage{bm}
\usepackage{rotating}
\usepackage[table]{xcolor}

\makeindex

\newcounter{ENUM}
\newcommand{\itm}{\item}
\newenvironment{ilist}[1][0]{\renewcommand{\theENUM}{\roman{ENUM}}\renewcommand{\itm}{\addtocounter{ENUM}{1}\item[(\theENUM)]}\begin{itemize}\setcounter{ENUM}{#1}}{\end{itemize}}

\newcommand{\margh}[1]{}

\def\risom{\overset{\sim}{\rightarrow}}

\input xy
\xyoption{all}
\CompileMatrices

\def\ZZ{{\mathbb Z}}

\def\PP{{\mathbb P}}

\def\cM{{\mathcal M}}

\def\sL{{\mathscr L}}

\def\sO{{\mathscr O}}

\def\fg{{\mathfrak g}}

\def\dv{\operatorname{div}}

\def\Pic{\operatorname{Pic}}

\def\ord{\operatorname{ord}}

\def\MR{\operatorname{MR}}

\def\Sym{\operatorname{Sym}}

\def\md{\operatorname{md}}

\newtheorem{thm}{Theorem}[section]
\newtheorem{prop}[thm]{Proposition}
\newtheorem{lem}[thm]{Lemma}
\newtheorem{cor}[thm]{Corollary}
\newtheorem{conj}[thm]{Conjecture}

\theoremstyle{definition}
\newtheorem{defn}[thm]{Definition}

\newtheorem{ex}[thm]{Example}

\theoremstyle{remark}
\newtheorem{notn}[thm]{Notation}
\newtheorem{rem}[thm]{Remark}

\numberwithin{equation}{section}

\begin{document}
\title{Limit linear series and ranks of multiplication maps}
\author{Fu Liu}
\author{Brian Osserman}
\author{Montserrat Teixidor i Bigas}
\author{Naizhen Zhang}
\begin{abstract} We develop a new technique to study ranks of multiplication maps for linear series via limit linear series
and degenerations to chains of elliptic curves. 
We prove an elementary criterion and apply it to  proving cases of the Maximal Rank Conjecture.
We give a new proof of the case of quadrics, and also treat several families in the case of cubics.
 Our proofs do not require restrictions on direction of approach, so we recover new information on the locus
in the moduli space of curves on which the maximal rank condition fails.
\end{abstract}

\thanks{Fu Liu was partially supported by NSF grant DMS-1265702. Brian 
Osserman was partially supported by a grant from the Simons Foundation 
\#279151.}

\maketitle

\section{Introduction}

The classical Brill-Noether theorem states that if we are given  $g,r,d \geq 0$, a general curve $X$ of genus $g$ carries a linear series
$(\sL,V)$ of rank $r$ and degree $d$ if and only if the quantity
$$\rho:=g-(r+1)(g-d+r)$$
is nonnegative \cite{g-h1}. 
Eisenbud and Harris proved that (at least in characteristic $0$) when $r \geq 3$, 
a general such linear series on $X$ will define an  imbedding of $X$ as a nondegenerate curve of degree $d$ in $\PP^r$  \cite{e-h4}. 
One of the most basic questions one might then ask is: what are the degrees of  the equations defining $X$? 
More precisely, for each $m \geq 2$, what is the dimension of the space of homogeneous polynomials of degree
$m$ vanishing on the image of $X$?
The question is about the dimension of the kernel of the natural restriction map
\begin{equation}\label{eq:restrict}
\Gamma(\PP^r,\sO(m)) \to \Gamma(X,\sL^{\otimes m} ).
\end{equation}
The dimension of the source space is $\binom{r+m}{m}$ while the dimension of the target space is $md+1-g$. 
The \textit{Maximal Rank Conjecture} states that the rank of this map is as large as possible,
 or equivalently, the  kernel of this map is as small as possible.

\begin{conj}\label{conj:mrc}
If $X$  is a generic curve with a generic immersion in $\PP^r$ for any $m\ge 2$, the rank of the restriction map \eqref{eq:restrict} is
$$\min\left\{\binom{r+m}{m},md+1-g\right\},$$
\end{conj}

At least in part, this conjecture goes back to work of Noether in late 1800's,
and of Severi in the early 1900's, but it was stated explicitly by Harris
in 1982, and has received considerable attention since then. Partial results 
are due to Ballico and Ellia \cite{ba3} \cite{ba4} \cite{ba5} 
\cite{b-e1} \cite{b-e2},
Voisin \cite{vo4},
Farkas \cite{fa2},
Teixidor \cite{te9}, 
Larson \cite{la4},
and most recently, Jensen and Payne \cite{j-p2}. 
  These results were in some cases motivated directly by the conjecture, but in other cases by a variety of applications, including to surjectivity of the Wahl  map,
   to higher-rank Brill-Noether theory, and to the birational geometry of moduli spaces of curves. 
Subsequently, Aprodu and Farkas  \cite{a-f1} introduced a \textit{Strong Maximal Rank Conjecture} motivated by applications to moduli spaces of curves  (see 5.4 in \cite{a-f1}).
    Farkas and Ortega  then developed the relationship to higher-rank Brill-Noether theory \cite{f-o1}.
Taken together, the above-mentioned papers have treated Conjecture  \ref{conj:mrc} in the following cases: when $d \geq r+g$; when $r=3$ or $r=4$; when $m=2$; 
when $d$ is sufficiently large relative to $r$ and $m$; and in several additional ranges of cases for $m=3$, including 
many cases with $r=5$. It is also important to note that if  \eqref{eq:restrict} is known to be surjective for a given $m$ and a given
linear series on a given curve, then surjectivity also follows for all larger $m$ (and the same linear series); see the proof of Theorem 1.2 of \cite{j-p2}.
Thus, knowing for instance the $m=2$ case mentioned above, we conclude that for any case $(g,r,d)$ with $\binom{r+2}{2} \geq 2d+1-g$, 
the Maximal Rank Conjecture holds for all $m$.
 After this paper was submitted, Larson proved the full (weak)  maximal rank conjecture \cite{la7}, \cite{la8}.
  Our main results are as follows.

\begin{thm}\label{thm:main} Given $g,r,d$ with $r \geq 3$, $r+g>d$, and $\rho \geq 0$, the Maximal Rank Conjecture \ref{conj:mrc} holds under the  following conditions:
\begin{ilist}
\itm when $m=2$;
\itm when $m=3$, and either $r=3$ with $g \geq 7$, $r=4$ with $g \geq 16$,
or $r=5$ with $g \geq 26$;
\itm when $g \geq (r+1)\left((m+1)^{r-1}-r\right)$;
\itm when $m \geq 3$, and either $g-d+r=1$ with $2r-3 \geq \rho+1$,
or $r+g-d=2$ with $r \geq 4$ and $2r-3 \geq \rho+2$.
\end{ilist}
\end{thm}

Most results about the   (weak) maximal rank conjecture have been obtained by deforming a special curve inside a projective space.
We instead translate the problem into a  question depending on the curve alone (rather than any immersion). 
Using the identification  $\Gamma(\PP^r,\sO(m))=\Sym^m V,$ the map   \eqref{eq:restrict}  can be interpreted in terms of the linear  series  as:
\begin{equation}\label{eq:multiply}\Sym^m V \to \Gamma(X,\sL^{\otimes m}).\end{equation}
In order to prove any given case of the Maximal Rank Conjecture, it is enough to produce a single smooth curve $X$
 for which the space of linear series of given rank and degree
has the expected dimension $\rho$, and a single linear series on $X$ such that \eqref{eq:restrict} has the predicted rank (see Corollary \ref{cor:lls-reduction} for details).
We use the theory of limit linear series studying ranks of multiplication maps by degenerating to a chain of elliptic curves
applying the fundamental smoothing theorem of Eisenbud and Harris \cite{e-h1}
 together with  substantial input from the alternative  approach to limit linear series developed in \cite{os8} and \cite{os20}.
 Previous approaches using limit linear series to study multiplication maps had focused on  injectivity, 
 considering a hypothetical nonzero element of the kernel, and  deriving a contradiction (see for instance \cite{e-h2} and \cite{te9}).
Here, instead of showing that the kernel of the map is small, we prove that the image is large, a strategy also used in the tropical context \cite{j-p1} \cite{j-p2}.

Our approach is relatively self-contained.
In \S \ref{sec:nondegen} we prove some auxiliary results related to elliptic curves with two marked points.
In  section \ref{sec:reddeg}, we reduce the result to the singular curve.
In \S \ref{sec:elem-1} we use limit linear series to prove  some criteria for independence of sections.
In the remaining sections, we apply these results to prove the various cases of Theorem \ref{thm:main}: 
in \S \ref{sec:injective} we make some observations on injectivity including a proof of case (iii); 
in \S \ref{sec:m2} we prove the case $m=2$; in \S \ref{sec:surjective} we make some observations on surjectivity and prove case (iv); 
and finally, in \S \ref{sec:m3} we prove the $m=3$ cases of Theorem \ref{thm:main}.
Our arguments apply over a base field of any characteristic, although 
 to simplify the exposition and to make use of other results in our applications, we assume characteristic $0$.
We do not need our base field to be algebraically closed either, but it will simplify the arguments to assume it is.

Our  method are quite flexible. They can be used  for  studying the Strong Maximal Rank Conjecture (see \cite{smr}) and even multiplication maps for arbitrary linear series.
While we will not need it in this paper,  it potentially allows us to take into  account direction of approach to the special  curve $X_0$ (see Remark \ref{dirap}). 
In particular, it could be used  to prove maximal rank near a curve that itself does not satisfy the maximal rank condition.

\subsection*{Acknowledgements} We would like to thank Gavril Farkas for bringing to our attention the difficulties in applying the usual Eisenbud-Harris approach to multiplication maps
 in the situation of the Maximal Rank Conjecture.

\section{Nondegeneracy on twice-marked elliptic curves}\label{sec:nondegen}

In this section, we study maps from  elliptic curves to projective space determined by comparing values of certain tuples of sections of a line bundle at  points
$Q$ and $P$, as we let the point $Q$ vary. 
We describe these maps explicitly, showing in the process that they are morphisms, and proving that they are non-degenerate in a family of cases of interest for the Maximal Rank Conjecture. 

Given a non-singular genus-$1$ curve $C$ and distinct $P,Q$ on $C$, and integers $c,d \geq 0$, let $\sL=\sO_C(cP+(d-c)Q)$. 
Then for any $a,b \geq 0$ with $a+b=d-1$, there is a  section of $\sL$ unique up to scaling   vanishing to order at least $a$ at $P$ and at least $b$ at $Q$. 
Thus, we have a uniquely determined point $R$ such that the divisor of the aforementioned section is $aP+bQ+R$.
In particular,   $R=P$ if and only if $Q-P$ is $|a+1-c|$-torsion, and $R=Q$ if and only if $Q-P$ is $|a-c|$-torsion. 
Note that this  makes sense even when $Q=P$ (in which case $R=Q=P$). 
To avoid trivial cases, we will assume that $a \neq c-1$, and $b \neq d-c-1$.

\begin{notn}\label{sit:nondegen} Fix $m \geq 2$, and set positive integers 
\[ c, d,\  a_1,\dots,a_m, \ b_1,\dots,b_m, \ a_1',\dots,a_m', \ b_1',\dots,b_m'\]
 \[\text{ s.t. } a_i+b_i=d-1, a_i-c \neq 0,-1 , a_i'+b_i'=d-1,  a'_i-c \neq 0,-1\  \forall i,\  \sum_i a_i = \sum_i a_i'.\]
 Let $P,Q \in C$  satisfying  $Q-P$ is not $|a_i-c|$- or $|a_i+1-c|$- or $|a'_i -c|$- or $|a_i'+1-c|$-torsion for any $i$, 
Let  $s_i$ be sections with divisors $a_iP+b_i Q +R_i$, and $s_i'$ with divisors $a_i'P+b_i' Q+R_i'$.
 Then, $s = s_1 \otimes \cdots \otimes s_m, s' = s_1' \otimes \cdots \otimes s_m' \in \Gamma(C,\sL^{\otimes m})$ have divisors
$$(\sum_i a_i) P + (\sum_i b_i) Q + R_1 + \dots + R_m,\ (\sum_i a_i') P + (\sum_i b_i') Q + R_1' + \dots + R_m'.$$
As $\sum_i a_i = \sum_i a_i'$,  $R_1+\dots+R_m \sim R_1'+ \dots + R_m'$. 
Let $g$ be the rational function unique up to scaling, such that
$$\dv g = R_1+ \dots + R_m - R_1' - \dots - R_m'.$$
then $g(P)$ and $g(Q)$ are both in $k^\times$.
The ratio $g(Q)/g(P) \in k^\times$ is independent of scaling $g$, so is canonically determined by the choice of $P,Q$ and the discrete data. 
Fix $P$, for a given $Q\in C$, denote by $R_i^Q$, $R_i'^Q$ and $g^Q$ the points and rational function determined as above by $P$ and $Q$.
Let $U$ be the open subset of $C$ consisting of all $Q$ such that $Q-P$ is not $|a_i-c|$- or $|a_i+1-c|$- or $|a'_i -c|$- or $|a_i'+1-c|$-torsion for any $i=1,\dots,m$.

Let $k$ be an integer. Denote by $L_1, \dots L_{k^2}$ the line bundles in  $\Pic^0(C)$  of order a  divisor of $|k|$.
Then, for $X$ a point in $C$,  ${\mathcal O}_C(X)\otimes  L_i$ is a line bundle of degree 1 and therefore can be written as ${\mathcal O}_C(Y_i)$ for a unique $Y_i\in C$.
We will denote $\sum _iY_i$ by $X+T[k]$.
\end{notn}

With the notation above, for all $Q \in U$, we get a $g^Q(Q)/g^Q(P) \in k^\times$. The main  technical result of this section is then the following characterization of the resulting function.

\begin{lem}\label{lem:nondegen} With notations as in  \ref{sit:nondegen}, the  function $f:U \to k^\times$ given by $Q \mapsto g^Q(Q)/g^Q(P)$ determines a rational function on $C$. 
We then have
\begin{equation*} \dv f  =
\sum_{i=1}^m ((P+T[|a_i-c|])-(P+T[|a_i'-c|])-(P+T[|a_i+1-c|])+(P+T[|a_i'+1-c|])),
\end{equation*}
\end{lem}

\begin{proof} Consider the  divisor $\bar R_i$ (respectively $\bar R_i'$) on $C \times C$ consisting of points $(Q,R_i^Q)$ (respectively, $(Q,R_i'^Q)$).
We can regard $\bar R_i$ as the graph of the morphism $C \to C$ sending $Q$ to $P+(a+1-c)(Q-P)$ with $a=a_i$ (respectively  $a=a_i'$ for $\bar R_i'$).
Now, set $Z= \sum_i  (P+T[|a_i+1-c|])$ and $Z'= \sum_i  (P+T[|a_i'+1-c|])$.

Our first claim is that $\bar R_1+\dots+\bar R_m + Z' \times C \sim \bar R_1'+\dots+\bar R'_m + Z \times C$.
The restriction of $\bar R_1+\dots+\bar R_m $ to  any fiber $\{Q\} \times C$ is $R_1^Q+\dots+R_m^Q \sim R_1^{'Q}+\dots+R_m^{'Q}$ 
which in turn is the restriction of $\bar R_1'+\dots+\bar R_m' $ to  the fiber $\{Q\} \times C$.
Hence, $\bar R_1+\dots+\bar R_m - \bar R_1'-\dots-\bar R'_m \sim D \times C$ for some divisor $D$ on $C$.
But we now consider the restriction to $C \times \{P\}$, observing that, by construction,
\begin{equation}\label{eq:z-def}
(\bar R_1+\dots+\bar R_m - \bar R_1'-\dots-\bar R'_m )|_{C \times \{P\}}=(Z - Z') \times \{P\}.
\end{equation}
We  conclude that  $D \sim Z-Z'$, proving our claim.

Now, let $t,t'$ be the sections (unique up to scaling) of $\sO_{C \times C}(\bar R_1+\dots+\bar R_m + C \times Z')$ with divisors
 $\bar R_1+\dots+\bar R_m + C \times Z'$ and $\bar R_1'+\dots+\bar R'_m + C \times Z$ respectively. 
 Our second claim is that there exist choices of $t,t'$ such that the function $f$ is obtained by composing the diagonal map $U \to C \times C$
 with  the rational map $C \times C \dashrightarrow \PP^1_k$ induced by $(t,t')$.
  For $Q \in U$, if we restrict $(t,t')$ to $\{Q\} \times C$, we obtain a rational function with the same zeroes and poles as $g^Q$, and which is hence a valid choice for $g^Q$.
We next observe that if we restrict $(t,t')$ to $C \times \{P\}$, then by \eqref{eq:z-def} after removing base points we have a rational function with no zeroes or poles,
which is thus necessarily constant, equal to some $z \in k^\times$.
Rescaling $t'$ by $z$, we may assume $z=1$, which means that on each $\{Q\} \times C$ for $Q \in U$, the pair $(t,t')$ induces a choice of $g^Q$ with $g^Q(P)=1$. 
Thus for the given $(t,t')$, $g^Q(Q)/g^Q(P)$ is obtained simply by evaluation at $(Q,Q)$, which is the same as saying that $f$ is induced as claimed.

It then follows that $f$ is a rational function on $C$, and the desired description of its divisor likewise follows: indeed, the diagonal meets
any fiber $\{Q\} \times C$ transversely, so the last two terms in the formula come directly from the restrictions of $Z \times C$ and $Z' \times C$, respectively. 
In general the diagonal may not meet the graph of the morphism $Q \mapsto P+(a+1-c)(Q-P)$ transversely, but in any case the intersection is always identified with $P+\Pic^0(C)[a-c]$,
which thus yields the first two terms of the asserted formula for $\dv f$, as desired.
\end{proof}

As a sample application of Lemma \ref{lem:nondegen}, we consider when the function $f$ is nonconstant in the case $m=2$.

\begin{cor}\label{cor:nonconst} In the situation of Lemma \ref{lem:nondegen}, assume further that $m=2$.
  Then the function  $f$ is non-constant if and only if  $\{a_1,a_2\} \neq \{a'_1, a'_2\}$ and $a_1+a_2 \neq 2c-1$.
\end{cor}

\begin{proof} By Lemma \ref{lem:nondegen}, we have that $f$ is constant if and only if
\begin{multline*} 
0=D:=(P+T(|a_1-c|]))-(P+T[|a_1'-c|])-(P+T[|a_1+1-c|])+T[|a_1'+1-c|]
\\+(P+T(|a_2-c|])-(P+T[|a_2'-c|]-(P+T[|a_2+1-c|])+T[|a_2'+1-c|])).
\end{multline*}
Without loss of generality, assume that $a_1 \leq a_2$ and $a_1' \leq a'_2$.
Because we have assumed $a_1+a_2=a_1'+a_2'$, we have $\{a_1,a_2\}=\{a_1',a_2'\}$ if and only if $a_1=a_1'$. 
Obviously, in this case, we have $D=0$.
 Similarly, if $a_1+a_2=2c-1=a_1'+a_2'$, then  $a_1-c=-(a_2+1-c)$, $a_2-c=-(a_1+1-c)$, and similarly for the $a_i'$, giving $D=0$ again.
  On the other hand, if $a_1 \neq a_1'$, we may assume without loss of generality that $a_1 <a_1'$, so that $a_2>a_2'$. 
  In  particular, we have $a_1<a_2$.

If $a_1+a_2 > 2c-1$,  then $a_2+1-c>c-a_1$, but also $a_2+1-c> a_1-c$, so $a_2+1-c>|a_1-c|\geq 0$. 
We likewise have  $a_2'+1-c>|a_1'-c|\geq 0$, but $a_2+1-c>a_2'+1-c$. We  conclude that $|a_2+1-c|$ is the (unique) maximal term appearing in the expression for $D$.
This implies that $f$ has poles at those points, and hence is nonconstant.

Similarly, if $a_1+a_2 < 2c-1$, we see that $|a_1-c|=c-a_1$ is the maximal term appearing in the expression for $D$, implying that $f$ has zeroes and is nonconstant.
\end{proof}

We now consider morphisms to higher-dimensional projective spaces.

\begin{notn}\label{sit:nondegen-2} Fix $m \geq 2$, and $\ell \geq 1$, and for 
$j=0,\dots,\ell$, set numbers $a_1^j,\dots,a_m^j$, $b_1^j,\dots,b_m^j$ 
satisfying:
\[ a_i^j+b_i^j=d-1, a_i^j - c \neq 0,-1 \forall i,j,\ \sum_i a_i^j\text{ is independent of }j.\]
There are sections $s_i^j$ with divisors $a_i^jP+b_i^j Q +R_i^j$, and  forming tensor products yields sections
$s^j = s_1^j \otimes \cdots \otimes s_m^j \in \Gamma(C,\sL^{\otimes m})$, with divisors
$$\left(\sum_i a_i^j\right) P + \left(\sum_i b_i^j\right) Q + R_1^j + \dots + R_m^j,$$
Any two $R_1^j+\dots+R_m^j$ are linearly equivalent.
If $Q-P$ is not $|a_i^j+1-c|$-torsion for any $i,j$, we can normalize the $s^j$, uniquely up to simultaneous scalar, so that their values at $P$ are all the same. 
Then provided that there is some $j$ such that $Q-P$ is not $|a_i^j-c|$-torsion for any $i$, considering  $(s^0(Q),\dots,s^{\ell}(Q))$ gives a well-defined point of $\PP^{\ell}$.
Suppose $P$ is fixed.
 For a given $Q\in C$, denote by $R_i^{j,Q}$ the point determined as above by $P$ and $Q$, and by $f_Q$ the point of $\PP^{\ell}$
determined by $(s^0(Q),\dots,s^{\ell}(Q))$.
Let $U$ be the open subset of $C$ consisting of all $Q$ such that $Q-P$ is not $|a_i^j-c|$- or $|a_i^j+1-c|$-torsion for any $i,j$.
\end{notn}

Our main result is then the following.

\begin{cor}\label{cor:nondegen} The map $U \to \PP^{\ell}$ given by  $Q \mapsto f_Q$ extends to a morphism $f:C \to \PP^{\ell}$.
If further,  all the $a_i^j$ are distinct, $a_1^j+a_2^j \neq 2c-1$, and for each $j$, we have  exactly one $a_i^j$ less than $c$, then $f$ is nondegenerate.
\end{cor}

\begin{proof} Indeed, we can view our map as being given by $(f_0,\dots,f_{\ell-1},1)$, where $f_j$ is the rational function constructed 
in Lemma \ref{lem:nondegen} from the sections $s^j,s^{\ell}$. 
We thus conclude immediately that our map extends to a morphism.
 Moreover, nondegeneracy is equivalent to linear  independence of the rational functions $f_0,\dots,f_{\ell-1},1$, whose zeroes and poles we have completely described.

Now, suppose we have the hypotheses for the nondegeneracy statement.
We may also without loss of generality reorder our data so that
$$a_1^0 < a_1^1 < \dots < a_1^{\ell} < c < a_2^{\ell} < a_2^{{\ell}-1} <  \dots < a_2^0.$$
Then we claim that for each $j<\ell$, if we set $N_j=\max(|a_1^j+1-c|,|a_2^j+1-c|)$, then $f_j$ has poles at the  strict $N_j$-torsion points of $C$, while none of $f_{j+1},\dots,f_{\ell-1}$ do.
The desired linear independence follows.

For the first assertion, we have to see that the zeroes at the $|a_i^j-c|$-torsion and $|a_i^{\ell}+1-c|$-torsion cannot cancel the poles at the $N_j$-torsion. 
Note that $N_j \geq a_2^j+1-c \geq 3$.
Certainly, we have $|a_2^j-c|=a_2^j-c<N_j$,  $|a_2^{\ell}+1-c|=a_2^{\ell}+1-c<N_j$, and  $|a_1^{\ell}+1-c|=c-1-a_l^{\ell}<c-1-a_1^j \leq N_j$, so there is no problem with these. 
Finally, as $|a_1^j-c|$ is relatively prime to $|a_1^j+1-c|$, so if $N_j=|a_1^j+1-c|$, the poles at the $N_j$-torsion cannot be cancelled by the zeroes at the $|a_1^j-c|$-torsion.
But if $N_j > |a_2^j+1-c|$, we must have $|a_1^j-c|-1=|a_1^j+1-c|<N_j$, and we cannot have $|a_1^j-c|=N_j$ because $a_1^j+a_2^j \neq 2c-1$, so we must have $|a^1_j-c|<N_j$, 
and again the  poles cannot be cancelled.

For the second assertion, choose $j'>j$; then $f_{j'}$ has potential poles at the $|a_i^{j'}+1-c|$-torsion and the $|a_i^{\ell}-c|$-torsion. 
But as above, we see that $|a_i^{j'}+1-c|<N_j$ and $|a_i^{\ell}-c|<N_j$ for $i=1,2$, so $f_{j'}$ cannot have poles at the strict $N_j$-torsion, as desired.

\end{proof}

\section{Reduction to the nodal curve}\label{sec:reddeg}

We begin by discussing generalities on the behavior of multiplication maps under degenerations, and the relationship to limit linear series. 
We remark that in order to prove any given case of the Maximal Rank Conjecture, it is enough to produce a single smooth curve $X$ 
for which the space of linear series of given rank and degree has the expected dimension $\rho$, and a single linear series on $X$ such that 
$\Gamma(\PP^r,\sO(m)) \to \Gamma(X,\sL^{\otimes m} )$ has the predicted rank. 
Indeed, while for small $m$ and $d$ the dimension of $\Gamma(X,\sL^{\otimes m} )$ may vary as 
$X$ and $\sL$ vary, if we use the usual trick of twisting up by a  sufficiently ample divisor on $X$, we can re-express the maximal rank condition in arbitrary families as a determinantal condition.
We conclude that over any family of smooth curves, satisfying the maximal rank condition is an open condition in the relative moduli space of linear series.
 Standard dimension arguments imply that this moduli space is open over the base at any point which has fiber dimension $\rho$, 
 proving that under the stated hypotheses, all nearby curves contain a nonempty open subset of linear series satisfying the maximal rank  condition. 
For $\rho \geq 1$, it follows that we have an  open family of curves for which a dense open subset of linear series satisfies the maximal rank  condition
 (note that the initial curve did not need to be Petri general). 
For $\rho=0$, we instead apply the monodromy theorem of Eisenbud and Harris \cite{e-h7} 
to conclude that we have an open family of curves for which  every linear series satisfies the maximal rank condition.

The next step will be to reduce the problem to the case of a singular curve.
Specifically, we will degenerate to a chain of elliptic curves as defined below:

\begin{notn}\label{defX0}
In this section  $X_0$ will  be a curve of compact type  obtained as follows $X_0=Z_1 \cup \dots \cup Z_n$ is a chain of curves, 
$g$ of which have genus one and the rest are rational  and where $Q_i$ on $Z_i$ is glued  to $P_{i+1}$ on $Z_{i+1}$ for $i=1,\dots,n-1$.
In addition, we will assume that $P_i-Q_i$ is not $\ell$-torsion for any $\ell \leq d$.
\end{notn}
We recall the definition  of limit linear series in this context.
\begin{defn}\label{defn:lls}  With the above notation,  a \textbf{limit linear series} of rank $r$ and degree $d$ on $X_0$
is a $n$-uple $(\sL^i,V^i)_{(i=1,\dots, n)}$ of linear series of rank $r$ and degree $d$ on the components $Z_i$ of $X_0$ satisfying the following condition: 
 let  $a^{i}_0<\dots<a^{i}_r$ and  $b^{i}_0>\dots>b^{i}_r$ be the vanishing sequences of $(\sL^i,V^i)$ at $P_i, Q_i$, respectively. 
Then we require that
$$a^{i+1}_j+b^{i}_{j} \geq d \quad \quad \text{ for } j=0,\dots,r.$$

We say that a limit linear series is \textbf{refined} if the above inequality is an equality for all $i$ and $j$.
\end{defn}

 We make choices of line bundles and sections with support on certain components:

\begin{defn}\label{sit:lls-fiber}  With the notations above, for $i=1,\dots,n$ let  $Z'_i$ be the closure of $X_0 \smallsetminus Z_i$. 
Define a line bundle  on $X_0$ by 
 \[ \sO^i =\begin{cases} \sO_{Z_i}(-(Z_i' \cap Z_i)) &\text{ on  } Z_i \\
                                 \sO_{Z'_i}(Z_i\cap Z'_i)    &\text{ on  }  Z' _i \end{cases}\]
Choose  sections $\sigma_i \in \Gamma(X_0,\sO^i)$ which vanish precisely on $Z_i$
 and choose an isomorphism  $\theta:\bigotimes_{i=1,\dots,n} \sO^i \risom \sO_{X_0}$.
\end{defn}

\begin{rem}\label{dirap}
As  $X_0$ is of compact type, each $\sO^i$ is unique up to isomorphism but in general $\sigma_i$ is not unique up to scaling: 
indeed, for  $i=2,\dots, n-1,\ X_0 \smallsetminus Z_i$ is disconnected, then $\sigma_i$ may be scaled independently on each connected component. 
On the other hand, a family induces a choice of $\sigma_i$ (see Proposition \ref{prop:twistors-family}) 
This is potentially useful as it is one way in which direction of approach could be incorporated into our analysis.
\end{rem}

As in   \cite{os20}, we consider  line bundles  of all possible multidegrees (and total degree $d'$) on the reducible curve  and construct maps between them.
Starting with  a limit linear series  $(\sL^i,V^i)_{i=1,\dots, n}$, choose a `base component' $Z_{i_0}$ of $X_0$.
Let $\omega_0$ be the multidegree assigning degree $d$ to $Z_{i_0}$ and degree $0$ to every other component of $X_0$, 
Define $\sL_{\omega_0}$ as the line bundle obtained by gluing the line bundles $\sL^{i_0}$ on $Z_{i_0}$ and $\sL^i(-dQ_i), \ i< i_0, \sL^i(-dP_i), \ i> i_0 $ on $Z_i$. 

Given an arbitrary multidegree $\omega$, there is a unique collection of  nonnegative integers $a_i, i=1,\dots, n$  such that 
at least one $a_i$ is equal to $0$, and such that  $\bigotimes_{i} (\sO^i)^{\otimes a_i}$ has multidegree $\omega-\omega_0$.
Then set 
$$\sL_{\omega}= \sL_{\omega_0} \otimes \left(\bigotimes_i (\sO^{i})^{\otimes a_i}\right).$$

Given another multidegree $\omega'$, if $\bigotimes_{i} (\sO^i)^{\otimes a'_i}$ has multidegree $\omega'-\omega_0$,
 we get a morphism $\sL_{\omega'} \to \sL_{\omega}$ as follows: let  $b = \max_{v} (a'_i - a_i)$, and for each $Z_i$, set 
$c_i=a_i-a'_i+b$. Then all $c_i$ are nonnegative with at least one equal to $0$, and
$$\sL_{\omega}\cong \sL_{\omega'} \otimes \left(\bigotimes_i (\sO^{i})^{\otimes c_i}\right).$$
More precisely, note that since the total degree of both ${\omega}, {\omega}', {\omega}_0$   is the same, $\sum_ia_i=0=\sum_ia'_i$ therefore $b \geq 0$. Then, 
$$\sL_{\omega}\otimes \left(\bigotimes_i (\sO^{i})^{b}\right)=\sL_{\omega'} \otimes \left(\bigotimes_i (\sO^{i})^{\otimes c_i}\right),$$
so we obtain an induced morphism $\sL_{\omega'} \to \sL_{\omega}$ 
from the appropriate tensor product of the $\sigma_i$, together with $\theta^{\otimes b}$. 
This morphism vanishes precisely on the components $Z_i$ of $X_0$ for which $c_i>0$.

Finally, we note that we have restriction maps as follows: given a component $Z_i$, 
let $\omega_i$ be the multidegree having degree $d$ on $Z_i$ and degree $0$ on all other components. 
Then for any multidegree $\omega$, we obtain a morphism $\sL_{\omega} \to \sL^i$, unique up to scalar,
by composing our constructed morphism $\sL_{\omega} \to \sL_{\omega_i}$ with the restriction map $\sL_{\omega_i}|_{Z_v} \risom \sL^i$.
Depending on the choice of $\omega$, this restriction map may vanish uniformly, but this will not happen in most cases of interest (see Proposition \ref{prop:zeroing-pattern}).

It is often useful to consider an alternative encoding of multidegrees as follows: 

\begin{notn}\label{notn:md} Given a tuple $c=(c_2,\dots,c_{n})$ of integers and  a total degree $d'$ (which will be equal to $d$ or $md$ in our situation),
 we obtain a unique multidegree $w_{d'}(c)$ by setting the degree to $c_2$ on $Z_1$, to $c_{i+1}-c_{i}$ on $Z_i$ for $1<i<n$, and to $d'-c_{n}$ on $Z_n$. We write $w(c)$ for $w_{d'}(c)$ where the total degree is fixed within the context. 
\end{notn}

Given a linear series with line bundles  $\sL^i$ on $Z_i$ of degree $d'$, we obtain the line bundle $\sL_{w(c)}$ by gluing together the following:
\begin{itemize}
\item $\sL^1(-(d'-c_2) Q_1)$ on $Z_1$; 
\item $\sL^i(-c_i P_i-(d'-c_{i+1}) Q_i)$ on $Z_i$ for $1<i<n$; 
\item and $\sL^n(-c_{n} P_n)$ on $Z_n$.
\end{itemize}

We describe the maps between different multidegrees as follows.

\begin{prop}\label{prop:zeroing-pattern} Given $c=(c_2,\dots,c_{n})$ and $c'=(c'_2,\dots,c'_{n})$ in $\ZZ^{n-1}$,
for any choice of line bundle $\sL_{\omega_0}$ the natural map $\sL_{w(c')}\to \sL_{w(c)}$ vanishes on a given $Z_i$ if and only if $\epsilon^i_{w',w}=0$
where $\epsilon^i_{w',w}$ is defined as 
$$\epsilon^i_{w',w}=\begin{cases} 0: & \sum_{j=i+1}^{n} (c'_j- c_j) > \min _{1 \leq i' \leq n} \sum_{j=i'+1}^{n} (c'_j-c_j) \\1 : & \text{ otherwise.}\end{cases}$$
In particular, as long as $0 \leq c'_i \leq d'$ for $i=2,\dots,n$ none of the restriction maps $\sL_{w(c')} \to \sL^i$ vanish uniformly.
\end{prop}
\begin{proof} Observe that for any $i \leq n-1$, the multidegree of $\sO^{1,i}:=\bigotimes_{i'=1}^i \sO^{i'}$ is zero on all components except 
$Z_i$ and $Z_{i+1}$; it is $-1$ on $Z_{i}$ and $1$ on $Z_{i+1}$. 
In the notation introduced above with $d'=0$, it is $w(c'')$ where $c''=(c''_2,\dots,c''_n)$ with all $c''_{i'}$ equal to $0$ except that $c''_{i+1}=-1$. 
We can go from $\sL_{w(c')}$ to $\sL_{w(c)}$, by first tensoring $\left(\sO^{1,n-1}\right)^{\otimes c'_n-c_n}$ to get the desired degree on $Z_n$,
 then by $\left(\sO^{1,n-2}\right)^{\otimes c'_{n-1}-c_{n-1}}$ to get the desired degree on $Z_{n-1}$, and so forth.
Thus, we conclude that 
$$\sL_{w(c)}\cong \sL_{w(c')} \bigotimes_{i=1}^{n-1} \left(\sO^{1,i}\right) ^{\otimes c'_{i+1}-c_{i+1}} = \sL_{w(c')} \bigotimes_{i=1}^{n-1} 
\left(\sO^{i}\right) ^{\otimes \sum_{j=i+1}^n c'_j-c_j}.$$
If we set $M=\min_{1 \leq i \leq n} \sum_{j=i+1}^n c'_j-c_j$, then we have
$M \leq 0$ by considering $i=n$, and we can write
$$\sL_{w(c)}\cong  \sL_{w(c')} \bigotimes_{i=1}^{n}  \left(\sO^{i}\right) ^{\otimes (\sum_{j=i+1}^n c'_j-c_j)-M},$$
with every tensor exponent nonnegative. 
Then the morphism $\sL_{w(c')} \to \sL_{w(c)}$ vanishes precisely where the tensor exponents are strictly positive, which is the definition of having
$\epsilon^i_{w',w}=0$.

For the second assertion, the $w$ yielding multidegree concentrated on $Z_i$ is given by $(c_2,\dots,c_n)$ with $c_{i'}=0$ for $i' \leq i$ and $c_{i'}=d'$ for $i' > i$. 
Thus, if $0 \leq c'_{i'} \leq d'$ for all $i'$, we have that $c'_{i'}-c_{i'} \leq 0$ for $i'>i$ and $c'_{i'}-c_{i'} \geq 0$ for $i' \leq i$, 
so $\sum_{j={i'}+1}^{g} (c'_j-c_j)$ achieves its minimum at $i'=i$, and hence $\epsilon^i_{w',w}=1$ in this case.
\end{proof}
\begin{defn}\label{defsteady}
We say that $(w(c'),w(c))$ is \textbf{steady} if there exists $i$ such that 
\[ \text{ for }  j< i, c_j\le c'_j,  \text{ for }   j\ge i, c'_j\le c_j  \]
\end{defn}. 

\begin{rem}\label{rem:zeroing-pattern}
For a steady pair, the definition of the $\epsilon$ and the above proposition are much easier to interpret.
We observe that if for some $i$ we have $c'_i<c_i$, then the map from the multidegree determined by $w'=w(c')$ to the multidegree determined by $w=w(c)$ should vanish
identically on $Z_i$, since we are twisting down more at $P_i$ in the latter. 
Indeed, in this case we have
$$\sum_{j=i+1}^{g}(c'_j-c_j)> \sum_{j=i}^{g}(c'_j-c_j)\geq  \min _{1 \leq i \leq g} \sum_{j=i+1}^{g} (c'_j-c_j)=M,$$
so $\epsilon^i_{w',w}=0$, as we knew from Proposition \ref{prop:zeroing-pattern}.
Similarly, if $d-c'_i < d-c_i$, considering twists by $Q_{i-1}$ we should have vanishing on $Z_{i-1}$, and we see that since $c_i<c'_i$, we have
$$\sum_{j=i}^{g}(c'_j-c_j)> \sum_{j=i+1}^{g}(c'_j-c_j)\geq  \min _{1 \leq i \leq g} \sum_{j=i+1}^{g} (c'_j-c_j)=M,$$
so $\epsilon^{i-1}_{w',w}=0$, again as expected. We conclude that If $c'_i<c_i$ or $d-c'_{i+1}<d-c_{i+1}$, then necessarily $\epsilon^i_{w',w}=0$. 

The converse doesn't hold in general, but it does hold  when the signs of $c'_i-c_i$ are weakly decreasing, so that
there is never a $0$ before a positive number or a negative number before a nonnegative number. 
In this situation, if $c'_i-c_i$ is never $0$, the minimum $M$ occurs at the unique $i$ such that
$c'_i>c_i$ and $c'_{i+1}<c_{i+1}$ (or at $i=1$ if $c'_{i+1}<c_{i+1}$ for all $i$, and at $i=g$ if $c'_i>c_i$ for all $i$).
If $c'_i-c_i=0$ for some $i$, the minimum $M$ occurs for the $i$ such that that $c'_i-c_i=0$ or $c'_{i+1}-c_{i+1}=0$.
In both cases, these are precisely  the $i$ such that $c'_i \geq c_i$ and $d-c'_{i+1} \geq d-c_{i+1}$.
 Moreover, in this situation,  the $i$ for which $\epsilon^i_{w',w}=1$ are contiguous.
\end{rem}

When we defined linear series, we looked at the orders of vanishing at the nodes 
$a^{i}_0<\dots<a^{i}_r$ and  $b^{i}_0>\dots>b^{i}_r$  of  the sections in $(\sL^i,V^i)$ at $P_i, Q_i$, respectively. 
In general, a set of sections will give the orders of vanishing at $P_i$ and a different set will give the vanishing at $Q_i$.
It will be useful when on each component $Z_2,\dots, Z_{n-1}$, there is  a set of sections that can be used at both points :
\begin{defn}\label{defn:chain-adaptable} A limit linear series $(\sL^i,V^i)$ on $X_0$ is \textbf{chain-adaptable} if, for $i=2,\dots,n-1$,
there exist sections $s^i_0,\dots,s^i_r$ in $V^i$ such that
\[ \ord_{P_i} s^i_0< \ord_{P_i} s^i_1<\dots<\ord_{P_i} s^i_r, \ \ord_{Q_i} s^i_0> \ord_{Q_i} s^i_1>\dots>\ord_{Q_i} s^i_r\]
 recovers the vanishing sequence of $V^i$ at $P_i, Q_i$ respectively.
\end{defn}
From  Propositions 5.2.3 and 5.2.6 of \cite{os20} (see also \cite{lstrop} sec 3), we see that for chain-adaptable limit linear series there exist global sections $s_0,...,s_r$ (in different multidegrees) on $X_0$ restricting to $s^i_j$ on $Z_i$ for all $i,j$:
\begin{prop}\label{prop:chain-adaptable} Let $(\sL^i,V^i)$ be a chain-adaptable limit linear series, with $s^i_j$ as in the definition.
Choose also 
\[ s^1_0,\dots,s^1_r\in V^1,  \ \ s^n_0,\dots,s^n_r \in V^n\]
satisfying 
\[ \ord_{Q_1} s^1_r< \ord_{Q_1} s^1_{r-1}<\dots<\ord_{Q_1} s^1_0,\ \  \ord_{P_n} s^n_0< \ord_{P_n} s^n_1<\dots<\ord_{P_n} s^n_r.\]
For $j=0,\dots,r$, set $c_j=(\ord_{P_2} s^2_j,\dots,\ord_{P_n} s^n_j), w_j=w(c_j)$.

Then for $j=0,\dots,r$, there exists $s_j \in \Gamma(X_0,\sL_{w_j})$ such that for each $i$, we have that $s_j|_{Z_i}$ agrees with $s^i_j$ up to scalar.
Moreover, for each $w_j$, the subspace of $\Gamma(X_0,\sL_{w_j})$  consisting of sections restricting to $V^i$ on $Z_i$ for all $i$ 
has dimension precisely $r+1$.
\end{prop}

We now  consider families of curves  degenerating to $X_0$:

\begin{notn}\label{sit:degen}
We will denote with  $\pi:X \to B$  a flat, proper morphism, with $1$-dimensional fibers, $B$ the spectrum of a discrete valuation ring.
We assume that $X$ is regular, the generic fiber $X_{\eta}$ is smooth, and the special fiber $X_0$ is as before.
 For each $i=1,\dots,n$, let $\hat{\sigma}_i \in \Gamma(X,\sO_X(Z_i))$ be a section vanishing precisely on $Z_i$. 
 Choose an isomorphism
$$\hat{\theta}:\bigotimes_{i =1,\dots, n} \sO_X(Z_i) \risom \sO_X.$$
\end{notn}

Then  $\sO_X(Z_i)$ and $\hat{\sigma}_i$ induce systems of line bundles and sections as we had previously constructed for $\sigma_i$ and we have

\begin{prop}\label{prop:twistors-family} For $i=1,\dots, n$, $\sO_X(Z_i)|_{X_0} \cong \sO^i$, and $\hat{\sigma}_i|_{X_0}$ is a valid  choice of $\sigma_i$. 
Similarly, $\hat{\theta}|_{X_0}$ is a valid choice of $\theta$.
\end{prop}

Given a flat base change $B' \to B$ with $B'$ still the spectrum of a discrete valuation ring,
it induces $\pi':X' \to B'$ with  special fiber $X'_0$ which is a base extension of $X_0$, and generic fiber $X'_{\eta}$, a base change of $X_{\eta}$. 
 Suppose we have a linear series $(\sL_{\eta},V_{\eta})$ of rank $r$ and degree $d$ on $X'_{\eta}$. 
 By the compact type hypothesis, we know that for every multidegree $\omega$ of total degree $d$, 
 there is a unique extension $\overline{\sL}$ of $\sL_{\eta}$ over all $X'$ such that the
restriction to $X'_0$ has multidegree $\omega$; denote this by $\overline{\sL}_{\omega}$.
 We can construct a system of choices of the $\overline{\sL}_{\omega}$ together with morphisms between them,
  just as we did above, with (the pullbacks to $X'$ of) $\sO_X(Z_i)$ and $\hat{\sigma}_i$ in place of $\sO^i$ and $\sigma_i$, 
  and $\hat{\theta}$ in place of $\theta$.
Then given an extension $\overline{\sL}_{\omega}$, we also obtain an extension $\overline{V}_{\omega}$ simply by taking
$$\overline{V}_{\omega}=V_{\eta} \cap \Gamma(X',\overline{\sL}_{\omega}) \subseteq \Gamma(X'_{\eta},\sL_{\eta}).$$
From the definition of this extension, we see that both it and the corresponding quotient are torsion-free, hence free.
A key observation for us (initially developed in \cite{os8}) is that for any multidegrees $\omega,\omega'$, we have that $\overline{V}_{\omega'}$ 
maps into $\overline{V}_{\omega}$ under the above-constructed morphism $\overline{\sL}_{\omega'} \to \overline{\sL}_{\omega}$.

We now want to consider several linear series as well as their products.
Although we are ultimately interested in the case of powers of a single line bundle on a smooth curve, 
when doing degeneration we will want to consider  
\textbf{distinct} extensions to the reducible special fiber.

Consider a base change family $\pi':X' \to B'$ and $(\sL_1,V_1),\dots,(\sL_m,V_m)$ linear series (possibly of different ranks and degrees) on $X'_{\eta}$.
Our objective is to study the multiplication map 
$$\mu:V_1 \otimes \cdots \otimes V_m \to \Gamma(X'_{\eta}, \sL_1 \otimes \cdots \otimes \sL_m)$$
by considering how it limits to $X'_0$. 
For each $\sL_k$, we fix systems of extensions $\overline{\sL}_{k,\omega_k}$ as above for each multidegree $\omega_k$ of total degree equal to $\deg \sL_k$.
If we set $\sL:=\sL_1 \otimes \cdots \otimes \sL_m$, we also fix a system of extensions $\overline{\sL}_{\omega}$ of $\sL$ 
for each $\omega$ of total degree equal to $\sum_k \deg \sL_k$. 
As discussed above, we can extend each $V_k$ in multidegree $\omega_k$ by setting
$$\overline{V}_{k,\omega_k}:=V_k \cap \Gamma(X',\overline{\sL}_{k,\omega_k}).$$ 
Similarly, if we write $W_{\eta}$ for the image of $\mu$, then $(\sL,W_{\eta})$ is itself a linear series, so we can extend it to
$$\overline{W}_{\omega}:=W_{\eta} \cap \Gamma(X',\overline{\sL}_{\omega}).$$
If we choose any $\omega_k$'s, and set $\omega=\sum_k \omega_k$, we can also extend our multiplication map to obtain
$$\overline{\mu}:\overline{V}_{1,\omega_1} \otimes \cdots \otimes  \overline{V}_{m,\omega_m} \to \Gamma(X', \overline{\sL}_{\omega}).$$
We see immediately from the construction that the image of $\overline{\mu}$ is contained in $\overline{W}_{\omega}$.
Because reduction to the special fiber is surjective, we likewise have that the image of the restriction of $\overline{\mu}$ to
$X'_0$ is contained in the restriction of $\overline{W}_{\omega}$.  
Finally, given multidegrees $\omega,\omega'$, as we observed above we have that $\overline{W}_{\omega'}$ maps into $\overline{W}_{\omega}$
under our constructed maps.

To summarize, if we restrict to the special fiber, we have a system of spaces $\overline{W}_{\omega}|_{X'_0}$, each of dimension equal to
$\dim W_{\eta}$, containing the images of the appropriate multiplication maps $\overline{\mu}|_{X'_0}$ and linked together by natural maps.
 So if we have an $m$- tuple of sections $s_1,\dots,s_m$ in $\overline{V}_{1,\omega_1}|_{X'_0},\dots, \overline{V}_{m,\omega_m}|_{X'_0}$, 
 and set $ \omega'=\sum_k \omega_k$, then $s_1 \otimes \dots \otimes s_m$ is in $\overline{W}_{\omega'}|_{X'_0}$. If we fix a multidegree $\omega$,
the image of $s_1 \otimes \dots \otimes s_m$ under the constructed map from multidegree $\omega'$ to multidegree $\omega$ lies in
$\overline{W}_{\omega}|_{X'_0}$. Our strategy is then to construct  many such sections in different multidegrees, and consider all of their 
images inside a single multidegree $\overline{W}_{\omega}|_{X'_0}$.
If we can show that the images span a space of dimension $N$, then this implies that $\overline{W}_{\omega}|_{X'_0}$ has dimension at least $N$, 
and hence that $W_{\eta}$ had dimension at least $N$ as well.

From now on, we restrict to the case of interest in the Maximal Rank Conjecture, where $\sL_1=\sL_2 =\dots = \sL_m$.

From the above discussion, we will be able to conclude the following criterion.
\begin{prop}\label{cor:lls-reduction} Given integers $(g,r,d,m), m\ge 2, r\ge 3, g-(r+1)(g-d+r)\ge 0$, 
Let $X_0$ be as before (see \ref{defX0}).
 Suppose we have a chain-adaptable limit linear series $(\sL^i,V^i)$ on $X_0$ of rank $r$ and degree $d$
such that, if $s_j \in \Gamma(X_0,\sL_{w_j})$ are the global sections arising from the chain-adaptability condition, 
there exists  $c \in \ZZ^{n-1}$ such that for all choices of the sections $\sigma_i$ as above, the images of
the $s_{j_1} \otimes \dots \otimes s_{j_m}$ in $\Gamma(X_0,(\sL^{\otimes m})_{w(c)})$ have at least $N$-dimensional span.
Then, for any smoothing $\pi:X \to B$ of $X_0$ as in \ref{sit:degen}, the generic fiber of the smoothing family is a smooth genus-$g$ curve $X$ which carries a linear series
$(\sL,V)$ of rank $r$ and degree $d$ on $X$ such that the $m$-multiplication map \eqref{eq:multiply} for $V$  has rank at least $N$.

If further $n=g$ and we have $(w_{j_1}+\dots+w_{j_m},w)$ steady for all  $(j_1,\dots,j_m)$, then $X_0$ is not in the closure of the locus in $\cM_g$
corresponding to curves which do not carry an $(\sL,V)$ having $m$-multiplication map of rank at least $N$.

In particular, if $N=\min\left(\binom{r+m}{m},md+1-g\right)$, the Maximal Rank Conjecture holds for $(g,r,d,m)$, and under the additional steadiness hypothesis,
 the locus in $\overline{\cM}_g$ consisting of  chains of genus-$1$ curves is not in the closure of the locus of $\cM_g$ for which the maximal rank condition fails.
\end{prop}

\begin{proof} First, the condition that $P_i-Q_i$ is not $\ell$-torsion for $\ell \leq d$ implies that the space of limit linear series on $X_0$ has expected dimension $\rho$. 
This implies that if $\pi:X \to B$ is any regular smoothing family of $X_0$, every limit linear series on $X_0$ is a limit of linear series on the smooth fibers of $\pi$:
that is, there exists a flat base change $B' \to B$ and an $(\sL,V)$ on the generic fiber of $X':= X \times_B B'$ such that $(\sL,V)$ extends as 
described above to the chosen limit linear series.
Indeed, since refinedness is part of the definition of chain adaptability, this follows from the original Eisenbud-Harris smoothing theorem (Corollary 3.5 of \cite{e-h1}).

Let $W$ denote the image of $V^{\otimes m}$ under multiplication. We want to prove that $W$ has dimension at least $N$.
We first observe that each section $s_j$ must be in the  multidegree-$w_j$ limit of $(\sL,V)$:
 indeed, the limit of $V$ has  dimension $r+1$ and maps into the $V^i$ under each restriction map, so according to the second part of Proposition \ref{prop:chain-adaptable}, 
 the limit of $V$ is the entire subspace of global sections of  $\sL_{w_j}$ which restricts into $V^i$ on $Z_i$ for all $i$, and in particular it contains $s_j$. 
We likewise have that each $s_{j_1} \otimes \dots \otimes s_{j_m}$ is in the multidegree- $(w_{j_1}+\dots+w_{j_m})$ limit of $(\sL^{\otimes m},W)$.
Then it follows from the above discussion that the image of $s_{j_1} \otimes \dots \otimes s_{j_m}$  lies in the multidegree-$w_{md}(c)$ (Notation \ref{notn:md}) limit of $(\sL^{\otimes m},W)$. 
If, as the $(j_1,\dots,j_m)$ vary, these images span a space of dimension $N$, then it follows that $W$ has dimension at least $N$, as desired.
This proves the first assertion of the corollary.

In order to prove the stronger statement under the additional steadiness hypothesis, we carry out a similar analysis when
the smoothing family $\pi:X \to B$ is not assumed regular. 
Because we have also assumed $n=g$, such families can be used to study arbitrary curves in $\overline{\cM}_g$ specializing to $X_0$.
 In this situation we can blow up $X$ to obtain a regular family  $\widetilde{\pi}:\widetilde{X} \to \widetilde{B}$, where the special fiber $\widetilde{X}_0$
  is obtained from $X_0$ by inserting (possibly empty) chains of projective lines at the nodes of $X_0$. 
  It suffices to show that in this case the hypotheses of the corollary apply equally to $\widetilde{X}_0$, 
  since we can then apply the first part of the corollary to conclude the desired statement for the family $\widetilde{\pi}$, whose smooth fibers agree with those of $\pi$.
   It is clear what limit linear series we should choose: if we insert a projective line with marked points $P,Q$ at a node 
   which has vanishing on one side (at $Q_i$) given by $b_0,\dots,b_r$ and on the other  (at $P_{i+1}$) by $a_0,\dots,a_r$, 
   so that $a_j+b_{j}=d$ for all $j$, then we  have on the rational curve sections of $\sO(d)$, unique up to scaling, with vanishing order $a_j$ at $P$ and $b_{j}$ at $Q$.
 If we take the span of these $r+1$ sections, we obtain a $\fg^r_d$ on the projective line, and if we repeat this procedure for every inserted projective line, 
 and keep the old linear series on the elliptic components, we will obtain a new chain-adaptable limit linear series on $\widetilde{X}_0$. 
 If $\widetilde{X}_0$ has $n'$ components, this limit linear series has corresponding global sections $\widetilde{s}_0,\dots,\widetilde{s}_r$ in  multidegrees determined by
$\widetilde{w}_0,\dots,\widetilde{w}_r $, where $\widetilde{w}_j$ is  obtained from the $w_j$ by assigning multidegree $0$ to every inserted component.

By construction, the sections  $\widetilde{s}_j$ agree with $s_j$ (at least, up to scalar) after restriction to any given component of $X_0$, 
so the same applies to their tensor products $\widetilde{s}_{\vec{j}}$ for any  $\vec{j}=(j_1,\dots,j_m)$.
 Now, in general the insertion of the new components can change which components are zeroed out in mapping from multidegree $\widetilde{w}_{\vec{j}}$ to 
multidegree $\widetilde{w}$, even on the components of $\widetilde{X}_0$ coming from $X_0$.\footnote{Here, $\widetilde{w}_{\vec{j}}=\sum_{k=1}^m \widetilde{w}_{j_k}$ and $\widetilde{w}$ is the multidegree obtained from $w$ in the statement of the proposition by assigning multidegree 0 to every inserted component.} 
Indeed, the sums  $\sum_{i'=i+1}^{g} (c'_{i'}-c_{i'})$ appearing in the definition of $\epsilon^i_{(\widetilde{w}_{\vec{j}},\widetilde{w})}$ will have some extra
repetitions inserted corresponding to the new components. 
If one has $i<i'$ such that $c'_i<c_i$ and $c'_{i'}>c_{i'}$, inserting repetitions can change where the minimum is achieved. 
However, this is precisely ruled out by the steadiness hypothesis, so we see that with this  hypothesis, we will have the map from multidegree $\widetilde{w}_{\vec{j}}$ to multidegree $\widetilde{w}$ nonzero precisely on the components $Z_i$ on which the original map was nonzero,
 together with any inserted components connecting two components on which the map is nonzero.
We conclude that on each component of $\widetilde{X}_0$ coming from $X_0$, the image of $\widetilde{s}_{\vec{j}}$ in multidegree $\widetilde{w}$ agrees up to scalar with the image of $s_{\vec{j}}$ in multidegree $\widetilde{w}$. 
Now, observe that since $\widetilde{w}$ induces multidegree $0$ on each inserted projective line, we have a canonical `contraction' isomorphism 
$$\Gamma(\widetilde{X}_0,(\sL^{\otimes m})_{\widetilde{w}}) \risom\Gamma(X_0,(\sL^{\otimes m})_{w})$$
and we see that under this isomorphism, the images of the $\widetilde{s}_{\vec{j}}$ in multidegree $\widetilde{w}$ agree up
to scalar with the images of the $s_{\vec{j}}$ in multidegree $w$. 
Indeed, this follows from the steadiness hypothesis, which ensures that not only do the sections in question agree up to scalar after restriction
to each component of $X_0$, but their support is a contiguous collection of components $Z_i \cup \dots \cup Z_{i'}$ for some $i' \geq i$, 
and the sections do not vanish at any of the nodes $Q_i,\dots,Q_{i'-1}$.
We conclude that if the images of the $\widetilde{s}_{\vec{j}}$ in 
multidegree $\widetilde{w}$ span a space of dimension at least
$N$, the same is true of the images of the $s_{\vec{j}}$ in multidegree $w$. Thus, our hypotheses on the limit linear series on $X_0$ 
imply that the same hypotheses are satisfied on $\widetilde{X}_0$,
as desired. The corollary follows.
\end{proof}

\section{Independence of sections and  examples}\label{sec:elem-1}

We take $X_0$ as in \ref{defX0}. We will assume $n=g$.
One can  construct a linear series on $X_0$ with optimal  vanishing at the points $P_i, Q_i$  by on  each component $Z_i$ pick a value of $j$, 
say $j_0(i)=\delta(i)$ define 
\[  a^0_j=j; \  \ b^i_{\delta(i)}=d-a^i_{\delta(i)};\  \ b^i_j=d-1-a^i_j, j\not=\delta(i);\  \ a^{i+1}_j=d-b^i_j \]
It corresponds to taking the line bundle ${\mathcal O}_{Z_i}( a^i_{\delta(i)}P_i+ b^i_{\delta(i)}Q_i)$ and sections with largest possible vanishing for this line bundle.
For this construction to be possible, one needs $a^i_{\delta(i)}>a^i_{\delta(i)-1}+1$. This in turn requires having picked all values $j<{\delta(i)}$ at least as many times  as $\delta(i)$ prior to picking $\delta(i)$.
One also needs the Brill-Noether number to be positive, that is $g \geq (r+1)(g-d+r)$.

\begin{defn} Given $g,r,d >0$ with $g \geq (r+1)(g-d+r)$, a  \textbf{$(g,r,d)$-sequence} $\delta_1,\dots,\delta_g$ is a sequence of $g$ integers between $0$ and $r$, 
with each integer between $0$ and $r$ occurring at least $g-d+r$ times, and satisfying the condition that for each $i=1,\dots,g$, 
no integer strictly less than $\delta_i$ occurs among $\delta_1,\dots,\delta_i$ strictly fewer times than $\delta_i$ does.

More generally, given also $a \geq 0$, an \textbf{$a$-shifted $(g,r,d)$-sequence} $\delta_1,\dots,\delta_g$ is a $(g,r,d)$-sequence in which every integer between $0$ and $r$ 
occurs at least $a+g-d+r$ times.
For an $a$-shifted  sequence, we construct the limit linear series starting with $a^0_j=j+a$.
\end{defn}
One can keep track of the choice of the index $\delta(i)$  by organizing them in a Young -tableau with $r+1$ columns numbered $0,\dots, r$ and an indeterminate number of rows.
The numbers from 1 to $g$ are placed successively on the tableau starting on the left top corner. The element $i$ is placed on the highest empty spot of the column $\delta(i)$.
By construction, the numbers on each column increase as you go down.
The condition for being a $\delta$ sequence is  that the numbers also increase as you move right and that the filled space contains an $(r+1)(g-d+r)$ rectangle.  
(see for instance \cite{lstrop}).
The condition for being an $a$-shifted $\delta$ sequence is  that the numbers increase from left to right and from top to bottom and that the filled space contains an $(r+1)(a+g-d+r)$ rectangle.

We can construct linear series and their sections with this method and consider the multisections obtained as their products.
Our next goal is to show that a set of multisections constructed in this way  is linearly independent if some conditions on their orders of vanishing at the nodes are satisfied.
\begin{lem}\label{rulesii-v}
With the above notations, assume we have a linear dependence of multisections on $X_0$,  $\sum_{\vec{j}} \gamma_{\vec{j}} s_{\vec{j},w}=0$.
\begin{enumerate}[(a)]
\item  If for some $i$, there is a single section $s_{\vec{j}}$ among those that appear in the linear combination and are not identically zero on $Z_i$  such that 
$a^i_{\vec{j}}$ is strictly minimal among the orders of vanishing at $P_i$, then $\gamma_{\vec{j}} =0$.
\item  If for some $i$, there is a single section $s_{\vec{j}}$ among those that appear in the linear combination  and are not identically zero on $Z_i$ such that 
$b^i_{\vec{j}}$ is strictly minimal among the orders of vanishing at $Q_i$, then $\gamma_{\vec{j}} =0$.
\item  If for some $i$, there are only one or  two sections that appear in the linear combination and are not identically zero on $Z_i$, then their coefficients are zero.
\item If for some $i$, there is some $k \geq 0$ such that for every section $s_{\vec{j}}$ that appears in the linear combination and is not identically zero on $Z_i$, at least $k$ of the $j_{\ell}$ are equal to $j\not= \delta_i$, and there is a unique $\vec{j}=(j_1,\dots,j_m)$ for which exactly $k$ of the $j_{\ell}$ are equal to $j$,
 then the coefficient of that one section $s_{\vec{j}}$ is zero.
\end{enumerate}
\end{lem}
\begin{proof}
The proof of (a) and (b) follows by evaluating the sections at $P_i, Q_i$ respectively
When there is a single section in (c), it automatically follows that the coefficient is 0. 
The case of two sections is a particular case of (d). 
Let us then prove (d).

As  $j \neq \delta_i$, $\dv s^i_j= a^i_j P_i + b^i_j Q_i + R^i_j$ for a uniquely determined $R^i_j$.
Under our non-torsion hypothesis, all the $R^i_{j_{\ell}}$ are distinct as the $j_{\ell}$ varies.
The hypothesis in (d) imply that $\ord_{R^i_j} s_{\vec{j},w}|_{Z_i}=k$, while  $\ord_{R^i_j} s_{\vec{j}',w}|_{Z_i}>k$ for all remaining $\vec{j}'\neq\vec{j}$.
We thus conclude $\gamma_{\vec{j}}=0$ (as in (a), (b)).
\end{proof}

\begin{lem}\label{rulesvi}
With the above notations, assume we have a linear dependence of multisections on $X_0$,  $\sum_{\vec{j}} \gamma_{\vec{j}} s_{\vec{j},w}=0$.
 If for some $Z_i$ we have sections corresponding to indices  $\vec{j}^1=(j_1^1,\dots,j_m^1)$, $\vec{j}^2=(j_1^2,\dots,j_m^2)$ satisfying the conditions below, then 
 $ \gamma_{\vec{j}^1}=0= \gamma_{\vec{j}^2}$

\begin{itemize}
\item $a^i_{\vec{j}^1}=a^i_{\vec{j}^2}<a^i_{\vec{j}'}$ for every section  $\vec{j}' \neq \vec{j}^1,\vec{j}^2$ that appears in the linear combination  and is not identically zero on $Z_i$.
\item $b^{i+1}_{\vec{j}^1}=b^{i+1}_{\vec{j}^2}<b^{i+1}_{\vec{j}'}$ for every section  $\vec{j}' \neq \vec{j}^1,\vec{j}^2$ that appears in the linear combination  and is not identically zero on $Z_{i+1}$.
\item For at least one of $i'=i$ or $i'=i+1$, we have all but exactly two of the $j_1^1,\dots,j_m^1$ are equal to $\delta_{i'}$, all but exactly two of the $j_1^2,\dots,j_m^2$ are equal to $\delta_{i'}$, 
and $a^{i'}_{\vec{j}^1} \neq a^{i'}_{(\delta_{i'},\dots,\delta_{i'})}-1$.
\end{itemize}
\end{lem}
\begin{proof} The conditions on the $a^i_{\vec{j}^e}$ and $b^{i+1}_{\vec{j}^e}$ imply that if a linear dependence has nonzero coefficients $\gamma_{\vec{j}^1}$ and $\gamma_{\vec{j}^2}$
 for $s_{\vec{j}^1}$ and $s_{\vec{j}^2}$, then the leading terms of $\gamma_{\vec{j}^1} s_{\vec{j}^1}$ and $\gamma_{\vec{j}^2} s_{\vec{j}^2}$ must cancel at both $P_i$ and $Q_{i+1}$.
  Note also that our hypotheses on the $\vec{j}^e$ imply that $b^{i}_{\vec{j}^1}= b^{i}_{\vec{j}^2}$ (they must either both be equal to $md-a^i_{\vec{j}^e}-2$ or to $md-a^i_{\vec{j}^e}-m$), 
  and thus that $a^{i+1}_{\vec{j}^1}= a^{i+1}_{\vec{j}^2}$ as well.
   It thus makes sense to normalize our scaling of $s_{\vec{j}^1}$ and $s_{\vec{j}^2}$ so that their values agree at $Q_{i}$ (equivalently, at $P_{i+1}$). 
   First suppose that all but exactly two of the $j_n^e$ are equal to $\delta_i$ for both $e=1$ and $e=2$, and $a^i_{\vec{j}^1} \neq a^i_{(\delta_i,\dots,\delta_i)}-1$.
In this case, with the stated normalization, and a given choice of $P_{i+1},Q_{i+1}$, the desired cancellation at $Q_{i+1}$ will determine a unique ratio 
for $\gamma_{\vec{j}^1}$ and $\gamma_{\vec{j}^2}$. 
It suffices then to show that if we vary $P_i,Q_i$, the ratio determined by cancellation at $P_i$ varies nontrivially. 
But note that the $m-2$ copies of $s^i_{\delta_i}$ in $s^i_{\vec{j}^1}$ and $s^i_{\vec{j}^2}$ do not affect this variation, so this follows from Corollary \ref{cor:nonconst}.
The other case follows similarly, except that we fix $P_i,Q_i$ and consider the effects of letting $P_{i+1},Q_{i+1}$ vary.
\end{proof}

\begin{lem}\label{rulesvii}
With the above notations, assume  $m=2$ and we have a linear dependence of multisections on $X_0$,  $\sum_{\vec{j}} \gamma_{\vec{j}} s_{\vec{j},w}=0$.
such that for some $i$ and $n \geq 2$ we have sections corresponding to indices $\vec{j}^e=(j_1^e,j_2^e)$ for $e=1,\dots,n$ satisfying the conditions below,
then the coefficients of $\vec{j}^1,\dots,\vec{j}^{n},(\delta_i,\delta_i)$ are zero.
\begin{itemize}
\item $\delta_{i'}=\delta_i$ for $i'=i,\dots,i+n-1$;
\item for $e=1,\dots,n$, we have 
$j_1^e < \delta_{i} < j_2^e$;
\item the value of $a^i_{\vec{j}^e}$ is independent of $e \in \{1,\dots,n\}$;
\item for $i'=i,\dots,i+n-1$, the sections with indices  $\vec{j}^1,\dots,\vec{j}^n$ are non-trivial on  $Z_{i'}$, 
and the only other remaining section appearing in the linear combination  which may be non-trivial  in these $Z_{i'}$ is the $(\delta_i,\delta_i)$ section.
\end{itemize}
\end{lem}
\begin{proof}
 if we let $Z_I=\cup_{i' =i}^{i+n-1} Z_{i'}$, we wish to show that with the given hypotheses, we cannot have a nontrivial linear relation
$$\gamma_{\vec{j}^1} s_{\vec{j}^1}|_{Z_I} +\dots + \gamma_{\vec{j}^n} s_{\vec{j}^n}|_{Z_I} + \gamma_{(\delta_i,\delta_i)} s_{(\delta_i,\delta_i)}|_{Z_I} = 0.$$ 
Observe that the conditions $\delta_{i'}=\delta_i$ and $j_1^e < \delta_{i} < j_2^e$ (and therefore  $j_1^e \not= \delta_{i} \not= j_2^e$) imply that
$a^{i'}_{\vec{j}^e}=a^{i}_{\vec{j}^e}+2(i'-i)$ for any $i'=i,\dots,i+k-1$.
In particular $a^{i'}_{\vec{j}^e}$ is also independent of $e$ for $i'>i$. 
The conditions   $c_{i'}\le a^{i'}_{\vec{j}^e},d'-c_{i'+1}\le d'-a^{i'}_{\vec{j}^e}-2$ imply that  $c_{i'}=a^{i'}_{\vec{j}^e}$.

The linear dependence, if non-trivial, must cancel all leading terms at all $P_{i'}$ and $Q_{i'}$; cancellation at  $Q_{i'}$ is equivalent to cancellation at $P_{i'+1}$.
This works out to at most $k+1$ conditions. 
The rough idea of our argument is that when the chosen marked points are general, we obtain either $k+1$ or $k$ conditions in this way.
The latter occurs in a situation where  $s_{(\delta_i,\delta_i)}$ never contributes to the leading terms.
More specifically, we proceed from $i'=i$ to $i'=i+k-1$, showing that if we fix the previous choices of $P_{i'},Q_{i'}$, a general choice 
of the current $Q_{i'}$ will impose an additional linear condition on the choice of the $\gamma_{\vec{j}}$, with at most one exception.  

We need some preliminary observations on when the $(\delta_i,\delta_i)$ section can contribute on a given component $Z_{i'}$.
 For every $i'=i,\dots,i+k-1$, the vanishing of the $(\delta_i,\delta_i)$ section on  $Z_{i'}$ add up to $2d$, 
 whereas the vanishings  of the $\vec{j}^e$ section add up to $2d-2$.
Let $i_0$ be the smallest number between $i$ and $i+k-1$ such that $a^{i_0}_{(\delta_i,\delta_i)} \geq c_{i_0}$ and
$b^{i_0}_{(\delta_i,\delta_i)}\geq 2d-c_{{i_0}+1}$, so that $s_{(\delta_i,\delta_i)}$ may give rise to a non-zero section on  $Z_{i_0}$.
 First observe that if there is any $i'$ with $a^{i'}_{\vec{j}^e} = a^{i'}_{(\delta_i,\delta_i)}-1$, then we have also $b^{i'}_{\vec{j}^e} = b^{i'}_{(\delta_i,\delta_i)}-1$.
Furthermore $i_0=i'$ is the only column between $i$ and $i+k-1$ in which $s_{(\delta_i,\delta_i)}$ may occur.
Similarly, if for some $i'$ we have $b^{i'}_{\vec{j}^e} = b^{i'}_{(\delta_i,\delta_i)}$, then  $a^{i'}_{\vec{j}^e} = a^{i'}_{(\delta_i,\delta_i)}-2$.
So we must have $i_0=i'$.
 In this case, if $i'<i+k-1$ we can also have $s_{(\delta_i,\delta_i)}$ occurring in the next column, but not in any others, since 
  $b^{i'+1}_{\vec{j}^e} = b^{i'+1}_{(\delta_i,\delta_i)}-2$.

We now begin our analysis with the case $i'=i$: let $W_i$ be the subspace of the $k$-dimensional vector space of  $\gamma_{\vec{j}^e}$
 such that there exists a  $\gamma_{(\delta_i,\delta_i)}$ giving a valid linear dependence on $Z_{i}$. 
 If $i_0> i$, then cancellation of lowest-order terms at $P_i$ is a codimension-$1$ subspace $H$ of the space of $\gamma_{\vec{j}^e}$ containing $W_i$ 
(specifically, given our normalization, it is the hyperplane $\sum_l x_l =0$). 
Moreover, when $i_0>i$ we have observed above that  $a^{i}_{\vec{j}^e} \neq a^{i}_{(\delta_i,\delta_i)}-1$.
So under our normalization hypotheses, the sections $(s_{\vec{j}^1}^{i}|_{Z_{i}},\dots,s_{\vec{j}^n}^{i}|_{Z_{i}})$ satisfy the hypotheses of 
Corollary \ref{cor:nondegen}, and the map
$$Q_{i} \mapsto (s_{\vec{j}^1}^{i}|_{Z_{i}}(Q_{i}),\dots, s_{\vec{j}^k}^{i}|_{Z_{i}}(Q_{i}))$$
is nondegenerate.
 In particular, it is nonconstant, so a general choice of $Q_{i}$ will not have image equal to (the projectivization of) the orthogonal complement of $H$.
 Hence, cancellation of lowest order terms at $Q_i$ will impose a different codimension-$1$ condition. 
 In this case, we thus have that $W_i$ is at most $(k-2)$-dimensional. 

On the other hand, if $i_0=i$, we claim that $W_i$ has dimension at most $k-1$.
 Indeed, if $a^{i}_{\vec{j}^e} > a^{i}_{(\delta_i,\delta_i)}$, then $a^i_{(\delta_i,\delta_i)}$ is a unique minimum.
 So by \ref{rulesii-v}(a), we can drop the $(\delta_i,\delta_i)$ row, and we are in the same situation as above, with $\dim W_i=n-2$. 
On the other hand, if $a^{i}_{\vec{j}^e} < a^{i}_{(\delta_i,\delta_i)}$, then $W_i$ is still contained in the hyperplane $H$ described above. 
Finally, if  $a^{i}_{\vec{j}^e} = a^{i}_{(\delta_i,\delta_i)}$, then  $b^{i}_{\vec{j}^e} = b^i_{(\delta_i,\delta_i)}-2$,
 and in this case $W_i$ is contained in the hyperplane obtained by looking at cancellation of the leading coefficients at $Q_i$.

Now, for $i'>i$, let $W_{i'-1} $ be the subspace of choices of $\gamma_{\vec{j}^s}$ such that there exists a choice of
$\gamma_{(\delta_i,\delta_i)}$ giving a valid linear dependence on $Z_{i}\cup \dots\cup Z_{i'-1}$.
 If $W_{i'-1}=0$, we are done.
  Otherwise, our inductive hypothesis is that $W_{i'-1}$ has codimension at least $i'-i+1$ if $i_0 \geq i'$, and $W_{i'-1}$ has codimension at least $i'-i$ if $i_0<i'$. 
  We then want to show that imposing linear dependence also on $Z_{i'}$ reduces the dimension of $W_{i'-1}$ by $1$ unless $i'=i_0$.
First, if we have either  $a^{i'}_{\vec{j}^e} = a^{i'}_{(\delta_i,\delta_i)}-1$ or $b^{i'}_{\vec{j}^e} = b^{i'}_{(\delta_i,\delta_i)}$, then necessarily $i'=i_0$;
in this case, there is nothing to show.
 So we can assume that $a^{i'}_{\vec{j}^e} \neq a^{i'}_{(\delta_i,\delta_i)}-1$ and $b^{i'}_{\vec{j}^e} \neq b^{i'}_{(\delta_i,\delta_i)}$. 
 The latter means that in order to have linear dependence on $Z_{i'}$, we must have cancellation among the leading coefficients at $Q_{i'}$ of the $s_{\vec{j}^e}$. 
 The former implies that, just as in the case $i'=i$,  we have that the map
$$Q_{i'} \mapsto (s_{\vec{j}^1}^{i'}|_{Z_{i'}}(Q_{i'}),\dots, s_{\vec{j}^k}^{i'}|_{Z_{i'}}(Q_{i'}))$$
is nondegenerate.
 In particular, a general choice of $Q_{i'}$ will have image not lying in the orthogonal complement of $W_{i'-1}$, meaning that 
requiring that the $\gamma_{\vec{\delta}}$ impose a linear dependence also on $Z_{i'}$ reduces the dimension of $W_{i'-1}$ by $1$, as desired.

Because $Z_I$ has $k$ components, we thus conclude that when we have imposed cancellation of leading terms at all $P_{i'}$ and $Q_{i'}$, 
we have reduced the space of possible linear dependences to $(0)$,  proving the result. 
\end{proof}

In order to visualize and more easily work with the data of sections and their orders of vanishing, we will organize them  in tables
\begin{defn}\label{defn:tensor-table}
Given a $(g,r,d)$-sequence $\vec{\delta}=\delta_1,\dots,\delta_g$, and $m \geq 2$,  $T'(\vec{\delta})$ is  the $(r+1)\times g$ table 
whose $j^{th}, j=0,\dots, r$ row consists of the orders of vanishing at the nodes of the $j^{th}$ section of the linear series associated to $\delta$.

Then, $T(\vec{\delta})$  is the $\binom{r+m}{m}\times g$ table with rows indexed by $\vec{j}=(j_1,\dots,j_m)$ (with $0 \leq j_1 \leq j_2 \leq \dots \leq j_m \leq r$), 
and each entry being a pair of integers $(a^i_{\vec{j}},b^i_{\vec{j}})$, by setting the $(j_1,\dots,j_m)$th row of $T(\vec{\delta})$ to be the sum of the $j_n$th rows of $T'(\vec{\delta})$, for 
$n=1,\dots,m$.

More generally, given an $a$-shifted $(g,r,d)$-sequence, define the tables $T'(\vec{\delta})$ and $T(\vec{\delta})$ just as above, except that we start with $a^1_j=a+j$ for $j=0,\dots,r$.
\end{defn}

Thus, in the $a$-shifted case, all the $a^i_j$ are $a$ larger than in the usual case, and all the $a^i_{\vec{j}}$ are $ma$ larger. This is  convenient for certain reduction arguments.

We now construct. a table $T_w(\vec{\delta})$ that keeps track of forced vanishing of sections when we look at the multidegree associated to a choice of  a $(g-1)$-tuple of integers 
$c$ as in \ref{notn:md}. The rows of $T_w(\vec{\delta})$ will correspond to a collection of global sections which all lie in the multidegree determined by $w$. 
According to Proposition \ref{prop:zeroing-pattern} and with the  $\epsilon^i_{w',w}$ as defined there, this amounts to the following:

\begin{defn}\label{defn:table-erasures} In the situation of Definition \ref{defn:tensor-table}, suppose that we are also given a $c=(c_2,\dots,c_{g}) \in \ZZ^{g-1}$, as defined in \ref{notn:md}. 
Then define the table $T_{w(c)}(\vec{\delta})$ obtained from $T(\vec{\delta})$ by erasing certain entries as follows: for the row 
of $T(\vec{\delta})$ indexed by $\vec{j}=(j_1,\dots,j_m)$, let $c'=(a^2_{\vec{j}},\dots,a^{g}_{\vec{j}})$.
Then for $i=1,\dots,g$, the $i$th entry in the $\vec{j}$th row of $T(\vec{\delta})$ is erased in $T_w(\vec{\delta})$ if the $\epsilon^i_{w',w}$  is equal to $0$.

In addition, we say that $T(\vec{\delta})$ is \textbf{steady} with respect to $w$ if for each $\vec{j}$, setting $c'=(a^2_{\vec{j}},\dots,a^g_{\vec{j}})$ as above, 
we have that $(w(c'),w(c))$ is steady (see definition \ref{defsteady}).
\end{defn}
We introduce another piece of notation that will make our language less cumbersome in the future
\begin{defn}\label{defn:Nexpungeable}
We will say that a table is $N$\textbf{-expungeable }if one can choose $N$ rows corresponding to $N$ sections that can be proved to be linearly independent by repeated application of 
Lemmas \ref{rulesii-v}, \ref{rulesvi}, \ref{rulesvii}.
\end{defn}

We will use these notations with $N=\min\left(\binom{r+m}{m},md+1-g\right)$.
Then Proposition \ref{cor:lls-reduction} will imply the maximal rank conjecture for the given numerical values of $g, r, d, m$.

We now give several examples. 
The first two are very simple cases for $r=3$, $m=2$, but as we will see in the proof of Theorem \ref{thm:m2} below, 
these examples fully handle the case $r=3$ and $m=2$, and also constitute the base for the general case with $m=2$.

\begin{ex}\label{ex:m2-r3-g4}
Consider the $r=3, g=4, d=6$. This is the  canonical case and  the only possible $(g,r,d)$-sequence is 
$\vec{\delta}=0,1,2,3$, which gives $T'(\vec{\delta})$ as follows.
\begin{center}
\begin{tabular}{lr|lr|lr|lr}
$0$ & $6$ & $0$ & $5$ & $1$ & $4$ & $2$ & $3$ \\
$1$ & $4$ & $2$ & $4$ & $2$ & $3$ & $3$ & $2$ \\
$2$ & $3$ & $3$ & $2$ & $4$ & $2$ & $4$ & $1$ \\
$3$ & $2$ & $4$ & $1$ & $5$ & $0$ & $6$ & $0$ \\
\end{tabular}
\end{center}
Take now $m=2$. Choose $c=(2,6,8)$. The highlighted entries in the  table below are the non-erased entries in $T_w(\vec{\delta})$.
We have placed the $c_i$ and $m-c_i$ at the top and bottom of the table in order to make the erasure procedure clearer.
\begin{center}
\begin{tabular}{llr|lr|lr|lr}
& & $10$ & $2$ & $6$ & $6$ & $4$ & $8$ & \\
\hline
$(0,0)$ &\cellcolor[gray]{.8} $0$ &\cellcolor[gray]{.8} $12$ & $0$ & $10$ & $2$ & $8$ & $4$ & $6$ \\
$(0,1)$ &\cellcolor[gray]{.8} $1$ &\cellcolor[gray]{.8} $10$ &\cellcolor[gray]{.8} $2$ &\cellcolor[gray]{.8} $9$ & $3$ & $7$ & $5$ & $5$ \\
$(0,2)$ & $2$ & $9$ & \cellcolor[gray]{.8} $3$ & \cellcolor[gray]{.8} $7$ & $5$ & $6$ & $6$ & $4$ \\
$(1,1)$ & $2$ & $8$ & \cellcolor[gray]{.8} $4$ & \cellcolor[gray]{.8} $8$ & $4$ & $6$ & $6$ & $4$ \\
$(0,3)$ & $3$ & $8$ & \cellcolor[gray]{.8} $4$ & \cellcolor[gray]{.8} $6$ & \cellcolor[gray]{.8} $6$ & \cellcolor[gray]{.8} $4$ & \cellcolor[gray]{.8} $8$ & \cellcolor[gray]{.8} $3$ \\
$(1,2)$ & $3$ & $7$ & \cellcolor[gray]{.8} $5$ & \cellcolor[gray]{.8} $6$ & \cellcolor[gray]{.8} $6$ & \cellcolor[gray]{.8} $5$ & $7$ & $3$ \\
$(1,3)$ & $4$ & $6$ & $6$ & $5$ & $7$ & $3$ & \cellcolor[gray]{.8} $9$ & \cellcolor[gray]{.8} $2$ \\
$(2,2)$ & $4$ & $6$ & $6$ & $4$ & \cellcolor[gray]{.8}$8$ & \cellcolor[gray]{.8}$4$ & \cellcolor[gray]{.8} $8$ & \cellcolor[gray]{.8} $2$ \\
$(2,3)$ & $5$ & $5$ & $7$ & $3$ & $9$ & $2$ & \cellcolor[gray]{.8}$10$ & \cellcolor[gray]{.8}$1$ \\
$(3,3)$ & $6$ & $4$ & $8$ & $2$ & $10$ & $0$ & \cellcolor[gray]{.8}$12$ & \cellcolor[gray]{.8}$0$ \\
\hline
& & $10$ & $2$ & $6$ & $6$ & $4$ & $8$ & \\
\end{tabular}
\end{center}

Since $\binom{r+2}{2}=10>2d+1-g=9$, this is a surjective case.
To prove surjectivity in this case we may drop any one section (as we had 10 to start with). 
If for instance we drop the $(0,3)$ section, we see that there are no remaining repetitions among the $a^i_{\vec{j}}$ in any column, so we have
that $T_w(\vec{\delta})$ corresponds to  $9$ different sections simply by repeated application of Lemma  \ref{rulesii-v}(a), proving the desired surjectivity.
\end{ex}

\begin{ex}\label{ex:m2-r3-g5}
Next consider the case $r=3, g=5, d=7$, and choose the  $(g,r,d)$-sequence  $\vec{\delta}=0,1,2,3,0$. which gives $T'(\vec{\delta})$ as follows.
\begin{center}
\begin{tabular}{lr|lr|lr|lr|lr}
$0$ & $7$ & $0$ & $6$ & $1$ & $5$ & $2$ & $4$ & $3$ & $4$ \\
$1$ & $5$ & $2$ & $5$ & $2$ & $4$ & $3$ & $3$ & $4$ & $2$ \\
$2$ & $4$ & $3$ & $3$ & $4$ & $3$ & $4$ & $2$ & $5$ & $1$ \\
$3$ & $3$ & $4$ & $2$ & $5$ & $1$ & $6$ & $1$ & $6$ & $0$ \\
\end{tabular}
\end{center}
Choose $m=2, c=(2,6,8,10)$. We then get $T(\vec{\delta})$ as follows.
\begin{center}
\begin{tabular}{llr|lr|lr|lr|lr}
 & & $12$ & $2$ & $8$ & $6$ & $6$ & $8$ & $4$ & $10$ &  \\
\hline
$(0,0)$ & \cellcolor[gray]{.8} $0$ & \cellcolor[gray]{.8} $14$ & $0$ & $12$ & $2$ & $10$ & $4$ & $8$ & $6$ & $8$ \\
$(0,1)$ & \cellcolor[gray]{.8} $1$ & \cellcolor[gray]{.8} $12$ & \cellcolor[gray]{.8} $2$ & \cellcolor[gray]{.8} $11$ & $3$ & $9$ & $5$ & $7$ & $7$ & $6$ \\
$(0,2)$ & $2$ & $11$ & \cellcolor[gray]{.8} $3$ & \cellcolor[gray]{.8} $9$ & $5$ & $8$ & $6$ & $6$ & $8$ & $5$ \\
$(1,1)$ & $2$ & $10$ & \cellcolor[gray]{.8} $4$ & \cellcolor[gray]{.8} $10$ & $4$ & $8$ & $6$ & $6$ & $8$ & $4$ \\
$(0,3)$ & $3$ & $10$ & \cellcolor[gray]{.8} $4$ & \cellcolor[gray]{.8} $8$ & \cellcolor[gray]{.8} $6$ & \cellcolor[gray]{.8} $6$ & \cellcolor[gray]{.8} $8$ & \cellcolor[gray]{.8} $5$ & $9$ & $4$ \\
$(1,2)$ & $3$ & $9$ & \cellcolor[gray]{.8} $5$ & \cellcolor[gray]{.8} $8$ & \cellcolor[gray]{.8} $6$ & \cellcolor[gray]{.8} $7$ & $7$ & $5$ & $9$ & $3$ \\
$(1,3)$ & $4$ & $8$ & $6$ & $7$ & $7$ & $5$ & \cellcolor[gray]{.8} $9$ & \cellcolor[gray]{.8} $4$ & \cellcolor[gray]{.8} $10$ & \cellcolor[gray]{.8} $2$ \\
$(2,2)$ & $4$ & $8$ & $6$ & $6$ & \cellcolor[gray]{.8} $8$ & \cellcolor[gray]{.8} $6$ & \cellcolor[gray]{.8} $8$ & \cellcolor[gray]{.8} $4$ & \cellcolor[gray]{.8} $10$ & \cellcolor[gray]{.8} $2$ \\
$(2,3)$ & $5$ & $7$ & $7$ & $5$ & $9$ & $4$ & \cellcolor[gray]{.8} $10$ & \cellcolor[gray]{.8} $3$ & \cellcolor[gray]{.8} $11$ & \cellcolor[gray]{.8} $1$ \\
$(3,3)$ & $6$ & $6$ & $8$ & $4$ & $10$ & $2$ & \cellcolor[gray]{.8} $12$ & \cellcolor[gray]{.8} $2$ & \cellcolor[gray]{.8} $12$ & \cellcolor[gray]{.8} $0$ \\
\hline
 & & $12$ & $2$ & $8$ & $6$ & $6$ & $8$ & $4$ & $10$ &  \\
\end{tabular}
\end{center}
This case is both surjective and injective, so we need to use all sections (corresponding to all rows on the table).
We can drop the last four rows by applying Lemma  \ref{rulesii-v}(b) twice and \ref{rulesii-v}(d) once to the last column, and then apply \ref{rulesii-v}(c) to the third column to cancel the coefficients 
of the  $(0,3)$ and $(1,2)$ rows. 
After this, no repetitions remain among either the $a^i_{\vec{j}}$ or  $b^i_{\vec{j}}$ in any column, so we can drop the rest of the rows using Lemma  \ref{rulesii-v} either  (a) or (b).
\end{ex}

The following example is the first requiring the use of Lemmas \ref{rulesvi} and \ref{rulesvii},  and is the first of the sequence of `critical' cases for $m=2$, 
treated more generally in Proposition \ref{prop:m2-critical} below.

\begin{ex}\label{ex:m2-critical} 
Consider the case  $r=4, g=10, d=12$, and take the $(g,r,d)$-sequence  $\vec{\delta}=0,0,1,1,2,2,3,3,4,4$. This gives $T'(\vec{\delta})$ as follows.
\begin{center}
\begin{tabular}{lr|lr|lr|lr|lr|lr|lr|lr|lr|lr}
$0$ & $12$ & $0$ & $12$ & $0$ & $11$ & $1$ & $10$ & $2$ & $9$ & $3$ & $8$ & 
$4$ & $7$ & $5$ & $6$ & $6$ & $5$ & $7$ & $4$ \\
$1$ & $10$ & $2$ & $9$ & $3$ & $9$ & $3$ & $9$ & $3$ & $8$ & $4$ & $7$ & 
$5$ & $6$ & $6$ & $5$ & $7$ & $4$ & $8$ & $3$ \\
$2$ & $9$ & $3$ & $8$ & $4$ & $7$ & $5$ & $6$ & $6$ & $6$ & $6$ & $6$ & 
$6$ & $5$ & $7$ & $4$ & $8$ & $3$ & $9$ & $2$ \\
$3$ & $8$ & $4$ & $7$ & $5$ & $6$ & $6$ & $5$ & $7$ & $4$ & $8$ & $3$ & 
$9$ & $3$ & $9$ & $3$ & $9$ & $2$ & $10$ & $1$ \\
$4$ & $7$ & $5$ & $6$ & $6$ & $5$ & $7$ & $4$ & $8$ & $3$ & $9$ & $2$ & 
$10$ & $1$ & $11$ & $0$ & $12$ & $0$ & $12$ & $0$ \\
\end{tabular}
\end{center}
Take $m=2, c=(2,4,7,9,12,15,17,20,22)$, We then get $T(\vec{\delta})$ as follows.

\begin{center}
\resizebox{\textwidth}{!}{
\begin{tabular}{llr|lr|lr|lr|lr|lr|lr|lr|lr|lr}
 &  &  $22$ & $2$ & $20$ & $4$ & $17$ & $7$ & $15$ & $9$ & $12$ & $12$ & $9$ & $15$ & $7$ & $17$ & $4$ & $20$ & $2$ & $22$ & \\
\hline
$(0,0)$ & \cellcolor[gray]{.8} $0$ & \cellcolor[gray]{.8} $24$ & $0$ & $24$ & $0$ & $22$ & $2$ & $20$ & $4$ & $18$ & $6$ & $16$ & $8$ & $14$ & $10$ & $12$ & $12$ & $10$ & $14$ & $8$ \\
$(0,1)$ & \cellcolor[gray]{.8} $1$ & \cellcolor[gray]{.8} $22$ & \cellcolor[gray]{.8} $2$ & \cellcolor[gray]{.8} $21$ & $3$ & $20$ & $4$ & $19$ & $5$ & $17$ & $7$ & $15$ & $9$ & $13$ & $11$ & $11$ & $13$ & $9$ & $15$ & $7$ \\
$(0,2)$ & $2$ & $21$ & \cellcolor[gray]{.8} $3$ & \cellcolor[gray]{.8} $20$ & \cellcolor[gray]{.8} $4$ & \cellcolor[gray]{.8} $18$ & $6$ & $16$ & $8$ & $15$ & $9$ & $14$ & $10$ & $12$ & $12$ & $10$ & $14$ & $8$ & $16$ & $6$ \\
$(1,1)$ & $2$ & $20$ & $4$ & $18$ & \cellcolor[gray]{.8} $6$ & \cellcolor[gray]{.8} $18$ & $6$ & $18$ & $6$ & $16$ & $8$ & $14$ & $10$ & $12$ & $12$ & $10$ & $14$ & $8$ & $16$ & $6$ \\
$(0,3)$ & $3$ & $20$ & $4$ & $19$ & \cellcolor[gray]{.8} $5$ & \cellcolor[gray]{.8} $17$ & \cellcolor[gray]{.8} $7$ & \cellcolor[gray]{.8} $15$ & \cellcolor[gray]{.8} $9$ & \cellcolor[gray]{.8} $13$ & $11$ & $11$ & $13$ & $10$ & $14$ & $9$ & $15$ & $7$ & $17$ & $5$ \\
$(1,2)$ & $3$ & $19$ & $5$ & $17$ & $7$ & $16$ & \cellcolor[gray]{.8} $8$ & \cellcolor[gray]{.8} $15$ & \cellcolor[gray]{.8} $9$ & \cellcolor[gray]{.8} $14$ & $10$ & $13$ & $11$ & $11$ & $13$ & $9$ & $15$ & $7$ & $17$ & $5$ \\
$(0,4)$ & $4$ & $19$ & $5$ & $18$ & $6$ & $16$ & $8$ & $14$ & \cellcolor[gray]{.8} $10$ & \cellcolor[gray]{.8} $12$ & \cellcolor[gray]{.8} $12$ & \cellcolor[gray]{.8} $10$ & $14$ & $8$ & $16$ & $6$ & $18$ & $5$ & $19$ & $4$ \\
$(1,3)$ & $4$ & $18$ & $6$ & $16$ & $8$ & $15$ & $9$ & $14$ & \cellcolor[gray]{.8} $10$ & \cellcolor[gray]{.8} $12$ & \cellcolor[gray]{.8} $12$ & \cellcolor[gray]{.8} $10$ & $14$ & $9$ & $15$ & $8$ & $16$ & $6$ & $18$ & $4$ \\
$(2,2)$ & $4$ & $18$ & $6$ & $16$ & $8$ & $14$ & $10$ & $12$ & \cellcolor[gray]{.8} $12$ & \cellcolor[gray]{.8} $12$ & \cellcolor[gray]{.8} $12$ & \cellcolor[gray]{.8} $12$ & $12$ & $10$ & $14$ & $8$ & $16$ & $6$ & $18$ & $4$ \\
$(1,4)$ & $5$ & $17$ & $7$ & $15$ & $9$ & $14$ & $10$ & $13$ & $11$ & $11$ & \cellcolor[gray]{.8} $13$ & \cellcolor[gray]{.8} $9$ & \cellcolor[gray]{.8} $15$ & \cellcolor[gray]{.8} $7$ & \cellcolor[gray]{.8} $17$ & \cellcolor[gray]{.8} $5$ & $19$ & $4$ & $20$ & $3$ \\
$(2,3)$ & $5$ & $17$ & $7$ & $15$ & $9$ & $13$ & $11$ & $11$ & $13$ & $10$ & \cellcolor[gray]{.8} $14$ & \cellcolor[gray]{.8} $9$ & \cellcolor[gray]{.8} $15$ & \cellcolor[gray]{.8} $8$ & $16$ & $7$ & $17$ & $5$ & $19$ & $3$ \\
$(2,4)$ & $6$ & $16$ & $8$ & $14$ & $10$ & $12$ & $12$ & $10$ & $14$ & $9$ & $15$ & $8$ & $16$ & $6$ & \cellcolor[gray]{.8} $18$ & \cellcolor[gray]{.8} $4$ & \cellcolor[gray]{.8} $20$ & \cellcolor[gray]{.8} $3$ & $21$ & $2$ \\
$(3,3)$ & $6$ & $16$ & $8$ & $14$ & $10$ & $12$ & $12$ & $10$ & $14$ & $8$ & $16$ & $6$ & $18$ & $6$ & \cellcolor[gray]{.8} $18$ & \cellcolor[gray]{.8} $6$ & $18$ & $4$ & $20$ & $2$ \\
$(3,4)$ & $7$ & $15$ & $9$ & $13$ & $11$ & $11$ & $13$ & $9$ & $15$ & $7$ & $17$ & $5$ & $19$ & $4$ & $20$ & $3$ & \cellcolor[gray]{.8} $21$ & \cellcolor[gray]{.8} $2$ & \cellcolor[gray]{.8} $22$ & \cellcolor[gray]{.8} $1$ \\
$(4,4)$ & $8$ & $14$ & $10$ & $12$ & $12$ & $10$ & $14$ & $8$ & $16$ & $6$ & $18$ & $4$ & $20$ & $2$ & $22$ & $0$ & $24$ & $0$ & \cellcolor[gray]{.8} $24$ & \cellcolor[gray]{.8} $0$ \\
\hline
 &  &  $22$ & $2$ & $20$ & $4$ & $17$ & $7$ & $15$ & $9$ & $12$ & $12$ & $9$ & $15$ & $7$ & $17$ & $4$ & $20$ & $2$ & $22$ & \\
\end{tabular}
}
\end{center}
 
If we go from left to right we can use Lemma \ref{rulesii-v} (a) on $Z_1$ (first column ) to prove linear independence of the sections with indices $(0,0)$ and $(0,1)$ , 
$(0,2)$ on $Z_2$ (second column), $(0,3)$ and $(1,1)$ on $Z_3$ (third column),  and $(1,2)$ on $Z_4$ (fourth column).
Then, using Lemma \ref{rulesii-v} (b) we can drop rows $(4,4)$ and $(3,4)$ from the last column, row $(2,4)$ from the ninth column, rows $(1,4)$ and $(3,3)$ from the eighth column,
 and row $(2,3)$ from the seventh column. 
This leaves only rows $(0,4)$, $(1,3)$  and $(2,2)$ in the fifth and sixth columns, which can be dropped using either Lemma \ref{rulesvi} (together with Lemmas \ref{rulesii-v} (c)) or
 Lemma \ref{rulesvii}.
\end{ex}

\section{Observations on injectivity}\label{sec:injective}

We now consider injective cases, meaning that $\binom{r+m}{m} \leq md+1-g$.
Our main result will be the observation that $N$-expungeability for an injective case implies that we get infinitely many additional cases by increasing $g$. 
In fact, we will give two versions of this statement,  with one adding a mild hypothesis but yielding more cases in return. 
A preliminary definition is the following.

\begin{defn}\label{defn:extendable} We say that a $(g,r,d)$-sequence $\delta=(\delta_1,\dots,\delta_g)$ is \textbf{extendable} if for all
$g' \geq g$, and all $d'$ with $g' \geq (r+1)(g'-d+r')$, we can extend $\delta$ to a valid $(g',r,d')$-sequence.
\end{defn}

 We have the following characterization:

\begin{prop}\label{prop:extendable} A $(g,r,d)$-sequence $\vec{\delta}$ is extendable if and only if $0$ occurs at most one time more than $r$ does in $\vec{\delta}$.
\end{prop}

\begin{proof} 
In the language of Young Tableaux introduced at the start of section \ref{sec:elem-1}, the condition that 0 appears at most one more time than $r$ 
can be written as the column corresponding to 0 is at most one taller than the column corresponding to $r$.
Equivalently, the Young Tableau is as close to a rectangle as it can possibly be for the given $g$.
Then, for any $g'\ge g$, one can add the additional $g'-g$ indices while keeping the tableau again as close to a rectangle as possible for that $g'$,
so the sequence is extendable.

On the other hand, assume that  the column corresponding to 0 has height $l$ that is at least two larger than the height of the column corresponding to $r$.
Let us say there are $t$ elements in the last bottom row.
Define $g'=(l-1)(r+1)+t-1,\ d'=g'+r-l+1$. By assumption $g\le (l-1)(r+1) +t-2\le g'$. From the definition of $d'$, $g'-d'+r=l-1$. 
Any Young Tableau associated to $g',d',r$ must contain the $(r+1)(g'-d'+r)$ rectangle, hence it cannot contain the last element of the bottom row of the initial Tableau.
Hence, the $\delta$ sequence is not extendable.
\end{proof}

The following notion will be useful for verifying the steadiness condition.

\begin{defn}\label{defn:unimaginative}
For a given $m$, we say $c=(c_2,\dots,c_g) \in \ZZ^{g-1}$ is \textbf{unimaginative} if $c_{i+1}-c_{i} \geq m$ for $i=2,\dots,g-1$.
\end{defn}

\begin{prop}\label{prop:unimaginative} If $c=(c_2,\dots,c_g)$ is unimaginative, then for any $\vec{\delta}$ we have $T(\vec{\delta})$ steady with respect to $w(c)$.
\end{prop}

\begin{proof} 
Recall that by definition $T(\vec{\delta})$ is steady with respect to $w$ if for each multi-index of a section $\vec{j}$, setting $c'=(a^2_{\vec{j}},\dots,a^g_{\vec{j}})$  
we have that $(w(c'),w(c))$ is steady that is, there exists $i$ such that for  $ l< i, c_l\le a^l_{\vec{j}}$  for  $ l\ge i, a^l_{\vec{j}}\le c_l $.
By construction of the table associated to a given $\delta$, for $i<g$ we have
$$a^i_{\vec{j}} \geq md-m-b^i_{\vec{j}} = a^{i+1}_{\vec{j}}-m,$$
so the sequence $c'_i-c_i$ is nonincreasing, and $(w',w)$ is steady.
\end{proof}

We have the following basic observation on `change of degree':  

\begin{prop}\label{prop:invariance} Given $(g,r,d,m), \vec{\delta}$ a $(g,r,d)$-sequence, $w, d' > d$ and $a$ such that $0 \leq a \leq d'-d$, then
$\vec{\delta}$ is also an $a$-shifted $(g,r,d')$-sequence.
If we obtain $w'$ from $w$ by adding $ma$ to every entry, then $T_w(\vec{\delta})$ is $N$-expungeable for $\vec{\delta}$ as a  $(g,r,d)$-sequence if and only if 
$T_{w'}(\vec{\delta})$ is $N$-expungeable for $\vec{\delta}$ as an $a$-shifted $(g,r,d')$-sequence.
\end{prop}
\begin{proof} As $\vec{\delta}$ is a $(g,r,d)$-sequence, its Young Tableau contains an $(r+1)(g-d+r)$ rectangle.
The condition $a \leq d'-d$ ensures that it also contains an $(r+1)(g-d'+a+r)$ rectangle.
Hence, $\vec{\delta}$ is also an $a$-shifted $(g,r,d')$-sequence.

By definition of the $a$-shifted table, $T_{w'}(\vec{\delta})$ is obtained from $T_w(\vec{\delta})$ by adding $ma$ to each $a^i_{\vec{j}}$, and
adding $m(d'-d-a)$ to each $b^i_{\vec{j}}$. 
One checks directly that the rules for expungeability are invariant under this operation.
\end{proof}

Below is our basic result on extending injective cases to higher genus.

\begin{prop}\label{prop:injective-extend} Given $(g,r,d,m)$, satisfying
\[ r\ge 3, m\ge 2, g\ge (r+1)(g-d+r), \binom{r+m}{m} \leq md+1-g\]
suppose that there exists a $(g,r,d)$-sequence $\vec{\delta}$ and a $c=(c_2,\dots,c_g)$ such that $T_{w(c)}(\vec{\delta})$ is $\binom{r+m}{m}$-expungeable.
Then for all $(g',r,d',m)$ with $g' \geq g$ and $g'-d' \leq g-d$, there exists a $(g',r,d')$-sequence $\vec{\delta}'$ and a $w'=(c'_2,\dots,c'_g)$ such that 
$T_{w'}(\vec{\delta}')$ is $\binom{r+m}{m}$-expungeable. In particular, the Maximal Rank Conjecture holds in all these cases. 

If further the above holds with $\vec{\delta}$ extendable, then the condition $g'-d' \leq g-d$ above is unnecessary. 
In either situation, if the chosen $w$ was unimaginative, then the new $w$ may also be chosen to be unimaginative.
\end{prop}

\begin{proof} Under either hypothesis, we have that $\delta$ can be extended to a $(g',r,d')$-sequence $\delta'$: 
in the first case, the condition $g'-d' \leq g-d$ allows us to extend simply by adding $g'-g$  zeroes, while in the second case we can extend by hypothesis.
 Moreover, we have that $\delta$ is a valid $(g,r,d')$-sequence, and Proposition \ref{prop:invariance} says that the $\binom{r+m}{m}$-expungeability of
$T_w(\vec{\delta})$ does not depend on whether we view $\vec{\delta}$ as a $(g,r,d)$-sequence or a $(g,r,d')$-sequence
 (the only difference is that the  $b_j^{\vec{i}}$ are all translated by $m(d'-d)$). 
Then by appending sufficiently large (e.g., larger than $md'$) numbers to $c$, we obtain $c' \in \ZZ^{g'-1}$ with the property that $T_{w(c')}(\vec{\delta'})$
 is exactly the same as $T_{w(c)}(\vec{\delta})$, when $\vec{\delta}$ is considered as a $(g,r,d')$-sequence: the  entries of $T_{w(c')}(\vec{\delta'})$ after the first $g$ columns are all  erased.
  Then the $\binom{r+m}{m}$-expungeability of $T_{w'}(\vec{\delta'})$ follows.

If $w$ was unimaginative, the above construction can clearly also make $w'$ unimaginative.
\end{proof}

We conclude by proving that for any fixed $m,r$ we have injectivity for all $g$ sufficiently large. 
Although the bound is very far from sharp (and is worse than that obtained in Larson \cite{la4}),
the proof is brief and we include it as an illustration of a different sort of approach to applying Proposition  \ref{cor:lls-reduction} than the ones which we will make below.

\begin{prop}\label{prop:big-g-inject} With $m,r$ fixed, if we have $g,d$ with $\rho \geq 0$ and 
$$g \geq (r+1)\left((m+1)^{r-1}-r\right),$$
then the Maximal Rank Conjecture holds for $(g,r,d,m)$.
Moreover, a general chain of genus-$1$ curves is not in the closure of  the locus on $\cM_g$ for which the maximal rank condition fails.
\end{prop}

\begin{proof} We will show that with the stated lower bound, we can always produce a $(g,r,d)$-sequence $\vec{\delta}$ so that for some
column $i_0$, the entries $a^{i_0}_{\vec{j}}$ of $T(\vec{\delta})$ are all  distinct. 
Observe that this will be the case if the ${i_0}$th column of $T'(\vec{\delta})$ is equal to $0,1,m+1,(m+1)^2,\dots,(m+1)^{r-1}$,
or more generally, (for some $m$)
\[  a,a+1,a+m+1,a+(m+1)^2,\dots,a+(m+1)^{r-1}\] .
Thus, we take $\vec{\delta}$ to be the sequence whose first $(m+1)^{r-1}-r$ entries are $0$, and then followed by $(m+1)^{r-1}-(m+1)^{i-1}-(r-i)$  entries equal to $i$, for $i=1,\dots,r-1$.
We then take the next $(m+1)-2$ entries equal to $2$, and then followed by $(m+1)^{i-1}-i$ entries equal to $i$ for $i=3,\dots,r$. 
This determines the first  $(r+1)\left((m+1)^{r-1}-r\right)$ entries of $\vec{\delta}$, with each entry occurring $(m+1)^{r-1}-r$ times. 
Any remaining entries of  $\vec{\delta}$ can be chosen to cycle from $0$ through $r$.

Then set $i_0$ to be the column immediately after the first sequence of $(r-1)$s occurring in $\vec{\delta}$, so that $i_0=r(m+1)^{r-1}-\binom{r}{2}-\sum_{i=1}^{r-2} (m+1)^i$.
By construction, the entries $a^{i_0}_j$ of $T'(\vec{\delta})$ have the desired form, so the entries $a^{i_0}_{\vec{j}}$ of $T(\vec{\delta})$ are all distinct, as desired. 
Finally, let  $c=(c_2,\dots,c_g)$ with $c_i=0$ for $i \leq i_0$ and $c_i=md$ for  $i>i_0$. 
Note that this is steady with respect to $T(\vec{\delta})$ (indeed, for any $\vec{\delta}$), and the effect is that every row occurs in the $i_0$th column of $T_w(\vec{\delta})$. 
We may then apply Lemma \ref{rulesii-v} (b)  repeatedly to $T_w(\vec{\delta})$ to prove the desired statement.
\end{proof}

\section{The case of quadrics}\label{sec:m2}

In this section, we use Proposition  \ref{cor:lls-reduction} to prove the Maximal Rank Conjecture for the $m=2$ case. 
The proof uses reduction constructions to show that we can always reduce either to smaller $r$ or to one of a sequence of `critical' cases which are in particular as close 
as possible to being simultaneously injective and surjective. 
Empirically, these critical cases are the most difficult cases to handle.

\begin{thm}\label{thm:m2} The Maximal Rank Conjecture holds in the $m=2$ case. More specifically, for any given $(g,r,d)$ with $\rho \geq 0$ and $g-d+r>0$, 
a general chain of  genus-$1$ curves is not in the closure of the locus on $\cM_g$ where the maximal rank condition fails.
\end{thm}
Explicitly, for every such $(g,r,d)$  there is a $(g,r,d)$-sequence $\vec{\delta}$ and an unimaginative $c \in \ZZ^{g-1}$ such that $T_{w(c)}(\vec{\delta})$ is 
$\min\left(\binom{r+2}{2},2d+1-g\right)$-expungeable.

In order to keep the overall structure of the proof as clear as possible, we will first state the necessary preliminary results, then give the proof of the theorem, and
finally prove the preliminary results. 
In fact, we will also prove the statement of the theorem for many cases where $g-d+r \leq 0$, but to keep  the statement as simple as possible we do not list precisely which cases
are handled by our constructions.

The following lemma constitutes the basic reduction used for surjective cases.

\begin{lem}\label{lem:basic-reduction} Given $(g,r,d)$ with $\rho \geq 0$, set $t=\min(\rho+(g-d+r), r-1)$.
Define $r'=r-1$, $g'=g-t$ and $d'=d-(t+1)$.
Suppose that there is a  $(g',r',d')$-sequence $\vec{\delta}'$ having no more than $r'$ of any given integer, and an unimaginative $c'=(c'_2,\dots,c'_{g'}) \in \ZZ^{g'-1}$ such that 
$T_{w(c')}(\vec{\delta}')$ is $N$-expungeable, and $c'_2 \geq 2$. 
Then there is a $(g,r,d)$-sequence $\vec{\delta}$ having no more than $r$ of any given integer, and an unimaginative $c=(c_2,\dots,c_g) \in \ZZ^{g-1}$ such that 
$T_{w(c)}(\vec{\delta})$ is $(N+t+2)$-expungeable and $c_2 \geq 2$. 
In particular, if  $T_{w(c')}(\vec{\delta}')$ is $(2d'+1-g')$-expungeable, then $T_{w(c)}(\vec{\delta})$ is $(2d+1-g)$-expungeable. 

Moreover, if either $\binom{r+2}{2}>2d+1-g$ or $\binom{r+2}{2} =2d+1-g$ and $\rho>0$, then $\binom{r'+2}{2} \geq 2d'+1-g'$.
\end{lem}

Thus, the reduction of the lemma can be applied to give lower bounds on rank  in all cases, but the resulting bound may not be sharp unless we are starting in a surjective case
 which is either non-injective, or where $\rho>0$. See Example \ref{ex:reductions} for further discussion.

We will also use Proposition \ref{prop:injective-extend} to reduce to the following sequence of `critical' injective cases, of which the first was examined in Example \ref{ex:m2-critical} above.

\begin{prop}\label{prop:m2-critical} Theorem \ref{thm:m2} holds when $r$ is even and $g=(r+1)r/2$, $d=(r+2)r/2$, and when $r$ is odd and $g=(r+1)^2/2$, $d=r(r+3)/2$.
\end{prop}

Finally, the following computation is very straightforward, but is used in the proofs of both Theorem \ref{thm:m2} and Lemma \ref{lem:basic-reduction}.

\begin{prop}\label{prop:m2-difference} Given $(g,r,d)$, we have
$$\binom{r+2}{2}-(2d+1-g)=\binom{r}{2}-\rho-(g-d+r)(r-1).$$
\end{prop}

We can now complete the proof of the $m=2$ case of the Maximal Rank  Conjecture.

\begin{proof}[Proof of Theorem \ref{thm:m2}] We work by induction on $r$, with the induction hypothesis being that Theorem \ref{thm:m2} holds with
the added stipulation that for any surjective case, we can arrange for the $(g,r,d)$-sequence $\vec{\delta}$ to have at most $r$ repetitions
of every integer, and for $w$ to be unimaginative, with $c_2 \geq 2$.
We begin by proving the desired statement in the base case $r=3$.
The conditions $\rho \geq 0$ and $g-d+r>0$ imply that we must have $g \geq r+1=4$.

 Start with the surjective cases, $\binom{r+2}{2} \geq 2d+1-g$. 
 By Proposition \ref{prop:m2-difference} and the assumption $r=3$, this is equivalent to having $3-\rho-(g-d+r)\cdot 2 \geq 0$, implying that we must have $g-d+r=1$ and  $\rho \leq 1$. 
 Thus, the only two cases are the canonical case $g=4,d=6$, or the case $g=5, d=7$, which are addressed (satisfying our extra stipulations on the $(g,r,d)$-sequences 
 and $w$) in Examples \ref{ex:m2-r3-g4} and \ref{ex:m2-r3-g5}. 
 Now, the two previous cases are the only ones with $g \leq 5$, but we observe that Example \ref{ex:m2-r3-g5} was injective, with $\vec{\delta}$ 
extendable, so the $r=3$ case follows by Proposition  \ref{prop:injective-extend}.

Next, if we assume our hypothesis holds for $r-1$, Lemma  \ref{lem:basic-reduction} together with the induction hypothesis
then gives us all surjective cases except for those which are also injective and have $\rho=0$.
Now, suppose that we are in the injective case $\binom{r+2}{2} \leq 2d+1-g$, and set $s=\min(2d+1-g-\binom{r+2}{2}, \rho)$. 
Then if we set $g'=g-s, r'=r, d'=d-s$, we see that $g'-d'+r'=g-d+r$, and $\rho'=\rho-s \geq 0$,
so we have another valid case with the same $r$. 
In addition, 
$$2d'+1-g'-\binom{r'+2}{2}=2d+1-g-\binom{r+2}{2}-s=\max\left(0,
2d+1-g-\binom{r+2}{2}-\rho\right),$$ 
so $(g',r',d')$ remains in the injective case, but either has $\rho'=0$, or is simultaneously in the surjective case. 
In either case, Proposition \ref{prop:injective-extend} implies that in order to treat $(g,r,d)$, it is enough to treat $(g',r',d')$.
Combined with our previous reductions in the surjective case, we see that it is enough to treat injective cases with $\rho=0$.
 We claim that all such cases have  $g \geq (r+1) \lceil \frac{r}{2} \rceil$.
  Indeed, $\rho=0$ means that $g=(r+1)(g-d+r)$, so it then suffices to see that injectivity (together with $\rho=0$) implies that
$g -d+r \geq r/2$, which is immediate from Proposition \ref{prop:m2-difference}.
Noting that any $(g,r,d)$-sequence with $\rho=0$ is extendable, the theorem then follows from Propositions \ref{prop:m2-critical} and \ref{prop:injective-extend}.
\end{proof}

We now give the proofs of the two intermediate results, starting with the basic reduction for the surjective case.

\begin{proof}[Proof of Lemma \ref{lem:basic-reduction}]
From the definition of $g',d',r'$, it follows that  $g'-d'+r'=g-d+r$. Then,  
$$\rho'= \rho-(t-(g-d+r))=\max (0, \rho+g-d+1) \geq 0.$$
We construct $\vec{\delta}$ by adding $1$ to each entry of $\vec{\delta}'$, and inserting $t$ zeroes at the beginning of the sequence.
 In terms of Young tableaux, it is adding a height $t$ column to the left of the Tableau$'$.
 Since $t \leq r-1$, $\vec{\delta}$ will have no number appearing more than $r$ times.
 Moreover, $\vec{\delta}$ is a $(g,r,d)$-sequence: since  $g-d+r=g'-d'+r'$, it suffices to check that the added column in the Young Tableau is at least as long as the others, or equivalently
  that no number in $\vec{\delta}'$ appears more than $t$ times. 
 If $t=r-1$, this is by hypothesis. If $t=\rho+(g-d+r)$, then $\rho'=0$  and the Tableau$'$ was a rectangle with all columns  of the same height $g'-d'+r'\le t$.

If $c'=(c'_2,\dots,c'_{g'})$, we construct $c=(c_2,\dots,c_g)$ by setting 
\[ c_2=3, c_i=c_{i-1}+2, i \leq t+1;\ \ c_i=c'_{i-t}+2t+2, i \geq t+2.\] 
Then if $w'$ is unimaginative with $c'_2\geq 2$, the same will be true of $w$. 
 By construction we will have that in $T_w(\vec{\delta})$, only rows of the form $(0,j_2)$ can appear in the first $t$ columns: 
 indeed, we have $ 2d-c_{i+1}=2d-2i-1$ while $b^i_{(1,1)}=2d-2i-2$ for $i \leq t$, so the $(1,1)$ row cannot appear, and $b^i_{(j_1,j_2)} \leq b^i_{(1,1)}$ when $j_1 \geq 1$.
Now, suppose there exists a choice of $N$ rows of $T_{w'}(\vec{\delta}')$ which can be used to verify $N$-expungeability of $T_{w'}(\vec{\delta}')$.
Our claim is that using these rows (appropriately reindexed by $1$ corresponding to the shift in $\vec{\delta}$) together with the rows
$(0,0),\dots,(0,t+1)$, we can verify $(N+t+2)$-expungeability of $T_w(\vec{\delta})$.
By construction we will have precisely the rows $(0,0)$, $(0,1)$, $(0,2)$ appearing in the first column, with  entries $a^1_{\vec{j}}$ equal to $0,1,2$ respectively, 
so repeatedly applying Lemma \ref{rulesii-v}(a), we can drop these three rows. 
Next, in the following $t-1$ columns, we can have at most one new row appearing in each column, so applying Lemma \ref{rulesii-v} (c) in each case,
 we can drop each of these rows, which are rows $(0,3),\dots,(0,t+1)$.
  The remaining rows are those of the form $(j_1,j_2)$ with $j_1>0$, which appear only in the final $g'$ columns. 
  These $g'$ columns of $T_w(\vec{\delta})$ agree precisely with the $T_{w'}(\vec{\delta}')$ one obtains from considering $\vec{\delta}$ as a $(t+1)$-shifted  $(g',r',d)$-sequence,
and the latter is $N$-expungeable by Proposition \ref{prop:invariance}.
We thus conclude the first statement of the lemma, and the particular case of $(2d'+1-g')$-expungeability follows immediately.

Finally, we verify by direct calculation that 
$$\binom{r'+2}{2}-(2d'+1-g')=\binom{r+2}{2}-(2d+1-g)-(r-1-t)$$
For the last statement in the Lemma, it suffices to prove that if $t<r-1$ and either $\binom{r+2}{2}>2d+1-g$ or $\binom{r+2}{2} =2d+1-g$ and $\rho>0$,  then $\binom{r+2}{2}-(2d+1-g)\geq r-1-t$. 
Now, if $t<r-1$ then $r-1-t=r-1-\rho-(r+g-d)$. 
Writing $\ell=r+g-d$, Proposition \ref{prop:m2-difference} implies first that our desired inequality can be written $\binom{r}{2}-\rho-\ell(r-1) \geq r-1-\rho-\ell$, 
and second, that under either of our hypotheses, we have $\binom{r}{2}>\ell(r-1)$.
The desired inequality simplifies to $(r-1)(r-2)/2 \geq \ell(r-2)$, or equivalently, $\ell \leq (r-1)/2$, while the given inequality yields $\ell < r/2$ and hence $\ell \leq (r-1)/2$, as desired.
\end{proof}

Finally, we treat our sequence of critical cases. Recall that an example of the $r=4$ case is given in Example \ref{ex:m2-critical}.

\begin{proof}[Proof of Proposition \ref{prop:m2-critical}]
We are assuming  that if  $r$ is even,  $g=(r+1)r/2$, $d=(r+2)r/2$, and if $r$ is odd then $g=(r+1)^2/2$, $d=r(r+3)/2$.

Write $\ell=g-d+r$, so that $\ell=\frac{r}{2}$ if $r$ is even, and $\ell=\frac{r+1}{2}$ if $r$ is odd. 
Choose $\vec{\delta}=\underbrace{0,\dots,0}_{\ell \text{ times}}, \underbrace{1,\dots,1}_{\ell \text{ times}}, \dots , \underbrace{r,\dots,r}_{\ell \text{ times}}$,
 or equivalently, the Young Tableau is a rectangle filled successively by column.
Choose $c=(c_2,\dots,c_g)$, where $c_2=2$, and 
\[ \text{for }2<i \leq g/2+1, \ \  c_i=c_{i-1}+\begin{cases} 2: & i \not\equiv 2 \pmod{\ell} \\  3 : & i \equiv 2 \pmod{\ell}.\end{cases}\]
\[ \text{for } g/2+1 < i \leq g, \ \  c_i=c_{i-1}+\begin{cases} 2: & i \not\equiv 1 \pmod{\ell} \\ 3 : & i \equiv 1 \pmod{\ell}.\end{cases}\]
The result is that we have $r+1$ blocks consisting of $\ell$ columns each, which can be analyzed essentially independently of one another. 
In addition, the situation is symmetric about the middle. 
Because our $w$ is unimaginative, in order to analyze the erasures in $T_w(\vec{\delta})$, we can simply look at how a given $(a^i_{\vec{j}}, b^i_{\vec{j}})$ compares to $(c_i, 2d-c_{i+1})$; 
see Remark \ref{rem:zeroing-pattern}.
Specifically, if $\vec{j}=(j_1,j_2)$, the columns are erased up until the  first time that  $b^i_{\vec{j}} \geq 2d-c_{i+1}$ (equivalently, $a^{i+1}_{\vec{j}} \leq c_{i+1}$), 
and will be erased after the last time  that $a^i_{\vec{j}} \geq c_i$. 
In particular, the $(j_1,j_2)$ row appears for the first time in the $i$th column if and only if $a^i_{\vec{j}}>c_i$ and $a^{i+1}_{\vec{j}} \leq c_{i+1}$.

Labeling our blocks $0,\dots,r$, we have the following formulas: if we write $i=\ell \cdot \alpha+\beta$ with  $0<\beta \leq \ell$, so that the $i$th column of $T(\vec{\delta})$ 
is the $\beta$th column of the $\alpha$th block, then provided that  $i \leq \frac{g}{2}$, we have 
$$c_i=2i-2+\alpha-\delta_{\beta,1}, \text{ and } a^i_j=\begin{cases} i+j-1: & \alpha<j \\ i+j-\beta: & \alpha=j \\ i+j-\ell-1: & \alpha>j, \end{cases}$$
where $\delta_{\beta,1}$ is the Kronecker $\delta$ function.
We then analyze which rows appear for the first time (reading left to right) in each column.

In the first column of the $k$th block, with $0 \leq k < \ell$, we will have the first appearances of the rows of the form $(j,\ell+k-j)$, for $j=0,\dots,k-1$. 
For the $i'$th column of the $k$th block, with $1<i'\leq k$, the only new row is the $(k,k)$ row, which occurs for the first time in the $\lceil k/2 \rceil$th column
of the $k$th block (if $k\leq 2$, the $(k,k)$ row occurs in the first column of the $k$th block). 
For $k<i' \leq \ell$, the row $(k,i')$ will appear for the first time in the $i'$th column of the $k$th block (note that this includes the $(0,1)$ row occurring in the $1$st
column of the $0$th block; for $k>0$, we will have $i'>1$).

Now, if $r$ even, the procedure we use to show that $T_w(\vec{\delta})$ is  $\binom{r+2}{2}$-expungeable is as follows: for $k<r/2$, we show that
if all rows appearing in previous blocks have already been dropped, then we can work from left to right in the $k$th block to drop all rows  appearing in that block. 
For $k>r/2$ we apply the same procedure from right to left, and finally in the central $r/2$ block, we have dropped all rows appearing in any other block, and we show that the rows
only appearing in the $r/2$ can be dropped as well.

The desired dropping behavior is clear in the 0th block, since according to the above description, we see that when we work from left to right, there are never more than
two new rows appearing in a given column, so repeated use of Lemma \ref{rulesii-v} (c) suffices to drop all rows. 
The same argument works for the 1st block. In  the $k$th block for $1<k<r/2$, we have at most $k+1$ new rows appearing in the first column: 
$(0,k+r/2), (1,k+r/2-1),\dots,(k-1,r/2+1)$ always appear, as well as $(k,k)$ when $k = 2$. 
However, in the next $k-1$ columns we have no new rows appearing other than $(k,k)$ in the $\lceil k/2 \rceil$th column, and in each subsequent column we have only one new row appearing. 
We claim that we can use Lemma \ref{rulesvii} with $n=k$ to drop the $k+1$ rows appearing in the first $k$ columns; 
this will then imply that the rest of the rows in the block can be dropped just using Lemma \ref{rulesii-v}(d), as in the 0th block.
 Now, within the $k$th block, the rows $(0,k+r/2), (1,k+r/2-1),\dots,(k-1,r/2+1)$ are all identical, starting at $(2k(r/2)+k, 2d-2i(r/2)-k-2)$, with the left side increasing by $2$
and the right side decreasing by $2$ in each subsequent column. 
Note that this precisely matches the behavior of $w$, so in fact these rows all appear throughout the $k$th block. In contrast, the $(k,k)$th row is a constant
$(2k(r/2 + 1), 2d-2k(r/2 + 1))$, and appears in the $\lceil k/2 \rceil$th column only if $k$ is odd, and in the $\lceil k/2 \rceil$th and $(\lceil k/2 \rceil+1)$st columns if $k$ is even. 
Because no other rows appear in these columns, we can apply Lemma \ref{rulesvii}, as claimed.

By symmetry, we can also work from right to left to drop all rows except those which occur solely in the $r/2$ block. 
But these rows are precisely the rows $(0,r),(1,r-1),\dots,(r/2,r/2)$, and we can again apply Lemma  \ref{rulesvii}, this time with $n=r/2$, to drop all the remaining rows. 
This handles the case that $r$ is even.

Next, if $r$ is odd, the situation is almost the same, except that the number of blocks is even.
Accordingly, we can drop all rows by first going from left to right in the first $(r+1)/2$ blocks, and then going right to left in the remaining  $(r+1)/2$ blocks.
We again have that the $0$th and first blocks each have at most two new  rows in each column, so we can eliminate all the rows simply using Lemma  \ref{rulesii-v}((c).
We also still have that the $k$th block for $k \leq (r+1)/2$ will have $k+1$ rows occurring in the first $i$ columns, and then one additional row in each subsequent column, 
so just as before, we can apply  Lemma  \ref{rulesvii}  to treat the first $k$ columns of the block simultaneously, and then Lemma  \ref{rulesii-v}((c)  to deal with the remaining columns. 
As before, the situation is  symmetric, so applying the same procedure from right to left on the remaining $(r+1)/2$ blocks will allow us to drop all rows, as desired.
\end{proof}

\begin{ex}\label{ex:reductions} We consider some examples of the reduction processes from the proof of Theorem \ref{thm:m2}.

First, if we have the canonical case, with $g=r+1$ and $d=2r$, then applying  Lemma \ref{lem:basic-reduction} we have $\rho=0$ and $g+r-d=1$, so $t=1$,  
and we get $r'=r-1$, $g'=g-1=r'+1$, $d'=d-2=2r'$. 
Thus, we reduce to the  canonical case in genus one less.

Next, suppose we have an injective case with $r$ even and $g$ strictly smaller than the critical case $\frac{r(r+1)}{2}$.
 Then our reduction process will lead to an injective (and surjective) case with  $r'=r-1$, and $g'$ strictly smaller than the critical case $\frac{(r'+1)^2}{2}$.
However, the next step in the reduction will not necessarily stay below the critical case. 
For instance, consider the case $r=6$, $g=20$, $d=24$. 
This is injective, with $\rho=6$ and $g+r-d=2$, and $2d+1-g-\binom{r+2}{2}=1$. 
In this case, the $s$ from the proof of Theorem \ref{thm:m2} is equal to $1$, so we first use Proposition \ref{prop:injective-extend} to reduce to considering the
case $r'=r=6, g'=g-1=19, d'=d-1=23$. 
This case is now injective and surjective, with $g'+r'-d'=2$ and $\rho'=5$, so when we apply Lemma  \ref{lem:basic-reduction}, we have $t=r'-1=5$, 
and reduce to the case $r''=r'-1=5$, $g''=g'-5=14$, $d''=d'-6=17$, which is still an injective and surjective case, and has $g''=14<\frac{(r''+1)^2}{2}=18$. 
The next step is another reduction via Lemma \ref{lem:basic-reduction}, where now we have $t=r''-1=4$, so the next reduction ends up at the critical case $r'''=4$,
$g'''=10$, $d'''=12$, which is addressed directly in Proposition  \ref{prop:m2-critical} (and in Example \ref{ex:m2-critical}).

Finally, consider what happens for the critical case $r=4$, $g=10$, $d=12$ if instead of handling the case directly as in our proof of Theorem 
\ref{thm:m2}, we instead attempt to apply Lemma \ref{lem:basic-reduction}. 
This case has $g+r-d=2$ and $\rho=0$, so we will have $t=2$, so we will `reduce' to the case $r'=3$, $d'=9$, $g'=8$. 
However, this latter case is non-surjective: $2d'+1-g'=11$, while $\binom{r'+2}{2}=10$. 
Thus, the best we can do in this case is to show that we have rank $10$ for $(g',r',d')=(8,9,3)$. 
Then Lemma \ref{lem:basic-reduction} says that we have rank at least $10+t+2=14$ for $(g,r,d)=(10,4,12)$, but the conjecture is that this case should have rank $15$. 
Thus, in this case Lemma \ref{lem:basic-reduction} does provide partial information, but falls short of the sharp result.
\end{ex}

\section{Observations on surjectivity}\label{sec:surjective}

We now consider the surjective range, where $\binom{r+m}{m} \geq md+1-g$.
We prove surjectivity in a range of cases for $m=3$ in Corollary \ref{cor:surj} below, but while these cases are somewhat different
from those considered by Jensen and Payne in \cite{j-p3}, they are fully covered by Ballico \cite{ba4}. 
For us, the purpose of this section is to illustrate a rather distinct type of argument from that found in other sections, and simultaneously to explain how the number $md+1-g$,
which arises naturally from the Riemann-Roch theorem on smooth curves, can be seen also in the context of limit linear series and our elementary criterion.
We start our discussion with the limit  linear series point of view, but this will not be used elsewhere: the  criteria which we will actually apply are stated in Proposition 
\ref{prop:surj-crit} below, and proved directly from our elementary criterion. 

Suppose we have $w=(c_2,\dots,c_g)$ inducing multidegree $(d_1,\dots,d_g)$, with $\sum_i d_i= md$. 
Then we can study $\Gamma(X_0,\sL_{w})$ via the Riemann-Roch theorem for reducible curves, 
but for our purposes, it is more instructive to carry out a direct analysis.
Considering restriction to components and nodes gives us an exact sequence
\begin{equation}\label{eq:gluing} 0 \to \Gamma(X_0,\sL_{w}) \to  \bigoplus_{i=1}^g \Gamma(Z_i,\sL_{w}|_{Z_i}) \to \bigoplus_{i=1}^{g-1} k, \end{equation}
and assuming all the $d_i$ are positive, we have  $\dim \Gamma(Z_i,\sL_{w}|_{Z_i})=d_i$ for $i=1,\dots,g$.
 We thus see that $\dim \Gamma(X_0,\sL_{w}) \geq md+1-g$, with equality if and only if the last map of \eqref{eq:gluing} is surjective. We then have

\begin{prop}\label{prop:gluing-surject} In the above situation, suppose that  $md > 2g-2$, and we have $d_1 \geq 1$, $d_i \geq 2$ for $1<i<g$, and $d_g \geq 1$. 
Then  \eqref{eq:gluing} is surjective, so $\dim \Gamma(X_0,\sL_{w}) = md+1-g$.
\end{prop}

\begin{proof} Since $md>2g-2$, there is some $i_0$ for which the above inequality on $d_{i_0}$ becomes strict. If $1<i_0<g$, and $d_{i_0}>2$,
then the map $\Gamma(Z_{i_0},\sL_{w}|_{Z_{i_0}}) \to k^{\oplus 2}$  induced by restriction to $P_{i_0}$ and $Q_{i_0}$ is necessarily surjective.
For $1<i<i_0$, because $d_i \geq 2$ we have surjectivity of the map  $\Gamma(Z_i,\sL_{\md(w)}|_{Z_i}) \to k$ induced by restriction to $P_i$, and
similarly for $i<i_0<g$ we have surjectivity of the map  $\Gamma(Z_i,\sL_{w}|_{Z_i}) \to k$ induced by restriction to $Q_i$.
Putting these together gives surjectivity of \eqref{eq:gluing}. A similar analysis of the cases $i_0=1$ and $i_0=g$ yields the proposition.
\end{proof}

\begin{rem}\label{rem:deg-bd}
The hypothesis in Proposition \ref{prop:gluing-surject} that $md>2g-2$ is quite mild: if $m=3$, it is always satisfied,
while for $m=2$, we observe that if we are in the surjective range, so that $\binom{r+2}{2} \geq 2d+1-g$, then we necessarily have $d > g$.
Indeed, Proposition \ref{prop:m2-difference} may be rewritten equivalently as $\binom{r+2}{2}-(2d+1-g)=(d-g)(r-1)-\binom{r}{2}-\rho$, from which
$d > g$ follows immediately when the lefthand side is nonnegative.
\end{rem}

The above point of view gives  a way to choose the sections that one wants to be linearly independent
If reading from left to right, the first column (corresponding to $Z_1$) has full $d_1$-dimensional span, and each subsequent column has full $(d_i-1)$-dimensional span among the sections
not appearing in previous columns, then we obtain surjectivity choosing the sections that appear in each of these columns.

We now generalize  the above observation and derive some consequences. 
In the Proposition  below, the  case $i_0=1$ corresponds to the above situation. 

\begin{prop}\label{prop:surj-crit} 
Assume  that $\binom{r+m}{m} \geq md+1-g$, $w=(d_1,\dots,d_g)$, such that  $d_i>0$ for $i>1$ and for some $i_0\geq 1$, $\sum_{i=1}^{i_0} (d_i-1)\geq 0$.
 Assume that there is some choice of $md+1-g$ rows of $T_w(\vec{\delta})$ such that Lemma \ref{rulesii-v} can be used in component $Z_{i_0}$ 
 to prove the independence of sections correspondng to $1+\sum_{i=1}^{i_0} (d_i-1)$ rows,
  and then for each $i>i_0$, to prove on $Z_i$ the independence of the sections corresponding to $d_i-1$ aditional rows none of which occur in previous columns. 
  Then $T_w(\vec{\delta})$ is $(md+1-g)$-expungeable.

In particular, suppose that $w$ and $i_0$ are as above, and  $T_w(\vec{\delta})$ has the property that the non-erased portion of each row is contiguous. 
Then if every number between  $0$ and $md$ other than $1,\dots,i_0-1$ and $md-1$ occurs among the $a^i_{\vec{j}}$ of $T_w(\vec{\delta})$ for $i \geq i_0$, 
we have that $T_w(\vec{\delta})$ is $(md+1-g)$-expungeable. 

More generally, if $w$ and $i_0$ are as above, and $T_w(\vec{\delta})$ has the property that the non-erased portion of each row is contiguous, suppose further that:
\begin{itemize}
\item in the $i_0$th column, either $0,i_0,i_0+1,\dots,c_{i_0+1}-1$ all occur among the $a^{i_0}_{\vec{j}}$, or $0,i_0,i_0+1,\dots,c_{i_0+1}-2$ all occur, with $c_{i_0+1}-2$ occurring at least twice;
\item for each $i>i_0$, in the $i$th column either $c_i+1,\dots,c_{i+1}-1$ all occur among the $a^{i}_{\vec{j}}$, or $c_i+1,\dots,c_{i+1}-2$ all occur, with $c_{i+1}-2$ occurring at least twice.
\end{itemize}
Then $T_w(\vec{\delta})$ is $(md+1-g)$-expungeable. 
\end{prop} 

Note that the condition on the non-erased portion of each row being contiguous is automatically satisfied for unimaginative $w$, or more generally for $w$ which are steady with respect to 
$T(\vec{\delta})$.

\begin{proof} The hypothesis of the first statement is just a special form of $(md+1-g)$-expungeability, since $1+\sum_{i=1}^g (d_i-1)=md+1-g$.

For the second statement, we observe that a number $a$ can occur as $a^i_{\vec{j}}$ in $T_w(\vec{\delta})$ only if we have $c_i \leq a \leq c_{i+1}$
 (here, we take $c_1=0$ and $c_{g+1}=md$): certainly, we must have $a \geq c_i$, but we must likewise have $b^i_{\vec{j}} \geq md-c_{i+1}$, and because
$a^i_{\vec{j}} + b^i_{\vec{j}} \leq md$, we also obtain $a \leq c_{i+1}$.
Now, we will denote by $S$ the set of $N$ rows chosen to verify $N$-expungeability, which we will construct one column at a time.

By hypothesis, we have $c_i<c_{i+1}$ for all $i>1$, so we see that if any of $0,\dots,c_{i_0+1}-1$ occur among the $a^i_{\vec{j}}$ in the $i$th  column with $i \geq i_0$, we must have $i=i_0$.
We have supposed that $c_{i_0+1}-(i_0-1)$ ($=1+\sum_{i=1}^{i_0} (d_i-1)$) of these values do occur,
so we can choose $S$ to contain exactly one row with each of these values in the $i_0$th column. 
Then, we can apply Lemma \ref{rulesii-v} (a) to drop the remaining $1+\sum_{i=1}^{i_0} (d_i-1)$ rows in this column. 
Then for $i>i_0$, the values $c_i+1,\dots,c_{i+1}-1$ can only occur in the $i$th column. 
Moreover, if $c_i+1 \leq a^i_{\vec{j}}$, then the $\vec{j}$th row cannot occur in a previous column, since $a^i_{\vec{j}}>c_i$ implies that the row cannot appear in the $(i-1)$st column, 
and we have assumed that the non-erased portions of each row are contiguous. 
Thus, we may again add rows to $S$ so that the $i$th column contains each value from $c_i+1$ to $c_{i+1}-1$ exactly once, and we can again apply Lemma \ref{rulesii-v} (a)
 to drop $c_{i+1}-c_i-1=d_i-1$ rows from the $i$th column.
Note that by construction, the number of rows in $S$ is precisely $1+\sum_{i=1}^g (d_i-1)=md+1-g$,and applying the first statement of the proposition, we conclude the desired result.

Finally, the more general case proceeds by exactly the same argument, except that in columns where $c_{i+1}-1$ is omitted, but $c_{i+1}-2$
occurs at least twice, we use Lemma \ref{rulesii-v} (c) to drop the final two rows in the column.
\end{proof}

\begin{ex} Consider the canonical series, with $r=g-1$ and $d=2g-2$. 
In this case, the only $(g,r,d)$-sequence is $\vec{\delta}=0,1,\dots,g-1$. The $i$th column of $T'(\vec{\delta})$ is:
\begin{center}
\begin{tabular}{lr}
$i-2$ & $2g-i-1$ \\
$i-1$ & $2g-i-2$ \\
$\vdots$ & $\vdots$ \\
$2i-4$ & $2g-2i+1$ \\
$2i-2$ & $2g-2i$ \\
$2i-1$ & $2g-2i-2$  \\
$\vdots$ & $\vdots$ \\
$i+g-2$ & $g-i-1$ \\
\end{tabular}
\end{center}
That is, 
\[ a^i_j=j+i-2, \ \ j\le i-2, \ a^i_j=j+i-1,\ j\ge i-1,\ \ \ b^i_j=2g-i-j-1, \ j \leq i-1, \ b^i_{j}=2g-i-j-2, j\ge i.\]

For any $m \geq 2$, $T(\vec{\delta})$ is obtained by adding $m$-tuples of rows of $T'(\vec{\delta})$.
Set $c=(c_2,\dots,c_g)$, with $c_i=a^i_{(i-2,\dots,i-2,g-1)}$ for all $i$. 
Then,  $c_{i+1}-c_i=2(m-1)+1$ for all $i$.

The rows $(0,\dots,0,j)$ for $0 \leq j \leq g-1$ all appear in the first column of  $T_w(\vec{\delta})$, 
and the corresponding values of $a^1_{\vec{j}}$ are $0,1,\dots,g-1=c_2-1$.
Next, in the $i$th column for $1 < i < g=r+1$, rows of the form $\vec{j}=(i-2,\dots,i-2,i-1,\dots,i-1,j)$ with $j =g-2$ or $g-1$ all appear in $T_w(\vec{\delta})$, except for $(i-2,\dots,i-2,g-2)$.
 The corresponding values of $a^i_{\vec{j}}$ yield $c_i,c_i+1,\dots,c_{i+1}-1$.
Finally, in the $g$th column, the rows  $(j_1,g-2,\dots,g-2,g-1,\dots,g-1, j_m)$ with $j_1=g-3$ or $g-2$ and $j_m=g-1$ all  appear with the exception of $(g-3,g-2,\dots,g-2,g-1)$ 
(which has $a^i_{\vec{j}}=c_g-1$), and the values of $a^i_{\vec{j}}$ these yield cover $c_g,c_g+1,\dots,md-2$.
Then the row $(g-1\dots,g-1)$ has $a^i_{\vec{j}}=md$, and (the $i_0=1$ case of) Proposition \ref{prop:surj-crit} gives us surjectivity.
\end{ex}

We now apply Proposition \ref{prop:surj-crit} to prove surjectivity within certain ranges, generalizing the canonical linear series, 
and including many cases which do not fall in the surjective range for $m=2$.
Recall from the introduction that although we only treat directly the case $m=3$,  surjectivity then follows for all higher $m$.

\begin{rem}\label{rem:allai}
Suppose that  $c=(c_2,\dots,c_g)$, and that the $c_i$ are nondecreasing.
Then in the $i$th column, each $a^i_{\vec{j}}$ whose corresponding section $s_{\vec{j}}$ does not vanish on the curve $Z_i$ is at least $c_i$. 
If we want every number to appear as some $a^{i'}_{\vec{j}}$ in $T_w(\vec{\delta})$, we need $c_i-1$ to appear as an  $a^{i'}_{\vec{j}}$ for some  $i'<i$ and unless  $c_{i-1}=c_i$,  $i'=i-1$.
If $c_i-1=a^{i-1}_{\vec{j}}$ for some $\vec{j}$, then $b^{i-1}_{\vec{j}} \geq md-c_i$. 
Since $a^{i-1}_{\vec{j}}+b^{i-1}_{\vec{j}}$ is given by $md-m$ plus the number of times $\delta_{i-1}$ occurs in $\vec{j}$, 
we conclude that $\delta_{i-1}$ must occur at least $m-1$ times in $\vec{j}$.
 Similarly, if $c_i-n$ appears as $a^{i-1}_{\vec{j}}$ for $1 \leq n < m$, we conclude that  $\delta_{i-1}$ occurs at least $m-n$ times in $\vec{j}$.
  If $w$ is unimaginative, we then derive a necessary and sufficient condition for numbers of the form $c_i-n$ to appear as $a^{i-1}_{\vec{j}}$ in $T_w(\vec{\delta})$ 
  for some $\vec{j}$: first, we must have $c_i-n \geq c_{i-1}$, second, $c_i-n$ must appear as some $a^{i-1}_{\vec{j}}$ in $T(\vec{\delta})$, and third, if $n<m$, it must do so in a row 
$\vec{j}$ with at least $m-n$ occurrences of $\delta_{i-1}$ in $\vec{j}$.
\end{rem}

\begin{cor}\label{cor:surj} Suppose that $m=3$, and $(g,r,d)$ satisfy $\rho \geq 0$. Then the Maximal Rank Conjecture holds in the following cases: 
\begin{ilist}
\itm if $g-d+r=1$, and $2r-3 \geq \rho + 1$;
\itm if $g-d+r=2$, $r \geq 4$, and $2r-3 \geq \rho + 2$.
\end{ilist}
Moreover, the locus of chains of genus-$1$ curves is not in the closure of the locus in $\cM_g$ where the maximal rank condition fails.
\end{cor}

\begin{proof} 
In case (i), we set $\vec{\delta}$ to be the sequence whose first $\rho$ entries are $0$, followed by $0,1,\dots,r$.
 As by assumption,  $g-d+r=1$,  it follows that $\rho+(r+1)=g$ and this choice gives a $\vec{\delta}$ sequence. 
Set $n=\min(r-1,\rho+2)$,  $c=(c_2,\dots,c_g)$, where  
\[ c_i=-3(\rho+3-n-i)-1,\  2 \leq i \leq \rho+3-n; \ \ \ c_i=a^i_{\vec{j}_i} ,\ i \ge \rho+4-n \]
with
\[ \vec{j}_{\rho+2-t} = (0,n-t,n-t) \text{ for }0 \leq t \leq n-2,\ \ \vec{j}_{\rho+t} = (t-2,t-2,r) \text{ for }3 \leq t \leq r+1\]
We first check that   $w$ is unimaginative: 
\[ c_i-c_{i-1}= 3 \text{ for }2<i \leq \rho+3-n\]
\[ c_{\rho+4-n}-c_{\rho+3-n}=a_{(0,2,2)}^{\rho+4-n}-(-1)=1+2(\rho+5-n)\ge 4\] 
\[ c_{\rho+2-t}-c_{\rho+1-t}=a^{\rho+2-t}_{(0,n-t,n-t)}-a^{\rho+1-t}_{(0,n-t-1,n-t-1)}=4. \text{ for }0 \leq t < n-2,\]   
\[ c_{\rho+3}-c_{\rho+2}=a^{\rho+3}_{(1,1,r)}-a^{\rho+2}_{(0,n,n)}=(2(\rho+2)+r+\rho+2)-2(\rho+1+n)=\rho+r-2n+4\]
and as  $n\leq r-1,  n\leq \rho+2$, then $\rho+r-2n+4 \geq 3$.
\[ c_{\rho+t}-c_{\rho+t-1}=a^{\rho+t}_{(t-2,t-2,r)}-a^{\rho+t-1}_{(t-3,t-3,r)}=5 \text{ for } 3 < t \leq r+1\] 

Also,  $T_w(\vec{\delta})$ satisfies the condition of Proposition \ref{prop:surj-crit}. 
Specifically, no rows will appear in the first $\rho+2-n$ columns. 
Using the inequality $2r-3 \geq \rho+1$, if $r-1\ge n$ we obtain  $r \geq \rho+3-n$ while if $n=\rho +2$, $2r-3 \geq 1=\rho+3-n$.
Then in the $(\rho+3-n)$th column, rows of the form $(0,0,j_3)$ with $0 \leq j \leq \rho+3-n$ will yield $a^i_{\vec{j}}$ equal to $0,\rho+3-n,\rho+4-n,\dots,2(\rho+3-n)-1$.
Then the rows $(0,1,1)$, $(0,1,2)$, $(0,2,2)$, $(0,1,3)$ give $2(\rho+3-n)$, $2(\rho+3-n)+1$, and $2(\rho+4-n)$ twice. 
We thus have the numbers $0$ through $2(\rho+4-n)$ occurring with $\rho+2-n$ gaps in this column, and with $2(\rho+4-n)$ occurring twice. 
Then in the $\rho+2-t$th column for $t=n-2,\dots,1$, we will have the rows  $(0,n-t,n-t)$, $(0,n-t,n+1-t)$, $(0,n+1-t,n+1-t)$ 
and $(0,n-t,n+2-t)$ contributing $2(\rho+1+n-2t)$, $2(\rho+1+n-2t)+1$, and  $2(\rho+2+n-2t)$ twice.
In each case, we will have skipped $c_{\rho+2-t}-1=2(\rho+1+n-2t)-1$, but  we can still apply Proposition \ref{prop:surj-crit} because 
$2(\rho+1+n-2t)-2$ will have appeared twice in the previous column.

Next, in the $(\rho+2)$nd column, the rows $(1,1,j_3)$ for  $1 \leq j_3 \leq r$ cover all values from $\max(3(\rho+2),c_{\rho+2})$ to 
$c_{\rho+3}-1$. If $3(\rho+2) \leq c_{\rho+2}$, these rows suffice in this column, and otherwise, we must have $n=r-1$.
The hypothesis $2r-3 \geq \rho+1$ implies that $c_{\rho+2} \geq 3(\rho+2)-2$, so adding in the rows $(0,r-1,r-1)$ and $(0,r-1,r)$ allows us to cover all values between
$c_{\rho+2}$ and $c_{\rho+3}-1$.
In the $(\rho+t)$th column for $t=3,\dots,r$, the rows $(t-2,t-2,r)$, $(t-2,t-1,r-1)$, $(t-2,t-1,r)$, $(t-1,t-1,r-1)$,  $(t-1,t-1,r)$ give the values from $c_{\rho+t}$ to $c_{\rho+t+1}-1$.
Finally, in the $(\rho+r+1)$st column, the rows $(r-1,r-1,r)$, $(r-2,r,r)$, $(r-1,r,r)$, $(r,r,r)$ give the values from $c_{\rho+r+1}$ to $3d$, skipping only $3d-1$.
Applying Proposition \ref{prop:surj-crit}, we conclude the desired  statement for case (i).

For case (ii), the pattern is similar, but a bit more complicated.
We set $\vec{\delta}$ to be the sequence whose first $\rho$ entries are $0$, followed by $0,0,1,1,\dots,r,r$.
 As by assumption,  $g-d+r=2$,  it follows that $\rho+2(r+1)=g$ and this choice gives a $\vec{\delta}$ sequence. 
 Define  
 \[ n=\min(r-1,\rho+1), \text {if } \rho>0,\ \ n=2 \text{  if  } \rho=0\] 
\[ c=(c_2,\dots,c_g), c_i= -3(\rho+4-n-i)-1\text{ for }2 \leq i \leq \rho+4-n, \ \ c_i=a^i_{\vec{j}_i},  \text{ for } \rho+4-n <i\]
with
\[ \vec{j}_{\rho+3-t} = (0,n-t,n-t) \text{ for }0 \leq t \leq n-2, \]
\[ \vec{j}_{\rho+2t} = (t-1,t-1,r-2), \ 2 \leq t \leq r-2;\ \vec{j}_{\rho+2t+1} = (t-1,t-1,r),   \ 2 \leq t \leq r-1;\]
\[ \vec{j}_{\rho+2r-2} = (r-3,r-2,r), \vec{j}_{\rho+2r} = (r-2,r-1,r-1), \vec{j}_{\rho+2r+1} = (r-3,r-1,r), \vec{j}_{\rho+2r+2} = (r-2,r,r) \]

We check that $w$ is unimaginative: 
\[ c_i-c_{i-1}=3\text{ for }2<i \leq \rho+4-n,\]
\[ c_{\rho+5-n}-c_{\rho+4-n}=a^{\rho+5-n}  _{(0,2,2)}-(-1)= 1+2(\rho+6-n),\]
\[ c_{\rho+3-t}-c_{\rho+2-t}= a^{\rho+3-t}  _{(0,n-t,n-t)}-a^{\rho+2-t}  _{(0,n-t-1,n-t-1)} =  4,  \text{ for }0 \leq t \leq n-3,\] 
\[ c_{\rho+4}-c_{\rho+3}=  a^{\rho+4}_{(1,1,r-2)} -a^{\rho+3}  _{(0,n,n)}=  (2(\rho+3)+r+\rho+1)-2(\rho+2+n)=\rho+r-2n+3,\]
If $\rho=0$, as $r\ge 4$,\ $\rho+r-2n+3=r-4+3\ge 3$. 
If $\rho>0$, as $n=\min(r-1,\rho+1)$, $\rho+r-2n+3\ge \rho+r-r+1-\rho-1+3\ge 3$.
\[c_{\rho+2t+1}-c_{\rho+2t}=a^{\rho+2t+1}_{(t-1,t-1,r)} -a^{\rho+2t}  _{(t-1,t-1,r-2)}=3\text{ for }2 \leq t \leq r-2,\] 
\[ c_{\rho+2t}-c_{\rho+2t-1}=a^{\rho+2t}  _{(t-1,t-1,r-2)}- a^{\rho+2t-1}_{(t-2,t-2,r)} =5\text{ for }3 \leq t \leq r-2,\]
\[c_{\rho+2r-2}-c_{\rho+2r-3}=5, \ c_{\rho+2r-1}-c_{\rho+2r-2}=3, \ c_{\rho+2r}-c_{\rho+2r-1}=3 \]
\[ c_{\rho+2r+1}-c_{\rho+2r}=3, \ c_{\rho+2r+2}-c_{\rho+2r+1}=5. \] 

We again verify that $T_w(\vec{\delta})$ will satisfy the condition of Proposition \ref{prop:surj-crit}. 
Specifically, no rows will appear in the first $\rho+3-n$ columns. 
We clam that $r \geq \rho+4-n$: if $\rho=0$, $\rho+4-n=2<4\leq r$, while if $\rho\not= 0$, from the definition of $n$,  $\rho+4-n$ equals either $\rho+5-r$  or $3$
and both these quantities are at most $r$ from the assumptions $2r-3\geq \rho+2$ and $r \geq 4$.
 Then in the $(\rho+4-n)$th column, rows of the form $(0,0,j_3)$ with $0 \leq j \leq \rho+4-n$ will include $a^i_{\vec{j}}$ equal to $0,\rho+4-n,\rho+5-n,\dots,2(\rho+4-n)-1$.
Then the rows $(0,1,1)$, $(0,1,2)$, $(0,2,2)$, $(0,1,3)$ give $2(\rho+4-n)$, $2(\rho+4-n)+1$, and $2(\rho+5-n)$ twice. 
We thus have the numbers $0$ through $2(\rho+5-n)$ occurring with $\rho+3-n$ gaps in this column, and with $2(\rho+5-n)$ occurring twice. 
Then in the $\rho+3-i$th column for $i=n-2,\dots,1$, we will have the rows  $(0,n-i,n-i)$, $(0,n-i,n+1-i)$, $(0,n+1-i,n+1-i)$ and
$(0,n-i,n+2-i)$ contributing $2(\rho+2+n-2i)$, $2(\rho+2+n-2i)+1$, and 
$2(\rho+3+n-2i)$ twice.
In each case, we will have skipped $c_{\rho+3-i}-1=2(\rho+2+n-2i)-1$, but 
we can still apply Proposition \ref{prop:surj-crit} because 
$2(\rho+2+n-2i)-2$ will have appeared twice in the previous column.

Next, in the $(\rho+3)$rd column, the rows $(1,1,j_3)$ for 
$1 \leq j_3 \leq r-2$ cover all values from $\max(3(\rho+3),c_{\rho+3})$ to 
$c_{\rho+4}-1$.
If $3(\rho+3) \leq c_{\rho+3}$, these rows suffice in this column, and
otherwise, the hypothesis
$2r-3 \geq \rho+2$ implies that adding the rows $(0,n,n)$ and $(0,n,n+1)$
suffices to cover all values from $c_{\rho+3}$ up to $3(\rho+3)-1$.
In the $(\rho+2i)$nd column for $i=2,\dots,r-2$, the rows
$(i-1,i-1,r-2)$, $(i-1,i-1,r-1)$ and $(i-1,i-1,r)$ give the values from 
$c_{\rho+2i}$ to $c_{\rho+2i+1}-1$. In the
$(\rho+2i+1)$st column for $i=2,\dots,r-2$,
the rows $(i-1,i-1,r)$, $(i-1,i,r-2)$, $(i-1,i,r-1)$, $(i-1,i,r)$, 
$(i,i,r-2)$ give the values from $c_{\rho+2i+1}$ to $c_{\rho+2i+2}-1$.
We have to change the pattern slightly in the final
five columns, as follows: in the $(\rho+2r-2)$nd column, the final row
of the previous column was $(r-2,r-2,r-2)$, but this does not appear in
the $(\rho+2r-2)$nd column, because $c_{\rho+2r-2}$ was chosen to be one
larger than the corresponding $a^i_{\vec{j}}$. Instead, $c_{\rho+2r-2}$
will be achieved by the $(r-3,r-2,r)$ row,
and then the $(r-2,r-2,r-1)$ and $(r-2,r-2,r)$ rows cover through
$c_{\rho+2r-1}-1$. In the $(\rho+2r-1)$st column, the rows 
$(r-2,r-2,r)$, $(r-3,r-1,r-1)$, $(r-2,r-1,r-1)$ cover from $c_{\rho+2r-1}$
to $c_{\rho+2r}-1$. In the $(\rho+2r)$th column, the rows
$(r-2,r-1,r-1)$, $(r-3,r-1,r)$, $(r-1,r-1,r-1)$ cover from $c_{\rho+2r}$ to
$c_{\rho+2r+1}-1$. In the $(\rho+2r+1)$st column, the rows 
$(r-3,r-1,r)$, $(r-2,r-1,r)$, $(r-1,r-1,r)$, $(r-3,r,r)$, $(r-2,r,r)$
cover from $c_{\rho+2r+1}$ to $c_{\rho+2r+2}-1$, and in the final
column, the rows $(r-2,r,r)$, $(r-1,r,r)$, $(r,r,r)$ will cover from
$c_{\rho+2r+2}$ to $3d$, omitting only $3d-1$.
Applying Proposition \ref{prop:surj-crit}, we conclude the desired 
statement for case (ii).
\end{proof}


\section{The case of cubics}\label{sec:m3}

We conclude with a discussion of the $m=3$ case. Rather than attempting
to prove that it holds for every case of given small $r$, which requires
extensive case-by-case analysis, we will treat what appear to be the
``hardest'' cases for each of $r=3,4,5$, each of which is in the injective 
range, and then conclude by Proposition \ref{prop:injective-extend} that
the Maximal Rank Conjecture holds for all but finitely many cases for
each $r$. The aforementioned ``hardest case'' for each $r$ is somewhat 
parallel to the critical cases for $m=2$ addressed in Proposition 
\ref{prop:m2-critical}; specifically, we take the smallest $g$ such that
all non-injective cases occur in genera strictly smaller than $g$. For
$r=5$, this case happens to be also in the surjective range.
For $r=3$ and $r=4$ these cases are not in the surjective range, although 
the $r=3$ example will imply a case having genus one
greater which is simultaneously in the injective and surjective ranges.

The three examples are as follows.


\begin{ex}\label{ex:m3-r3} Consider the case $r=3$, $g=7$, $d=9$.
Then $\binom{r+3}{3}=20$, and $3d+1-g=21$; we see that this is in 
the injective range. We take the extendable $(g,r,d)$-sequence 
$\vec{\delta}=0,0,1,1,2,2,3$, which gives $T'(\vec{\delta})$ as follows.
\begin{center}
\begin{tabular}{lr|lr|lr|lr|lr|lr|lr}
$0$ & $9$ & $0$ & $9$ & $0$ & $8$ & $1$ & $7$ & $2$ & $6$ & $3$ & $5$ & $4$ & $4$ \\
$1$ & $7$ & $2$ & $6$ & $3$ & $6$ & $3$ & $6$ & $3$ & $5$ & $4$ & $4$ & $5$ & $3$ \\
$2$ & $6$ & $3$ & $5$ & $4$ & $4$ & $5$ & $3$ & $6$ & $3$ & $6$ & $3$ & $6$ & $2$ \\
$3$ & $5$ & $4$ & $4$ & $5$ & $3$ & $6$ & $2$ & $7$ & $1$ & $8$ & $0$ & $9$ & $0$ \\
\end{tabular}
\end{center}
We then get $T(\vec{\delta})$ as follows.
\begin{center}
\begin{tabular}{llr|lr|lr|lr|lr|lr|lr}
 &  &  $23$ & $4$ & $20$ & $7$ & $17$ & $10$ & $14$ & $13$ & $10$ & $17$ & $6$ & $21$ & \\
\hline
$(0,0,0)$ & \cellcolor[gray]{.8} $0$ & \cellcolor[gray]{.8} $27$ & $0$ & $27$ & $0$ & $24$ & $3$ & $21$ & $6$ & $18$ & $9$ & $15$ & $12$ & $12$ \\
$(0,0,1)$ & \cellcolor[gray]{.8} $1$ & \cellcolor[gray]{.8} $25$ & $2$ & $24$ & $3$ & $22$ & $5$ & $20$ & $7$ & $17$ & $10$ & $14$ & $13$ & $11$ \\
$(0,0,2)$ & \cellcolor[gray]{.8} $2$ & \cellcolor[gray]{.8} $24$ & $3$ & $23$ & $4$ & $20$ & $7$ & $17$ & $10$ & $15$ & $12$ & $13$ & $14$ & $10$ \\
$(0,1,1)$ & \cellcolor[gray]{.8} $2$ & \cellcolor[gray]{.8} $23$ & \cellcolor[gray]{.8} $4$ & \cellcolor[gray]{.8} $21$ & $6$ & $20$ & $7$ & $19$ & $8$ & $16$ & $11$ & $13$ & $14$ & $10$ \\
$(0,0,3)$ & \cellcolor[gray]{.8} $3$ & \cellcolor[gray]{.8} $23$ & \cellcolor[gray]{.8} $4$ & \cellcolor[gray]{.8} $22$ & $5$ & $19$ & $8$ & $16$ & $11$ & $13$ & $14$ & $10$ & $17$ & $8$ \\
$(0,1,2)$ & $3$ & $22$ & \cellcolor[gray]{.8} $5$ & \cellcolor[gray]{.8} $20$ & \cellcolor[gray]{.8} $7$ & \cellcolor[gray]{.8} $18$ & $9$ & $16$ & $11$ & $14$ & $13$ & $12$ & $15$ & $9$ \\
$(1,1,1)$ & $3$ & $21$ & $6$ & $18$ & \cellcolor[gray]{.8} $9$ & \cellcolor[gray]{.8} $18$ & $9$ & $18$ & $9$ & $15$ & $12$ & $12$ & $15$ & $9$ \\
$(0,1,3)$ & $4$ & $21$ & $6$ & $19$ & \cellcolor[gray]{.8} $8$ & \cellcolor[gray]{.8} $17$ & \cellcolor[gray]{.8} $10$ & \cellcolor[gray]{.8} $15$ & $12$ & $12$ & $15$ & $9$ & $18$ & $7$ \\
$(0,2,2)$ & $4$ & $21$ & $6$ & $19$ & $8$ & $16$ & $11$ & $13$ & \cellcolor[gray]{.8} $14$ & \cellcolor[gray]{.8} $12$ & $15$ & $11$ & $16$ & $8$ \\
$(1,1,2)$ & $4$ & $20$ & $7$ & $17$ & $10$ & $16$ & \cellcolor[gray]{.8} $11$ & \cellcolor[gray]{.8} $15$ & $12$ & $13$ & $14$ & $11$ & $16$ & $8$ \\
$(0,2,3)$ & $5$ & $20$ & $7$ & $18$ & $9$ & $15$ & $12$ & $12$ & \cellcolor[gray]{.8} $15$ & \cellcolor[gray]{.8} $10$ & \cellcolor[gray]{.8} $17$ & \cellcolor[gray]{.8} $8$ & $19$ & $6$ \\
$(1,1,3)$ & $5$ & $19$ & $8$ & $16$ & $11$ & $15$ & \cellcolor[gray]{.8} $12$ & \cellcolor[gray]{.8} $14$ & \cellcolor[gray]{.8} $13$ & \cellcolor[gray]{.8} $11$ & $16$ & $8$ & $19$ & $6$ \\
$(1,2,2)$ & $5$ & $19$ & $8$ & $16$ & $11$ & $14$ & $13$ & $12$ & \cellcolor[gray]{.8} $15$ & \cellcolor[gray]{.8} $11$ & $16$ & $10$ & $17$ & $7$ \\
$(0,3,3)$ & $6$ & $19$ & $8$ & $17$ & $10$ & $14$ & $13$ & $11$ & $16$ & $8$ & $19$ & $5$ & \cellcolor[gray]{.8} $22$ & \cellcolor[gray]{.8} $4$ \\
$(1,2,3)$ & $6$ & $18$ & $9$ & $15$ & $12$ & $13$ & $14$ & $11$ & $16$ & $9$ & \cellcolor[gray]{.8} $18$ & \cellcolor[gray]{.8} $7$ & $20$ & $5$ \\
$(2,2,2)$ & $6$ & $18$ & $9$ & $15$ & $12$ & $12$ & $15$ & $9$ & $18$ & $9$ & \cellcolor[gray]{.8} $18$ & \cellcolor[gray]{.8} $9$ & $18$ & $6$ \\
$(1,3,3)$ & $7$ & $17$ & $10$ & $14$ & $13$ & $12$ & $15$ & $10$ & $17$ & $7$ & $20$ & $4$ & \cellcolor[gray]{.8} $23$ & \cellcolor[gray]{.8} $3$ \\
$(2,2,3)$ & $7$ & $17$ & $10$ & $14$ & $13$ & $11$ & $16$ & $8$ & $19$ & $7$ & \cellcolor[gray]{.8} $20$ & \cellcolor[gray]{.8} $6$ & \cellcolor[gray]{.8} $21$ & \cellcolor[gray]{.8} $4$ \\
$(2,3,3)$ & $8$ & $16$ & $11$ & $13$ & $14$ & $10$ & $17$ & $7$ & $20$ & $5$ & $22$ & $3$ & \cellcolor[gray]{.8} $24$ & \cellcolor[gray]{.8} $2$ \\
$(3,3,3)$ & $9$ & $15$ & $12$ & $12$ & $15$ & $9$ & $18$ & $6$ & $21$ & $3$ & $24$ & $0$ & \cellcolor[gray]{.8} $27$ & \cellcolor[gray]{.8} $0$ \\
\hline
 &  &  $23$ & $4$ & $20$ & $7$ & $17$ & $10$ & $14$ & $13$ & $10$ & $17$ & $6$ & $21$ & \\
\end{tabular}
\end{center}
The highlighted entries show $T_w(\vec{\delta})$ for $c=(4,7,10,13,17,21)$, which is unimaginative. 
As in earlier examples, we have placed the $c_i$ and $md-c_i$ at the top and bottom of the table to make the erasure procedures clearer.

Now, by applying Lemma \ref{rulesii-v}(a) to the first, third, fourth and seventh columns, 
we can drop rows $(0,0,0)$, $(0,0,1)$, $(0,1,2)$, $(0,1,3)$, $(1,1,1)$, $(1,1,2)$, $(1,1,3)$, $(2,2,3)$, $(0,3,3)$, $(1,3,3)$, $(2,3,3)$ and  $(3,3,3)$. 
Applying Lemma \ref{rulesii-v}(b) to the sixth column, we can also drop rows $(1,2,3)$, $(0,2,3)$, and $(2,2,2)$. 
This leaves only five rows, which can all be dropped using Lemma \ref{rulesii-v}(c) in the second, first and fifth columns.
\end{ex}

\begin{ex}\label{ex:m3-r4} Consider the case $r=4$, $g=16$, $d=17$.
Then $\binom{r+3}{3}=35$, and $3d+1-g=36$, so this is in the injective range, but not the surjective range.
We take the extendable $(g,r,d)$-sequence 
$\vec{\delta}=0,0,0,0,1,1,1,2,2,2,3,3,3$, which gives $T'(\vec{\delta})$ as follows.
\begin{center}
\resizebox{\textwidth}{!}{
\begin{tabular}{lr|lr|lr|lr|lr|lr|lr|lr|lr|lr|lr|lr|lr|lr|lr|lr}
$0$ & $17$ & $0$ & $17$ & $0$ & $17$ & $0$ & $17$ & $0$ & $16$ & $1$ & $15$ & $2$ & $14$ & $3$ & $13$ & $4$ & $12$ & $5$ & $11$ & $6$ & $10$ & $7$ & $9$ & $8$ & $8$ & $9$ & $7$ & $10$ & $6$ & $11$ & $5$ \\
$1$ & $15$ & $2$ & $14$ & $3$ & $13$ & $4$ & $12$ & $5$ & $12$ & $5$ & $12$ & $5$ & $12$ & $5$ & $11$ & $6$ & $10$ & $7$ & $9$ & $8$ & $8$ & $9$ & $7$ & $10$ & $6$ & $11$ & $5$ & $12$ & $4$ & $13$ & $3$ \\
$2$ & $14$ & $3$ & $13$ & $4$ & $12$ & $5$ & $11$ & $6$ & $10$ & $7$ & $9$ & $8$ & $8$ & $9$ & $8$ & $9$ & $8$ & $9$ & $8$ & $9$ & $7$ & $10$ & $6$ & $11$ & $5$ & $12$ & $4$ & $13$ & $3$ & $14$ & $2$ \\
$3$ & $13$ & $4$ & $12$ & $5$ & $11$ & $6$ & $10$ & $7$ & $9$ & $8$ & $8$ & $9$ & $7$ & $10$ & $6$ & $11$ & $5$ & $12$ & $4$ & $13$ & $4$ & $13$ & $4$ & $13$ & $4$ & $13$ & $3$ & $14$ & $2$ & $15$ & $1$ \\
$4$ & $12$ & $5$ & $11$ & $6$ & $10$ & $7$ & $9$ & $8$ & $8$ & $9$ & $7$ & $10$ & $6$ & $11$ & $5$ & $12$ & $4$ & $13$ & $3$ & $14$ & $2$ & $15$ & $1$ & $16$ & $0$ & $17$ & $0$ & $17$ & $0$ & $17$ & $0$ \\
\end{tabular}
}
\end{center}
We then get $T(\vec{\delta})$ as follows.
\begin{center}
\resizebox{\textwidth}{!}{
\begin{tabular}{llr|lr|lr|lr|lr|lr|lr|lr|lr|lr|lr|lr|lr|lr|lr|lr}
 &   &  $48$ & $3$ & $46$ & $5$ & $44$ & $7$ & $39$ & $12$ & $35$ & $16$ & $32$ & $19$ & $29$ & $22$ & $27$ & $24$ & $23$ & $28$ & $20$ & $31$ & $16$ & $35$ & $14$ & $37$ & $10$ & $41$ & $7$ & $44$ & $4$ & $47$ &  \\
\hline
$(0,0,0)$ & \cellcolor[gray]{.8} $0$ & \cellcolor[gray]{.8} $51$ & $0$ & $51$ & $0$ & $51$ & $0$ & $51$ & $0$ & $48$ & $3$ & $45$ & $6$ & $42$ & $9$ & $39$ & $12$ & $36$ & $15$ & $33$ & $18$ & $30$ & $21$ & $27$ & $24$ & $24$ & $27$ & $21$ & $30$ & $18$ & $33$ & $15$ \\
$(0,0,1)$ & \cellcolor[gray]{.8} $1$ & \cellcolor[gray]{.8} $49$ & $2$ & $48$ & $3$ & $47$ & $4$ & $46$ & $5$ & $44$ & $7$ & $42$ & $9$ & $40$ & $11$ & $37$ & $14$ & $34$ & $17$ & $31$ & $20$ & $28$ & $23$ & $25$ & $26$ & $22$ & $29$ & $19$ & $32$ & $16$ & $35$ & $13$ \\
$(0,0,2)$ & \cellcolor[gray]{.8} $2$ & \cellcolor[gray]{.8} $48$ & \cellcolor[gray]{.8} $3$ & \cellcolor[gray]{.8} $47$ & $4$ & $46$ & $5$ & $45$ & $6$ & $42$ & $9$ & $39$ & $12$ & $36$ & $15$ & $34$ & $17$ & $32$ & $19$ & $30$ & $21$ & $27$ & $24$ & $24$ & $27$ & $21$ & $30$ & $18$ & $33$ & $15$ & $36$ & $12$ \\
$(0,1,1)$ & $2$ & $47$ & $4$ & $45$ & $6$ & $43$ & \cellcolor[gray]{.8} $8$ & \cellcolor[gray]{.8} $41$ & $10$ & $40$ & $11$ & $39$ & $12$ & $38$ & $13$ & $35$ & $16$ & $32$ & $19$ & $29$ & $22$ & $26$ & $25$ & $23$ & $28$ & $20$ & $31$ & $17$ & $34$ & $14$ & $37$ & $11$ \\
$(0,0,3)$ & $3$ & $47$ & \cellcolor[gray]{.8} $4$ & \cellcolor[gray]{.8} $46$ & \cellcolor[gray]{.8} $5$ & \cellcolor[gray]{.8} $45$ & $6$ & $44$ & $7$ & $41$ & $10$ & $38$ & $13$ & $35$ & $16$ & $32$ & $19$ & $29$ & $22$ & $26$ & $25$ & $24$ & $27$ & $22$ & $29$ & $20$ & $31$ & $17$ & $34$ & $14$ & $37$ & $11$ \\
$(0,1,2)$ & $3$ & $46$ & $5$ & $44$ & $7$ & $42$ & \cellcolor[gray]{.8} $9$ & \cellcolor[gray]{.8} $40$ & $11$ & $38$ & $13$ & $36$ & $15$ & $34$ & $17$ & $32$ & $19$ & $30$ & $21$ & $28$ & $23$ & $25$ & $26$ & $22$ & $29$ & $19$ & $32$ & $16$ & $35$ & $13$ & $38$ & $10$ \\
$(1,1,1)$ & $3$ & $45$ & $6$ & $42$ & $9$ & $39$ & $12$ & $36$ & \cellcolor[gray]{.8} $15$ & \cellcolor[gray]{.8} $36$ & $15$ & $36$ & $15$ & $36$ & $15$ & $33$ & $18$ & $30$ & $21$ & $27$ & $24$ & $24$ & $27$ & $21$ & $30$ & $18$ & $33$ & $15$ & $36$ & $12$ & $39$ & $9$ \\
$(0,0,4)$ & $4$ & $46$ & $5$ & $45$ & \cellcolor[gray]{.8} $6$ & \cellcolor[gray]{.8} $44$ & \cellcolor[gray]{.8} $7$ & \cellcolor[gray]{.8} $43$ & $8$ & $40$ & $11$ & $37$ & $14$ & $34$ & $17$ & $31$ & $20$ & $28$ & $23$ & $25$ & $26$ & $22$ & $29$ & $19$ & $32$ & $16$ & $35$ & $14$ & $37$ & $12$ & $39$ & $10$ \\
$(0,1,3)$ & $4$ & $45$ & $6$ & $43$ & $8$ & $41$ & \cellcolor[gray]{.8} $10$ & \cellcolor[gray]{.8} $39$ & \cellcolor[gray]{.8} $12$ & \cellcolor[gray]{.8} $37$ & $14$ & $35$ & $16$ & $33$ & $18$ & $30$ & $21$ & $27$ & $24$ & $24$ & $27$ & $22$ & $29$ & $20$ & $31$ & $18$ & $33$ & $15$ & $36$ & $12$ & $39$ & $9$ \\
$(0,2,2)$ & $4$ & $45$ & $6$ & $43$ & $8$ & $41$ & \cellcolor[gray]{.8} $10$ & \cellcolor[gray]{.8} $39$ & \cellcolor[gray]{.8} $12$ & \cellcolor[gray]{.8} $36$ & $15$ & $33$ & $18$ & $30$ & $21$ & $29$ & $22$ & $28$ & $23$ & $27$ & $24$ & $24$ & $27$ & $21$ & $30$ & $18$ & $33$ & $15$ & $36$ & $12$ & $39$ & $9$ \\
$(1,1,2)$ & $4$ & $44$ & $7$ & $41$ & $10$ & $38$ & $13$ & $35$ & $16$ & $34$ & \cellcolor[gray]{.8} $17$ & \cellcolor[gray]{.8} $33$ & $18$ & $32$ & $19$ & $30$ & $21$ & $28$ & $23$ & $26$ & $25$ & $23$ & $28$ & $20$ & $31$ & $17$ & $34$ & $14$ & $37$ & $11$ & $40$ & $8$ \\
$(0,1,4)$ & $5$ & $44$ & $7$ & $42$ & $9$ & $40$ & $11$ & $38$ & \cellcolor[gray]{.8} $13$ & \cellcolor[gray]{.8} $36$ & $15$ & $34$ & $17$ & $32$ & $19$ & $29$ & $22$ & $26$ & $25$ & $23$ & $28$ & $20$ & $31$ & $17$ & $34$ & $14$ & $37$ & $12$ & $39$ & $10$ & $41$ & $8$ \\
$(0,2,3)$ & $5$ & $44$ & $7$ & $42$ & $9$ & $40$ & $11$ & $38$ & \cellcolor[gray]{.8} $13$ & \cellcolor[gray]{.8} $35$ & \cellcolor[gray]{.8} $16$ & \cellcolor[gray]{.8} $32$ & \cellcolor[gray]{.8} $19$ & \cellcolor[gray]{.8} $29$ & \cellcolor[gray]{.8} $22$ & \cellcolor[gray]{.8} $27$ & \cellcolor[gray]{.8} $24$ & \cellcolor[gray]{.8} $25$ & $26$ & $23$ & $28$ & $21$ & $30$ & $19$ & $32$ & $17$ & $34$ & $14$ & $37$ & $11$ & $40$ & $8$ \\
$(1,1,3)$ & $5$ & $43$ & $8$ & $40$ & $11$ & $37$ & $14$ & $34$ & $17$ & $33$ & \cellcolor[gray]{.8} $18$ & \cellcolor[gray]{.8} $32$ & \cellcolor[gray]{.8} $19$ & \cellcolor[gray]{.8} $31$ & $20$ & $28$ & $23$ & $25$ & $26$ & $22$ & $29$ & $20$ & $31$ & $18$ & $33$ & $16$ & $35$ & $13$ & $38$ & $10$ & $41$ & $7$ \\
$(1,2,2)$ & $5$ & $43$ & $8$ & $40$ & $11$ & $37$ & $14$ & $34$ & $17$ & $32$ & $19$ & $30$ & $21$ & $28$ & \cellcolor[gray]{.8} $23$ & \cellcolor[gray]{.8} $27$ & \cellcolor[gray]{.8} $24$ & \cellcolor[gray]{.8} $26$ & $25$ & $25$ & $26$ & $22$ & $29$ & $19$ & $32$ & $16$ & $35$ & $13$ & $38$ & $10$ & $41$ & $7$ \\
$(0,2,4)$ & $6$ & $43$ & $8$ & $41$ & $10$ & $39$ & $12$ & $37$ & $14$ & $34$ & $17$ & $31$ & $20$ & $28$ & $23$ & $26$ & \cellcolor[gray]{.8} $25$ & \cellcolor[gray]{.8} $24$ & $27$ & $22$ & $29$ & $19$ & $32$ & $16$ & $35$ & $13$ & $38$ & $11$ & $40$ & $9$ & $42$ & $7$ \\
$(0,3,3)$ & $6$ & $43$ & $8$ & $41$ & $10$ & $39$ & $12$ & $37$ & $14$ & $34$ & $17$ & $31$ & $20$ & $28$ & $23$ & $25$ & $26$ & $22$ & $29$ & $19$ & \cellcolor[gray]{.8} $32$ & \cellcolor[gray]{.8} $18$ & $33$ & $17$ & $34$ & $16$ & $35$ & $13$ & $38$ & $10$ & $41$ & $7$ \\
$(1,1,4)$ & $6$ & $42$ & $9$ & $39$ & $12$ & $36$ & $15$ & $33$ & $18$ & $32$ & $19$ & $31$ & \cellcolor[gray]{.8} $20$ & \cellcolor[gray]{.8} $30$ & $21$ & $27$ & $24$ & $24$ & $27$ & $21$ & $30$ & $18$ & $33$ & $15$ & $36$ & $12$ & $39$ & $10$ & $41$ & $8$ & $43$ & $6$ \\
$(1,2,3)$ & $6$ & $42$ & $9$ & $39$ & $12$ & $36$ & $15$ & $33$ & $18$ & $31$ & $20$ & $29$ & $22$ & $27$ & $24$ & $25$ & \cellcolor[gray]{.8} $26$ & \cellcolor[gray]{.8} $23$ & \cellcolor[gray]{.8} $28$ & \cellcolor[gray]{.8} $21$ & $30$ & $19$ & $32$ & $17$ & $34$ & $15$ & $36$ & $12$ & $39$ & $9$ & $42$ & $6$ \\
$(2,2,2)$ & $6$ & $42$ & $9$ & $39$ & $12$ & $36$ & $15$ & $33$ & $18$ & $30$ & $21$ & $27$ & $24$ & $24$ & $27$ & $24$ & \cellcolor[gray]{.8} $27$ & \cellcolor[gray]{.8} $24$ & $27$ & $24$ & $27$ & $21$ & $30$ & $18$ & $33$ & $15$ & $36$ & $12$ & $39$ & $9$ & $42$ & $6$ \\
$(0,3,4)$ & $7$ & $42$ & $9$ & $40$ & $11$ & $38$ & $13$ & $36$ & $15$ & $33$ & $18$ & $30$ & $21$ & $27$ & $24$ & $24$ & $27$ & $21$ & $30$ & $18$ & \cellcolor[gray]{.8} $33$ & \cellcolor[gray]{.8} $16$ & \cellcolor[gray]{.8} $35$ & \cellcolor[gray]{.8} $14$ & \cellcolor[gray]{.8} $37$ & \cellcolor[gray]{.8} $12$ & $39$ & $10$ & $41$ & $8$ & $43$ & $6$ \\
$(1,2,4)$ & $7$ & $41$ & $10$ & $38$ & $13$ & $35$ & $16$ & $32$ & $19$ & $30$ & $21$ & $28$ & $23$ & $26$ & $25$ & $24$ & $27$ & $22$ & \cellcolor[gray]{.8} $29$ & \cellcolor[gray]{.8} $20$ & \cellcolor[gray]{.8} $31$ & \cellcolor[gray]{.8} $17$ & $34$ & $14$ & $37$ & $11$ & $40$ & $9$ & $42$ & $7$ & $44$ & $5$ \\
$(1,3,3)$ & $7$ & $41$ & $10$ & $38$ & $13$ & $35$ & $16$ & $32$ & $19$ & $30$ & $21$ & $28$ & $23$ & $26$ & $25$ & $23$ & $28$ & $20$ & $31$ & $17$ & \cellcolor[gray]{.8} $34$ & \cellcolor[gray]{.8} $16$ & \cellcolor[gray]{.8} $35$ & \cellcolor[gray]{.8} $15$ & $36$ & $14$ & $37$ & $11$ & $40$ & $8$ & $43$ & $5$ \\
$(2,2,3)$ & $7$ & $41$ & $10$ & $38$ & $13$ & $35$ & $16$ & $32$ & $19$ & $29$ & $22$ & $26$ & $25$ & $23$ & $28$ & $22$ & $29$ & $21$ & \cellcolor[gray]{.8} $30$ & \cellcolor[gray]{.8} $20$ & \cellcolor[gray]{.8} $31$ & \cellcolor[gray]{.8} $18$ & $33$ & $16$ & $35$ & $14$ & $37$ & $11$ & $40$ & $8$ & $43$ & $5$ \\
$(0,4,4)$ & $8$ & $41$ & $10$ & $39$ & $12$ & $37$ & $14$ & $35$ & $16$ & $32$ & $19$ & $29$ & $22$ & $26$ & $25$ & $23$ & $28$ & $20$ & $31$ & $17$ & $34$ & $14$ & $37$ & $11$ & $40$ & $8$ & \cellcolor[gray]{.8} $43$ & \cellcolor[gray]{.8} $7$ & \cellcolor[gray]{.8} $44$ & \cellcolor[gray]{.8} $6$ & $45$ & $5$ \\
$(1,3,4)$ & $8$ & $40$ & $11$ & $37$ & $14$ & $34$ & $17$ & $31$ & $20$ & $29$ & $22$ & $27$ & $24$ & $25$ & $26$ & $22$ & $29$ & $19$ & $32$ & $16$ & $35$ & $14$ & $37$ & $12$ & \cellcolor[gray]{.8} $39$ & \cellcolor[gray]{.8} $10$ & \cellcolor[gray]{.8} $41$ & \cellcolor[gray]{.8} $8$ & $43$ & $6$ & $45$ & $4$ \\
$(2,2,4)$ & $8$ & $40$ & $11$ & $37$ & $14$ & $34$ & $17$ & $31$ & $20$ & $28$ & $23$ & $25$ & $26$ & $22$ & $29$ & $21$ & $30$ & $20$ & $31$ & $19$ & $32$ & $16$ & $35$ & $13$ & \cellcolor[gray]{.8} $38$ & \cellcolor[gray]{.8} $10$ & \cellcolor[gray]{.8} $41$ & \cellcolor[gray]{.8} $8$ & $43$ & $6$ & $45$ & $4$ \\
$(2,3,3)$ & $8$ & $40$ & $11$ & $37$ & $14$ & $34$ & $17$ & $31$ & $20$ & $28$ & $23$ & $25$ & $26$ & $22$ & $29$ & $20$ & $31$ & $18$ & $33$ & $16$ & $35$ & $15$ & \cellcolor[gray]{.8} $36$ & \cellcolor[gray]{.8} $14$ & \cellcolor[gray]{.8} $37$ & \cellcolor[gray]{.8} $13$ & $38$ & $10$ & $41$ & $7$ & $44$ & $4$ \\
$(1,4,4)$ & $9$ & $39$ & $12$ & $36$ & $15$ & $33$ & $18$ & $30$ & $21$ & $28$ & $23$ & $26$ & $25$ & $24$ & $27$ & $21$ & $30$ & $18$ & $33$ & $15$ & $36$ & $12$ & $39$ & $9$ & $42$ & $6$ & $45$ & $5$ & \cellcolor[gray]{.8} $46$ & \cellcolor[gray]{.8} $4$ & \cellcolor[gray]{.8} $47$ & \cellcolor[gray]{.8} $3$ \\
$(2,3,4)$ & $9$ & $39$ & $12$ & $36$ & $15$ & $33$ & $18$ & $30$ & $21$ & $27$ & $24$ & $24$ & $27$ & $21$ & $30$ & $19$ & $32$ & $17$ & $34$ & $15$ & $36$ & $13$ & $38$ & $11$ & $40$ & $9$ & \cellcolor[gray]{.8} $42$ & \cellcolor[gray]{.8} $7$ & \cellcolor[gray]{.8} $44$ & \cellcolor[gray]{.8} $5$ & $46$ & $3$ \\
$(3,3,3)$ & $9$ & $39$ & $12$ & $36$ & $15$ & $33$ & $18$ & $30$ & $21$ & $27$ & $24$ & $24$ & $27$ & $21$ & $30$ & $18$ & $33$ & $15$ & $36$ & $12$ & $39$ & $12$ & $39$ & $12$ & \cellcolor[gray]{.8} $39$ & \cellcolor[gray]{.8} $12$ & $39$ & $9$ & $42$ & $6$ & $45$ & $3$ \\
$(2,4,4)$ & $10$ & $38$ & $13$ & $35$ & $16$ & $32$ & $19$ & $29$ & $22$ & $26$ & $25$ & $23$ & $28$ & $20$ & $31$ & $18$ & $33$ & $16$ & $35$ & $14$ & $37$ & $11$ & $40$ & $8$ & $43$ & $5$ & $46$ & $4$ & $47$ & $3$ & \cellcolor[gray]{.8} $48$ & \cellcolor[gray]{.8} $2$ \\
$(3,3,4)$ & $10$ & $38$ & $13$ & $35$ & $16$ & $32$ & $19$ & $29$ & $22$ & $26$ & $25$ & $23$ & $28$ & $20$ & $31$ & $17$ & $34$ & $14$ & $37$ & $11$ & $40$ & $10$ & $41$ & $9$ & $42$ & $8$ & $43$ & $6$ & \cellcolor[gray]{.8} $45$ & \cellcolor[gray]{.8} $4$ & $47$ & $2$ \\
$(3,4,4)$ & $11$ & $37$ & $14$ & $34$ & $17$ & $31$ & $20$ & $28$ & $23$ & $25$ & $26$ & $22$ & $29$ & $19$ & $32$ & $16$ & $35$ & $13$ & $38$ & $10$ & $41$ & $8$ & $43$ & $6$ & $45$ & $4$ & $47$ & $3$ & $48$ & $2$ & \cellcolor[gray]{.8} $49$ & \cellcolor[gray]{.8} $1$ \\
$(4,4,4)$ & $12$ & $36$ & $15$ & $33$ & $18$ & $30$ & $21$ & $27$ & $24$ & $24$ & $27$ & $21$ & $30$ & $18$ & $33$ & $15$ & $36$ & $12$ & $39$ & $9$ & $42$ & $6$ & $45$ & $3$ & $48$ & $0$ & $51$ & $0$ & $51$ & $0$ & \cellcolor[gray]{.8} $51$ & \cellcolor[gray]{.8} $0$ \\
\hline
 &   &  $48$ & $3$ & $46$ & $5$ & $44$ & $7$ & $39$ & $12$ & $35$ & $16$ & $32$ & $19$ & $29$ & $22$ & $27$ & $24$ & $23$ & $28$ & $20$ & $31$ & $16$ & $35$ & $14$ & $37$ & $10$ & $41$ & $7$ & $44$ & $4$ & $47$ &  \\
\end{tabular}
}
\end{center}
The highlighted entries show $T_w(\vec{\delta})$ for
$$c=(3,5,7,12,16,19,22,24,28,31,35,37,41,44,47).$$ 
Note that the $w(c)$ is not unimaginative, although one can check that it is still steady with respect to $T(\vec{\delta})$.

We may use Lemma \ref{rulesii-v}(a) and (b) to drop all rows in the first, second, third, sixth, seventh, eighth, $10$th and $16$th columns.
 We can also drop the rows in the fourth column with Lemma \ref{rulesii-v}(a) and (c).
  This leaves two rows in each of the fifth and ninth columns, which can thus be dropped with Lemma \ref{rulesii-v}(c).
   The remaining rows in the $11$th column can then be dropped with Lemma \ref{rulesii-v} (a), 
   and the remaining rows in the $15$th column can then be dropped with Lemma \ref{rulesii-v}(b).
    This leaves only one row in the $12$th column and two rows in the $14$th column, so these can be dropped with Lemma \ref{rulesii-v}(c).
     Finally, this leaves only one row in the $13$th column, so we are done. 
\end{ex}

\begin{ex}\label{ex:m3-r5} Consider the case $r=5$, $g=26$, $d=27$.
Then $\binom{r+3}{3}=56$, and $3d+1-g=56$, so this is in 
both the injective and surjective ranges.
We take the extendable $(g,r,d)$-sequence 
$\vec{\delta}=0,0,0,0,0,1,1,1,1,1,2,2,2,2,3,3,3,3,4,4,4,4,5,5,5,5$, which gives 
$T'(\vec{\delta})$ as follows.
\begin{center}
\resizebox{\textwidth}{!}{
\begin{tabular}{lr|lr|lr|lr|lr|lr|lr|lr|lr|lr|lr|lr|lr|lr|lr|lr|lr|lr|lr|lr|lr|lr|lr|lr|lr|lr}
$0$ & $27$ & $0$ & $27$ & $0$ & $27$ & $0$ & $27$ & $0$ & $27$ & $0$ & $26$ & $1$ & $25$ & $2$ & $24$ & $3$ & $23$ & $4$ & $22$ & $5$ & $21$ & $6$ & $20$ & $7$ & $19$ & $8$ & $18$ & $9$ & $17$ & $10$ & $16$ & $11$ & $15$ & $12$ & $14$ & $13$ & $13$ & $14$ & $12$ & $15$ & $11$ & $16$ & $10$ & $17$ & $9$ & $18$ & $8$ & $19$ & $7$ & $20$ & $6$ \\
$1$ & $25$ & $2$ & $24$ & $3$ & $23$ & $4$ & $22$ & $5$ & $21$ & $6$ & $21$ & $6$ & $21$ & $6$ & $21$ & $6$ & $21$ & $6$ & $21$ & $6$ & $20$ & $7$ & $19$ & $8$ & $18$ & $9$ & $17$ & $10$ & $16$ & $11$ & $15$ & $12$ & $14$ & $13$ & $13$ & $14$ & $12$ & $15$ & $11$ & $16$ & $10$ & $17$ & $9$ & $18$ & $8$ & $19$ & $7$ & $20$ & $6$ & $21$ & $5$ \\
$2$ & $24$ & $3$ & $23$ & $4$ & $22$ & $5$ & $21$ & $6$ & $20$ & $7$ & $19$ & $8$ & $18$ & $9$ & $17$ & $10$ & $16$ & $11$ & $15$ & $12$ & $15$ & $12$ & $15$ & $12$ & $15$ & $12$ & $15$ & $12$ & $14$ & $13$ & $13$ & $14$ & $12$ & $15$ & $11$ & $16$ & $10$ & $17$ & $9$ & $18$ & $8$ & $19$ & $7$ & $20$ & $6$ & $21$ & $5$ & $22$ & $4$ & $23$ & $3$ \\
$3$ & $23$ & $4$ & $22$ & $5$ & $21$ & $6$ & $20$ & $7$ & $19$ & $8$ & $18$ & $9$ & $17$ & $10$ & $16$ & $11$ & $15$ & $12$ & $14$ & $13$ & $13$ & $14$ & $12$ & $15$ & $11$ & $16$ & $10$ & $17$ & $10$ & $17$ & $10$ & $17$ & $10$ & $17$ & $10$ & $17$ & $9$ & $18$ & $8$ & $19$ & $7$ & $20$ & $6$ & $21$ & $5$ & $22$ & $4$ & $23$ & $3$ & $24$ & $2$ \\
$4$ & $22$ & $5$ & $21$ & $6$ & $20$ & $7$ & $19$ & $8$ & $18$ & $9$ & $17$ & $10$ & $16$ & $11$ & $15$ & $12$ & $14$ & $13$ & $13$ & $14$ & $12$ & $15$ & $11$ & $16$ & $10$ & $17$ & $9$ & $18$ & $8$ & $19$ & $7$ & $20$ & $6$ & $21$ & $5$ & $22$ & $5$ & $22$ & $5$ & $22$ & $5$ & $22$ & $5$ & $22$ & $4$ & $23$ & $3$ & $24$ & $2$ & $25$ & $1$ \\
$5$ & $21$ & $6$ & $20$ & $7$ & $19$ & $8$ & $18$ & $9$ & $17$ & $10$ & $16$ & $11$ & $15$ & $12$ & $14$ & $13$ & $13$ & $14$ & $12$ & $15$ & $11$ & $16$ & $10$ & $17$ & $9$ & $18$ & $8$ & $19$ & $7$ & $20$ & $6$ & $21$ & $5$ & $22$ & $4$ & $23$ & $3$ & $24$ & $2$ & $25$ & $1$ & $26$ & $0$ & $27$ & $0$ & $27$ & $0$ & $27$ & $0$ & $27$ & $0$ \\
\end{tabular}
}
\end{center}
We then get $T(\vec{\delta})$ as follows.
\begin{center}
\resizebox{\textwidth}{!}{
\begin{tabular}{llr|lr|lr|lr|lr|lr|lr|lr|lr|lr|lr|lr|lr|lr|lr|lr|lr|lr|lr|lr|lr|lr|lr|lr|lr|lr}
 &   &  $78$ & $3$ & $75$ & $6$ & $73$ & $8$ & $70$ & $11$ & $66$ & $15$ & $63$ & $18$ & $60$ & $21$ & $57$ & $24$ & $53$ & $28$ & $50$ & $31$ & $47$ & $34$ & $44$ & $37$ & $40$ & $41$ & $37$ & $44$ & $34$ & $47$ & $31$ & $50$ & $28$ & $53$ & $25$ & $56$ & $21$ & $60$ & $18$ & $63$ & $15$ & $66$ & $11$ & $70$ & $8$ & $73$ & $5$ & $76$ & $3$ & $78$ &  \\
\hline
$(0,0,0)$ & \cellcolor[gray]{.8} $0$ & \cellcolor[gray]{.8} $81$ & $0$ & $81$ & $0$ & $81$ & $0$ & $81$ & $0$ & $81$ & $0$ & $78$ & $3$ & $75$ & $6$ & $72$ & $9$ & $69$ & $12$ & $66$ & $15$ & $63$ & $18$ & $60$ & $21$ & $57$ & $24$ & $54$ & $27$ & $51$ & $30$ & $48$ & $33$ & $45$ & $36$ & $42$ & $39$ & $39$ & $42$ & $36$ & $45$ & $33$ & $48$ & $30$ & $51$ & $27$ & $54$ & $24$ & $57$ & $21$ & $60$ & $18$ \\
$(0,0,1)$ & \cellcolor[gray]{.8} $1$ & \cellcolor[gray]{.8} $79$ & $2$ & $78$ & $3$ & $77$ & $4$ & $76$ & $5$ & $75$ & $6$ & $73$ & $8$ & $71$ & $10$ & $69$ & $12$ & $67$ & $14$ & $65$ & $16$ & $62$ & $19$ & $59$ & $22$ & $56$ & $25$ & $53$ & $28$ & $50$ & $31$ & $47$ & $34$ & $44$ & $37$ & $41$ & $40$ & $38$ & $43$ & $35$ & $46$ & $32$ & $49$ & $29$ & $52$ & $26$ & $55$ & $23$ & $58$ & $20$ & $61$ & $17$ \\
$(0,0,2)$ & \cellcolor[gray]{.8} $2$ & \cellcolor[gray]{.8} $78$ & \cellcolor[gray]{.8} $3$ & \cellcolor[gray]{.8} $77$ & $4$ & $76$ & $5$ & $75$ & $6$ & $74$ & $7$ & $71$ & $10$ & $68$ & $13$ & $65$ & $16$ & $62$ & $19$ & $59$ & $22$ & $57$ & $24$ & $55$ & $26$ & $53$ & $28$ & $51$ & $30$ & $48$ & $33$ & $45$ & $36$ & $42$ & $39$ & $39$ & $42$ & $36$ & $45$ & $33$ & $48$ & $30$ & $51$ & $27$ & $54$ & $24$ & $57$ & $21$ & $60$ & $18$ & $63$ & $15$ \\
$(0,1,1)$ & $2$ & $77$ & \cellcolor[gray]{.8} $4$ & \cellcolor[gray]{.8} $75$ & \cellcolor[gray]{.8} $6$ & \cellcolor[gray]{.8} $73$ & \cellcolor[gray]{.8} $8$ & \cellcolor[gray]{.8} $71$ & $10$ & $69$ & $12$ & $68$ & $13$ & $67$ & $14$ & $66$ & $15$ & $65$ & $16$ & $64$ & $17$ & $61$ & $20$ & $58$ & $23$ & $55$ & $26$ & $52$ & $29$ & $49$ & $32$ & $46$ & $35$ & $43$ & $38$ & $40$ & $41$ & $37$ & $44$ & $34$ & $47$ & $31$ & $50$ & $28$ & $53$ & $25$ & $56$ & $22$ & $59$ & $19$ & $62$ & $16$ \\
$(0,0,3)$ & $3$ & $77$ & \cellcolor[gray]{.8} $4$ & \cellcolor[gray]{.8} $76$ & $5$ & $75$ & $6$ & $74$ & $7$ & $73$ & $8$ & $70$ & $11$ & $67$ & $14$ & $64$ & $17$ & $61$ & $20$ & $58$ & $23$ & $55$ & $26$ & $52$ & $29$ & $49$ & $32$ & $46$ & $35$ & $44$ & $37$ & $42$ & $39$ & $40$ & $41$ & $38$ & $43$ & $35$ & $46$ & $32$ & $49$ & $29$ & $52$ & $26$ & $55$ & $23$ & $58$ & $20$ & $61$ & $17$ & $64$ & $14$ \\
$(0,1,2)$ & $3$ & $76$ & $5$ & $74$ & $7$ & $72$ & \cellcolor[gray]{.8} $9$ & \cellcolor[gray]{.8} $70$ & \cellcolor[gray]{.8} $11$ & \cellcolor[gray]{.8} $68$ & $13$ & $66$ & $15$ & $64$ & $17$ & $62$ & $19$ & $60$ & $21$ & $58$ & $23$ & $56$ & $25$ & $54$ & $27$ & $52$ & $29$ & $50$ & $31$ & $47$ & $34$ & $44$ & $37$ & $41$ & $40$ & $38$ & $43$ & $35$ & $46$ & $32$ & $49$ & $29$ & $52$ & $26$ & $55$ & $23$ & $58$ & $20$ & $61$ & $17$ & $64$ & $14$ \\
$(1,1,1)$ & $3$ & $75$ & $6$ & $72$ & $9$ & $69$ & $12$ & $66$ & $15$ & $63$ & \cellcolor[gray]{.8} $18$ & \cellcolor[gray]{.8} $63$ & \cellcolor[gray]{.8} $18$ & \cellcolor[gray]{.8} $63$ & $18$ & $63$ & $18$ & $63$ & $18$ & $63$ & $18$ & $60$ & $21$ & $57$ & $24$ & $54$ & $27$ & $51$ & $30$ & $48$ & $33$ & $45$ & $36$ & $42$ & $39$ & $39$ & $42$ & $36$ & $45$ & $33$ & $48$ & $30$ & $51$ & $27$ & $54$ & $24$ & $57$ & $21$ & $60$ & $18$ & $63$ & $15$ \\
$(0,0,4)$ & $4$ & $76$ & \cellcolor[gray]{.8} $5$ & \cellcolor[gray]{.8} $75$ & \cellcolor[gray]{.8} $6$ & \cellcolor[gray]{.8} $74$ & $7$ & $73$ & $8$ & $72$ & $9$ & $69$ & $12$ & $66$ & $15$ & $63$ & $18$ & $60$ & $21$ & $57$ & $24$ & $54$ & $27$ & $51$ & $30$ & $48$ & $33$ & $45$ & $36$ & $42$ & $39$ & $39$ & $42$ & $36$ & $45$ & $33$ & $48$ & $31$ & $50$ & $29$ & $52$ & $27$ & $54$ & $25$ & $56$ & $22$ & $59$ & $19$ & $62$ & $16$ & $65$ & $13$ \\
$(0,1,3)$ & $4$ & $75$ & $6$ & $73$ & $8$ & $71$ & $10$ & $69$ & \cellcolor[gray]{.8} $12$ & \cellcolor[gray]{.8} $67$ & $14$ & $65$ & $16$ & $63$ & $18$ & $61$ & $20$ & $59$ & $22$ & $57$ & $24$ & $54$ & $27$ & $51$ & $30$ & $48$ & $33$ & $45$ & $36$ & $43$ & $38$ & $41$ & $40$ & $39$ & $42$ & $37$ & $44$ & $34$ & $47$ & $31$ & $50$ & $28$ & $53$ & $25$ & $56$ & $22$ & $59$ & $19$ & $62$ & $16$ & $65$ & $13$ \\
$(0,2,2)$ & $4$ & $75$ & $6$ & $73$ & $8$ & $71$ & $10$ & $69$ & \cellcolor[gray]{.8} $12$ & \cellcolor[gray]{.8} $67$ & $14$ & $64$ & $17$ & $61$ & $20$ & $58$ & $23$ & $55$ & $26$ & $52$ & $29$ & $51$ & $30$ & $50$ & $31$ & $49$ & $32$ & $48$ & $33$ & $45$ & $36$ & $42$ & $39$ & $39$ & $42$ & $36$ & $45$ & $33$ & $48$ & $30$ & $51$ & $27$ & $54$ & $24$ & $57$ & $21$ & $60$ & $18$ & $63$ & $15$ & $66$ & $12$ \\
$(1,1,2)$ & $4$ & $74$ & $7$ & $71$ & $10$ & $68$ & $13$ & $65$ & $16$ & $62$ & $19$ & $61$ & \cellcolor[gray]{.8} $20$ & \cellcolor[gray]{.8} $60$ & \cellcolor[gray]{.8} $21$ & \cellcolor[gray]{.8} $59$ & $22$ & $58$ & $23$ & $57$ & $24$ & $55$ & $26$ & $53$ & $28$ & $51$ & $30$ & $49$ & $32$ & $46$ & $35$ & $43$ & $38$ & $40$ & $41$ & $37$ & $44$ & $34$ & $47$ & $31$ & $50$ & $28$ & $53$ & $25$ & $56$ & $22$ & $59$ & $19$ & $62$ & $16$ & $65$ & $13$ \\
$(0,0,5)$ & $5$ & $75$ & $6$ & $74$ & \cellcolor[gray]{.8} $7$ & \cellcolor[gray]{.8} $73$ & \cellcolor[gray]{.8} $8$ & \cellcolor[gray]{.8} $72$ & $9$ & $71$ & $10$ & $68$ & $13$ & $65$ & $16$ & $62$ & $19$ & $59$ & $22$ & $56$ & $25$ & $53$ & $28$ & $50$ & $31$ & $47$ & $34$ & $44$ & $37$ & $41$ & $40$ & $38$ & $43$ & $35$ & $46$ & $32$ & $49$ & $29$ & $52$ & $26$ & $55$ & $23$ & $58$ & $20$ & $61$ & $18$ & $63$ & $16$ & $65$ & $14$ & $67$ & $12$ \\
$(0,1,4)$ & $5$ & $74$ & $7$ & $72$ & $9$ & $70$ & $11$ & $68$ & \cellcolor[gray]{.8} $13$ & \cellcolor[gray]{.8} $66$ & \cellcolor[gray]{.8} $15$ & \cellcolor[gray]{.8} $64$ & $17$ & $62$ & $19$ & $60$ & $21$ & $58$ & $23$ & $56$ & $25$ & $53$ & $28$ & $50$ & $31$ & $47$ & $34$ & $44$ & $37$ & $41$ & $40$ & $38$ & $43$ & $35$ & $46$ & $32$ & $49$ & $30$ & $51$ & $28$ & $53$ & $26$ & $55$ & $24$ & $57$ & $21$ & $60$ & $18$ & $63$ & $15$ & $66$ & $12$ \\
$(0,2,3)$ & $5$ & $74$ & $7$ & $72$ & $9$ & $70$ & $11$ & $68$ & \cellcolor[gray]{.8} $13$ & \cellcolor[gray]{.8} $66$ & \cellcolor[gray]{.8} $15$ & \cellcolor[gray]{.8} $63$ & \cellcolor[gray]{.8} $18$ & \cellcolor[gray]{.8} $60$ & \cellcolor[gray]{.8} $21$ & \cellcolor[gray]{.8} $57$ & \cellcolor[gray]{.8} $24$ & \cellcolor[gray]{.8} $54$ & $27$ & $51$ & $30$ & $49$ & $32$ & $47$ & $34$ & $45$ & $36$ & $43$ & $38$ & $41$ & $40$ & $39$ & $42$ & $37$ & $44$ & $35$ & $46$ & $32$ & $49$ & $29$ & $52$ & $26$ & $55$ & $23$ & $58$ & $20$ & $61$ & $17$ & $64$ & $14$ & $67$ & $11$ \\
$(1,1,3)$ & $5$ & $73$ & $8$ & $70$ & $11$ & $67$ & $14$ & $64$ & $17$ & $61$ & $20$ & $60$ & $21$ & $59$ & \cellcolor[gray]{.8} $22$ & \cellcolor[gray]{.8} $58$ & $23$ & $57$ & $24$ & $56$ & $25$ & $53$ & $28$ & $50$ & $31$ & $47$ & $34$ & $44$ & $37$ & $42$ & $39$ & $40$ & $41$ & $38$ & $43$ & $36$ & $45$ & $33$ & $48$ & $30$ & $51$ & $27$ & $54$ & $24$ & $57$ & $21$ & $60$ & $18$ & $63$ & $15$ & $66$ & $12$ \\
$(1,2,2)$ & $5$ & $73$ & $8$ & $70$ & $11$ & $67$ & $14$ & $64$ & $17$ & $61$ & $20$ & $59$ & $22$ & $57$ & $24$ & $55$ & \cellcolor[gray]{.8} $26$ & \cellcolor[gray]{.8} $53$ & \cellcolor[gray]{.8} $28$ & \cellcolor[gray]{.8} $51$ & $30$ & $50$ & $31$ & $49$ & $32$ & $48$ & $33$ & $47$ & $34$ & $44$ & $37$ & $41$ & $40$ & $38$ & $43$ & $35$ & $46$ & $32$ & $49$ & $29$ & $52$ & $26$ & $55$ & $23$ & $58$ & $20$ & $61$ & $17$ & $64$ & $14$ & $67$ & $11$ \\
$(0,1,5)$ & $6$ & $73$ & $8$ & $71$ & $10$ & $69$ & $12$ & $67$ & $14$ & $65$ & \cellcolor[gray]{.8} $16$ & \cellcolor[gray]{.8} $63$ & \cellcolor[gray]{.8} $18$ & \cellcolor[gray]{.8} $61$ & $20$ & $59$ & $22$ & $57$ & $24$ & $55$ & $26$ & $52$ & $29$ & $49$ & $32$ & $46$ & $35$ & $43$ & $38$ & $40$ & $41$ & $37$ & $44$ & $34$ & $47$ & $31$ & $50$ & $28$ & $53$ & $25$ & $56$ & $22$ & $59$ & $19$ & $62$ & $17$ & $64$ & $15$ & $66$ & $13$ & $68$ & $11$ \\
$(0,2,4)$ & $6$ & $73$ & $8$ & $71$ & $10$ & $69$ & $12$ & $67$ & $14$ & $65$ & $16$ & $62$ & $19$ & $59$ & $22$ & $56$ & \cellcolor[gray]{.8} $25$ & \cellcolor[gray]{.8} $53$ & \cellcolor[gray]{.8} $28$ & \cellcolor[gray]{.8} $50$ & \cellcolor[gray]{.8} $31$ & \cellcolor[gray]{.8} $48$ & $33$ & $46$ & $35$ & $44$ & $37$ & $42$ & $39$ & $39$ & $42$ & $36$ & $45$ & $33$ & $48$ & $30$ & $51$ & $28$ & $53$ & $26$ & $55$ & $24$ & $57$ & $22$ & $59$ & $19$ & $62$ & $16$ & $65$ & $13$ & $68$ & $10$ \\
$(0,3,3)$ & $6$ & $73$ & $8$ & $71$ & $10$ & $69$ & $12$ & $67$ & $14$ & $65$ & $16$ & $62$ & $19$ & $59$ & $22$ & $56$ & \cellcolor[gray]{.8} $25$ & \cellcolor[gray]{.8} $53$ & \cellcolor[gray]{.8} $28$ & \cellcolor[gray]{.8} $50$ & \cellcolor[gray]{.8} $31$ & \cellcolor[gray]{.8} $47$ & \cellcolor[gray]{.8} $34$ & \cellcolor[gray]{.8} $44$ & \cellcolor[gray]{.8} $37$ & \cellcolor[gray]{.8} $41$ & $40$ & $38$ & $43$ & $37$ & $44$ & $36$ & $45$ & $35$ & $46$ & $34$ & $47$ & $31$ & $50$ & $28$ & $53$ & $25$ & $56$ & $22$ & $59$ & $19$ & $62$ & $16$ & $65$ & $13$ & $68$ & $10$ \\
$(1,1,4)$ & $6$ & $72$ & $9$ & $69$ & $12$ & $66$ & $15$ & $63$ & $18$ & $60$ & $21$ & $59$ & $22$ & $58$ & \cellcolor[gray]{.8} $23$ & \cellcolor[gray]{.8} $57$ & \cellcolor[gray]{.8} $24$ & \cellcolor[gray]{.8} $56$ & $25$ & $55$ & $26$ & $52$ & $29$ & $49$ & $32$ & $46$ & $35$ & $43$ & $38$ & $40$ & $41$ & $37$ & $44$ & $34$ & $47$ & $31$ & $50$ & $29$ & $52$ & $27$ & $54$ & $25$ & $56$ & $23$ & $58$ & $20$ & $61$ & $17$ & $64$ & $14$ & $67$ & $11$ \\
$(1,2,3)$ & $6$ & $72$ & $9$ & $69$ & $12$ & $66$ & $15$ & $63$ & $18$ & $60$ & $21$ & $58$ & $23$ & $56$ & $25$ & $54$ & $27$ & $52$ & \cellcolor[gray]{.8} $29$ & \cellcolor[gray]{.8} $50$ & \cellcolor[gray]{.8} $31$ & \cellcolor[gray]{.8} $48$ & $33$ & $46$ & $35$ & $44$ & $37$ & $42$ & $39$ & $40$ & $41$ & $38$ & $43$ & $36$ & $45$ & $34$ & $47$ & $31$ & $50$ & $28$ & $53$ & $25$ & $56$ & $22$ & $59$ & $19$ & $62$ & $16$ & $65$ & $13$ & $68$ & $10$ \\
$(2,2,2)$ & $6$ & $72$ & $9$ & $69$ & $12$ & $66$ & $15$ & $63$ & $18$ & $60$ & $21$ & $57$ & $24$ & $54$ & $27$ & $51$ & $30$ & $48$ & $33$ & $45$ & $36$ & $45$ & \cellcolor[gray]{.8} $36$ & \cellcolor[gray]{.8} $45$ & $36$ & $45$ & $36$ & $45$ & $36$ & $42$ & $39$ & $39$ & $42$ & $36$ & $45$ & $33$ & $48$ & $30$ & $51$ & $27$ & $54$ & $24$ & $57$ & $21$ & $60$ & $18$ & $63$ & $15$ & $66$ & $12$ & $69$ & $9$ \\
$(0,2,5)$ & $7$ & $72$ & $9$ & $70$ & $11$ & $68$ & $13$ & $66$ & $15$ & $64$ & $17$ & $61$ & $20$ & $58$ & $23$ & $55$ & $26$ & $52$ & $29$ & $49$ & \cellcolor[gray]{.8} $32$ & \cellcolor[gray]{.8} $47$ & \cellcolor[gray]{.8} $34$ & \cellcolor[gray]{.8} $45$ & $36$ & $43$ & $38$ & $41$ & $40$ & $38$ & $43$ & $35$ & $46$ & $32$ & $49$ & $29$ & $52$ & $26$ & $55$ & $23$ & $58$ & $20$ & $61$ & $17$ & $64$ & $15$ & $66$ & $13$ & $68$ & $11$ & $70$ & $9$ \\
$(0,3,4)$ & $7$ & $72$ & $9$ & $70$ & $11$ & $68$ & $13$ & $66$ & $15$ & $64$ & $17$ & $61$ & $20$ & $58$ & $23$ & $55$ & $26$ & $52$ & $29$ & $49$ & $32$ & $46$ & $35$ & $43$ & \cellcolor[gray]{.8} $38$ & \cellcolor[gray]{.8} $40$ & \cellcolor[gray]{.8} $41$ & \cellcolor[gray]{.8} $37$ & \cellcolor[gray]{.8} $44$ & \cellcolor[gray]{.8} $35$ & $46$ & $33$ & $48$ & $31$ & $50$ & $29$ & $52$ & $27$ & $54$ & $25$ & $56$ & $23$ & $58$ & $21$ & $60$ & $18$ & $63$ & $15$ & $66$ & $12$ & $69$ & $9$ \\
$(1,1,5)$ & $7$ & $71$ & $10$ & $68$ & $13$ & $65$ & $16$ & $62$ & $19$ & $59$ & $22$ & $58$ & $23$ & $57$ & $24$ & $56$ & \cellcolor[gray]{.8} $25$ & \cellcolor[gray]{.8} $55$ & $26$ & $54$ & $27$ & $51$ & $30$ & $48$ & $33$ & $45$ & $36$ & $42$ & $39$ & $39$ & $42$ & $36$ & $45$ & $33$ & $48$ & $30$ & $51$ & $27$ & $54$ & $24$ & $57$ & $21$ & $60$ & $18$ & $63$ & $16$ & $65$ & $14$ & $67$ & $12$ & $69$ & $10$ \\
$(1,2,4)$ & $7$ & $71$ & $10$ & $68$ & $13$ & $65$ & $16$ & $62$ & $19$ & $59$ & $22$ & $57$ & $24$ & $55$ & $26$ & $53$ & $28$ & $51$ & $30$ & $49$ & \cellcolor[gray]{.8} $32$ & \cellcolor[gray]{.8} $47$ & \cellcolor[gray]{.8} $34$ & \cellcolor[gray]{.8} $45$ & $36$ & $43$ & $38$ & $41$ & $40$ & $38$ & $43$ & $35$ & $46$ & $32$ & $49$ & $29$ & $52$ & $27$ & $54$ & $25$ & $56$ & $23$ & $58$ & $21$ & $60$ & $18$ & $63$ & $15$ & $66$ & $12$ & $69$ & $9$ \\
$(1,3,3)$ & $7$ & $71$ & $10$ & $68$ & $13$ & $65$ & $16$ & $62$ & $19$ & $59$ & $22$ & $57$ & $24$ & $55$ & $26$ & $53$ & $28$ & $51$ & $30$ & $49$ & $32$ & $46$ & $35$ & $43$ & \cellcolor[gray]{.8} $38$ & \cellcolor[gray]{.8} $40$ & \cellcolor[gray]{.8} $41$ & \cellcolor[gray]{.8} $37$ & \cellcolor[gray]{.8} $44$ & \cellcolor[gray]{.8} $36$ & $45$ & $35$ & $46$ & $34$ & $47$ & $33$ & $48$ & $30$ & $51$ & $27$ & $54$ & $24$ & $57$ & $21$ & $60$ & $18$ & $63$ & $15$ & $66$ & $12$ & $69$ & $9$ \\
$(2,2,3)$ & $7$ & $71$ & $10$ & $68$ & $13$ & $65$ & $16$ & $62$ & $19$ & $59$ & $22$ & $56$ & $25$ & $53$ & $28$ & $50$ & $31$ & $47$ & $34$ & $44$ & $37$ & $43$ & $38$ & $42$ & \cellcolor[gray]{.8} $39$ & \cellcolor[gray]{.8} $41$ & $40$ & $40$ & $41$ & $38$ & $43$ & $36$ & $45$ & $34$ & $47$ & $32$ & $49$ & $29$ & $52$ & $26$ & $55$ & $23$ & $58$ & $20$ & $61$ & $17$ & $64$ & $14$ & $67$ & $11$ & $70$ & $8$ \\
$(0,3,5)$ & $8$ & $71$ & $10$ & $69$ & $12$ & $67$ & $14$ & $65$ & $16$ & $63$ & $18$ & $60$ & $21$ & $57$ & $24$ & $54$ & $27$ & $51$ & $30$ & $48$ & $33$ & $45$ & $36$ & $42$ & $39$ & $39$ & $42$ & $36$ & \cellcolor[gray]{.8} $45$ & \cellcolor[gray]{.8} $34$ & \cellcolor[gray]{.8} $47$ & \cellcolor[gray]{.8} $32$ & $49$ & $30$ & $51$ & $28$ & $53$ & $25$ & $56$ & $22$ & $59$ & $19$ & $62$ & $16$ & $65$ & $14$ & $67$ & $12$ & $69$ & $10$ & $71$ & $8$ \\
$(0,4,4)$ & $8$ & $71$ & $10$ & $69$ & $12$ & $67$ & $14$ & $65$ & $16$ & $63$ & $18$ & $60$ & $21$ & $57$ & $24$ & $54$ & $27$ & $51$ & $30$ & $48$ & $33$ & $45$ & $36$ & $42$ & $39$ & $39$ & $42$ & $36$ & $45$ & $33$ & $48$ & $30$ & $51$ & $27$ & $54$ & $24$ & \cellcolor[gray]{.8} $57$ & \cellcolor[gray]{.8} $23$ & $58$ & $22$ & $59$ & $21$ & $60$ & $20$ & $61$ & $17$ & $64$ & $14$ & $67$ & $11$ & $70$ & $8$ \\
$(1,2,5)$ & $8$ & $70$ & $11$ & $67$ & $14$ & $64$ & $17$ & $61$ & $20$ & $58$ & $23$ & $56$ & $25$ & $54$ & $27$ & $52$ & $29$ & $50$ & $31$ & $48$ & $33$ & $46$ & \cellcolor[gray]{.8} $35$ & \cellcolor[gray]{.8} $44$ & \cellcolor[gray]{.8} $37$ & \cellcolor[gray]{.8} $42$ & $39$ & $40$ & $41$ & $37$ & $44$ & $34$ & $47$ & $31$ & $50$ & $28$ & $53$ & $25$ & $56$ & $22$ & $59$ & $19$ & $62$ & $16$ & $65$ & $14$ & $67$ & $12$ & $69$ & $10$ & $71$ & $8$ \\
$(1,3,4)$ & $8$ & $70$ & $11$ & $67$ & $14$ & $64$ & $17$ & $61$ & $20$ & $58$ & $23$ & $56$ & $25$ & $54$ & $27$ & $52$ & $29$ & $50$ & $31$ & $48$ & $33$ & $45$ & $36$ & $42$ & $39$ & $39$ & $42$ & $36$ & \cellcolor[gray]{.8} $45$ & \cellcolor[gray]{.8} $34$ & \cellcolor[gray]{.8} $47$ & \cellcolor[gray]{.8} $32$ & $49$ & $30$ & $51$ & $28$ & $53$ & $26$ & $55$ & $24$ & $57$ & $22$ & $59$ & $20$ & $61$ & $17$ & $64$ & $14$ & $67$ & $11$ & $70$ & $8$ \\
$(2,2,4)$ & $8$ & $70$ & $11$ & $67$ & $14$ & $64$ & $17$ & $61$ & $20$ & $58$ & $23$ & $55$ & $26$ & $52$ & $29$ & $49$ & $32$ & $46$ & $35$ & $43$ & $38$ & $42$ & $39$ & $41$ & \cellcolor[gray]{.8} $40$ & \cellcolor[gray]{.8} $40$ & \cellcolor[gray]{.8} $41$ & \cellcolor[gray]{.8} $39$ & $42$ & $36$ & $45$ & $33$ & $48$ & $30$ & $51$ & $27$ & $54$ & $25$ & $56$ & $23$ & $58$ & $21$ & $60$ & $19$ & $62$ & $16$ & $65$ & $13$ & $68$ & $10$ & $71$ & $7$ \\
$(2,3,3)$ & $8$ & $70$ & $11$ & $67$ & $14$ & $64$ & $17$ & $61$ & $20$ & $58$ & $23$ & $55$ & $26$ & $52$ & $29$ & $49$ & $32$ & $46$ & $35$ & $43$ & $38$ & $41$ & $40$ & $39$ & $42$ & $37$ & $44$ & $35$ & \cellcolor[gray]{.8} $46$ & \cellcolor[gray]{.8} $34$ & \cellcolor[gray]{.8} $47$ & \cellcolor[gray]{.8} $33$ & $48$ & $32$ & $49$ & $31$ & $50$ & $28$ & $53$ & $25$ & $56$ & $22$ & $59$ & $19$ & $62$ & $16$ & $65$ & $13$ & $68$ & $10$ & $71$ & $7$ \\
$(0,4,5)$ & $9$ & $70$ & $11$ & $68$ & $13$ & $66$ & $15$ & $64$ & $17$ & $62$ & $19$ & $59$ & $22$ & $56$ & $25$ & $53$ & $28$ & $50$ & $31$ & $47$ & $34$ & $44$ & $37$ & $41$ & $40$ & $38$ & $43$ & $35$ & $46$ & $32$ & $49$ & $29$ & $52$ & $26$ & $55$ & $23$ & \cellcolor[gray]{.8} $58$ & \cellcolor[gray]{.8} $21$ & \cellcolor[gray]{.8} $60$ & \cellcolor[gray]{.8} $19$ & $62$ & $17$ & $64$ & $15$ & $66$ & $13$ & $68$ & $11$ & $70$ & $9$ & $72$ & $7$ \\
$(1,3,5)$ & $9$ & $69$ & $12$ & $66$ & $15$ & $63$ & $18$ & $60$ & $21$ & $57$ & $24$ & $55$ & $26$ & $53$ & $28$ & $51$ & $30$ & $49$ & $32$ & $47$ & $34$ & $44$ & $37$ & $41$ & $40$ & $38$ & $43$ & $35$ & $46$ & $33$ & \cellcolor[gray]{.8} $48$ & \cellcolor[gray]{.8} $31$ & \cellcolor[gray]{.8} $50$ & \cellcolor[gray]{.8} $29$ & $52$ & $27$ & $54$ & $24$ & $57$ & $21$ & $60$ & $18$ & $63$ & $15$ & $66$ & $13$ & $68$ & $11$ & $70$ & $9$ & $72$ & $7$ \\
$(1,4,4)$ & $9$ & $69$ & $12$ & $66$ & $15$ & $63$ & $18$ & $60$ & $21$ & $57$ & $24$ & $55$ & $26$ & $53$ & $28$ & $51$ & $30$ & $49$ & $32$ & $47$ & $34$ & $44$ & $37$ & $41$ & $40$ & $38$ & $43$ & $35$ & $46$ & $32$ & $49$ & $29$ & $52$ & $26$ & $55$ & $23$ & \cellcolor[gray]{.8} $58$ & \cellcolor[gray]{.8} $22$ & $59$ & $21$ & $60$ & $20$ & $61$ & $19$ & $62$ & $16$ & $65$ & $13$ & $68$ & $10$ & $71$ & $7$ \\
$(2,2,5)$ & $9$ & $69$ & $12$ & $66$ & $15$ & $63$ & $18$ & $60$ & $21$ & $57$ & $24$ & $54$ & $27$ & $51$ & $30$ & $48$ & $33$ & $45$ & $36$ & $42$ & $39$ & $41$ & $40$ & $40$ & $41$ & $39$ & \cellcolor[gray]{.8} $42$ & \cellcolor[gray]{.8} $38$ & $43$ & $35$ & $46$ & $32$ & $49$ & $29$ & $52$ & $26$ & $55$ & $23$ & $58$ & $20$ & $61$ & $17$ & $64$ & $14$ & $67$ & $12$ & $69$ & $10$ & $71$ & $8$ & $73$ & $6$ \\
$(2,3,4)$ & $9$ & $69$ & $12$ & $66$ & $15$ & $63$ & $18$ & $60$ & $21$ & $57$ & $24$ & $54$ & $27$ & $51$ & $30$ & $48$ & $33$ & $45$ & $36$ & $42$ & $39$ & $40$ & $41$ & $38$ & $43$ & $36$ & $45$ & $34$ & $47$ & $32$ & $49$ & $30$ & \cellcolor[gray]{.8} $51$ & \cellcolor[gray]{.8} $28$ & \cellcolor[gray]{.8} $53$ & \cellcolor[gray]{.8} $26$ & $55$ & $24$ & $57$ & $22$ & $59$ & $20$ & $61$ & $18$ & $63$ & $15$ & $66$ & $12$ & $69$ & $9$ & $72$ & $6$ \\
$(3,3,3)$ & $9$ & $69$ & $12$ & $66$ & $15$ & $63$ & $18$ & $60$ & $21$ & $57$ & $24$ & $54$ & $27$ & $51$ & $30$ & $48$ & $33$ & $45$ & $36$ & $42$ & $39$ & $39$ & $42$ & $36$ & $45$ & $33$ & $48$ & $30$ & $51$ & $30$ & $51$ & $30$ & \cellcolor[gray]{.8} $51$ & \cellcolor[gray]{.8} $30$ & $51$ & $30$ & $51$ & $27$ & $54$ & $24$ & $57$ & $21$ & $60$ & $18$ & $63$ & $15$ & $66$ & $12$ & $69$ & $9$ & $72$ & $6$ \\
$(0,5,5)$ & $10$ & $69$ & $12$ & $67$ & $14$ & $65$ & $16$ & $63$ & $18$ & $61$ & $20$ & $58$ & $23$ & $55$ & $26$ & $52$ & $29$ & $49$ & $32$ & $46$ & $35$ & $43$ & $38$ & $40$ & $41$ & $37$ & $44$ & $34$ & $47$ & $31$ & $50$ & $28$ & $53$ & $25$ & $56$ & $22$ & $59$ & $19$ & $62$ & $16$ & $65$ & $13$ & $68$ & $10$ & \cellcolor[gray]{.8} $71$ & \cellcolor[gray]{.8} $9$ & $72$ & $8$ & $73$ & $7$ & $74$ & $6$ \\
$(1,4,5)$ & $10$ & $68$ & $13$ & $65$ & $16$ & $62$ & $19$ & $59$ & $22$ & $56$ & $25$ & $54$ & $27$ & $52$ & $29$ & $50$ & $31$ & $48$ & $33$ & $46$ & $35$ & $43$ & $38$ & $40$ & $41$ & $37$ & $44$ & $34$ & $47$ & $31$ & $50$ & $28$ & $53$ & $25$ & $56$ & $22$ & $59$ & $20$ & \cellcolor[gray]{.8} $61$ & \cellcolor[gray]{.8} $18$ & \cellcolor[gray]{.8} $63$ & \cellcolor[gray]{.8} $16$ & $65$ & $14$ & $67$ & $12$ & $69$ & $10$ & $71$ & $8$ & $73$ & $6$ \\
$(2,3,5)$ & $10$ & $68$ & $13$ & $65$ & $16$ & $62$ & $19$ & $59$ & $22$ & $56$ & $25$ & $53$ & $28$ & $50$ & $31$ & $47$ & $34$ & $44$ & $37$ & $41$ & $40$ & $39$ & $42$ & $37$ & $44$ & $35$ & $46$ & $33$ & $48$ & $31$ & $50$ & $29$ & $52$ & $27$ & \cellcolor[gray]{.8} $54$ & \cellcolor[gray]{.8} $25$ & \cellcolor[gray]{.8} $56$ & \cellcolor[gray]{.8} $22$ & $59$ & $19$ & $62$ & $16$ & $65$ & $13$ & $68$ & $11$ & $70$ & $9$ & $72$ & $7$ & $74$ & $5$ \\
$(2,4,4)$ & $10$ & $68$ & $13$ & $65$ & $16$ & $62$ & $19$ & $59$ & $22$ & $56$ & $25$ & $53$ & $28$ & $50$ & $31$ & $47$ & $34$ & $44$ & $37$ & $41$ & $40$ & $39$ & $42$ & $37$ & $44$ & $35$ & $46$ & $33$ & $48$ & $30$ & $51$ & $27$ & $54$ & $24$ & $57$ & $21$ & $60$ & $20$ & \cellcolor[gray]{.8} $61$ & \cellcolor[gray]{.8} $19$ & $62$ & $18$ & $63$ & $17$ & $64$ & $14$ & $67$ & $11$ & $70$ & $8$ & $73$ & $5$ \\
$(3,3,4)$ & $10$ & $68$ & $13$ & $65$ & $16$ & $62$ & $19$ & $59$ & $22$ & $56$ & $25$ & $53$ & $28$ & $50$ & $31$ & $47$ & $34$ & $44$ & $37$ & $41$ & $40$ & $38$ & $43$ & $35$ & $46$ & $32$ & $49$ & $29$ & $52$ & $28$ & $53$ & $27$ & $54$ & $26$ & \cellcolor[gray]{.8} $55$ & \cellcolor[gray]{.8} $25$ & \cellcolor[gray]{.8} $56$ & \cellcolor[gray]{.8} $23$ & $58$ & $21$ & $60$ & $19$ & $62$ & $17$ & $64$ & $14$ & $67$ & $11$ & $70$ & $8$ & $73$ & $5$ \\
$(1,5,5)$ & $11$ & $67$ & $14$ & $64$ & $17$ & $61$ & $20$ & $58$ & $23$ & $55$ & $26$ & $53$ & $28$ & $51$ & $30$ & $49$ & $32$ & $47$ & $34$ & $45$ & $36$ & $42$ & $39$ & $39$ & $42$ & $36$ & $45$ & $33$ & $48$ & $30$ & $51$ & $27$ & $54$ & $24$ & $57$ & $21$ & $60$ & $18$ & $63$ & $15$ & $66$ & $12$ & $69$ & $9$ & \cellcolor[gray]{.8} $72$ & \cellcolor[gray]{.8} $8$ & \cellcolor[gray]{.8} $73$ & \cellcolor[gray]{.8} $7$ & $74$ & $6$ & $75$ & $5$ \\
$(2,4,5)$ & $11$ & $67$ & $14$ & $64$ & $17$ & $61$ & $20$ & $58$ & $23$ & $55$ & $26$ & $52$ & $29$ & $49$ & $32$ & $46$ & $35$ & $43$ & $38$ & $40$ & $41$ & $38$ & $43$ & $36$ & $45$ & $34$ & $47$ & $32$ & $49$ & $29$ & $52$ & $26$ & $55$ & $23$ & $58$ & $20$ & $61$ & $18$ & $63$ & $16$ & $65$ & $14$ & \cellcolor[gray]{.8} $67$ & \cellcolor[gray]{.8} $12$ & $69$ & $10$ & $71$ & $8$ & $73$ & $6$ & $75$ & $4$ \\
$(3,3,5)$ & $11$ & $67$ & $14$ & $64$ & $17$ & $61$ & $20$ & $58$ & $23$ & $55$ & $26$ & $52$ & $29$ & $49$ & $32$ & $46$ & $35$ & $43$ & $38$ & $40$ & $41$ & $37$ & $44$ & $34$ & $47$ & $31$ & $50$ & $28$ & $53$ & $27$ & $54$ & $26$ & $55$ & $25$ & $56$ & $24$ & \cellcolor[gray]{.8} $57$ & \cellcolor[gray]{.8} $21$ & \cellcolor[gray]{.8} $60$ & \cellcolor[gray]{.8} $18$ & \cellcolor[gray]{.8} $63$ & \cellcolor[gray]{.8} $15$ & \cellcolor[gray]{.8} $66$ & \cellcolor[gray]{.8} $12$ & $69$ & $10$ & $71$ & $8$ & $73$ & $6$ & $75$ & $4$ \\
$(3,4,4)$ & $11$ & $67$ & $14$ & $64$ & $17$ & $61$ & $20$ & $58$ & $23$ & $55$ & $26$ & $52$ & $29$ & $49$ & $32$ & $46$ & $35$ & $43$ & $38$ & $40$ & $41$ & $37$ & $44$ & $34$ & $47$ & $31$ & $50$ & $28$ & $53$ & $26$ & $55$ & $24$ & $57$ & $22$ & $59$ & $20$ & $61$ & $19$ & \cellcolor[gray]{.8} $62$ & \cellcolor[gray]{.8} $18$ & \cellcolor[gray]{.8} $63$ & \cellcolor[gray]{.8} $17$ & $64$ & $16$ & $65$ & $13$ & $68$ & $10$ & $71$ & $7$ & $74$ & $4$ \\
$(2,5,5)$ & $12$ & $66$ & $15$ & $63$ & $18$ & $60$ & $21$ & $57$ & $24$ & $54$ & $27$ & $51$ & $30$ & $48$ & $33$ & $45$ & $36$ & $42$ & $39$ & $39$ & $42$ & $37$ & $44$ & $35$ & $46$ & $33$ & $48$ & $31$ & $50$ & $28$ & $53$ & $25$ & $56$ & $22$ & $59$ & $19$ & $62$ & $16$ & $65$ & $13$ & $68$ & $10$ & $71$ & $7$ & $74$ & $6$ & \cellcolor[gray]{.8} $75$ & \cellcolor[gray]{.8} $5$ & \cellcolor[gray]{.8} $76$ & \cellcolor[gray]{.8} $4$ & $77$ & $3$ \\
$(3,4,5)$ & $12$ & $66$ & $15$ & $63$ & $18$ & $60$ & $21$ & $57$ & $24$ & $54$ & $27$ & $51$ & $30$ & $48$ & $33$ & $45$ & $36$ & $42$ & $39$ & $39$ & $42$ & $36$ & $45$ & $33$ & $48$ & $30$ & $51$ & $27$ & $54$ & $25$ & $56$ & $23$ & $58$ & $21$ & $60$ & $19$ & $62$ & $17$ & $64$ & $15$ & $66$ & $13$ & \cellcolor[gray]{.8} $68$ & \cellcolor[gray]{.8} $11$ & \cellcolor[gray]{.8} $70$ & \cellcolor[gray]{.8} $9$ & $72$ & $7$ & $74$ & $5$ & $76$ & $3$ \\
$(4,4,4)$ & $12$ & $66$ & $15$ & $63$ & $18$ & $60$ & $21$ & $57$ & $24$ & $54$ & $27$ & $51$ & $30$ & $48$ & $33$ & $45$ & $36$ & $42$ & $39$ & $39$ & $42$ & $36$ & $45$ & $33$ & $48$ & $30$ & $51$ & $27$ & $54$ & $24$ & $57$ & $21$ & $60$ & $18$ & $63$ & $15$ & $66$ & $15$ & $66$ & $15$ & \cellcolor[gray]{.8} $66$ & \cellcolor[gray]{.8} $15$ & \cellcolor[gray]{.8} $66$ & \cellcolor[gray]{.8} $15$ & $66$ & $12$ & $69$ & $9$ & $72$ & $6$ & $75$ & $3$ \\
$(3,5,5)$ & $13$ & $65$ & $16$ & $62$ & $19$ & $59$ & $22$ & $56$ & $25$ & $53$ & $28$ & $50$ & $31$ & $47$ & $34$ & $44$ & $37$ & $41$ & $40$ & $38$ & $43$ & $35$ & $46$ & $32$ & $49$ & $29$ & $52$ & $26$ & $55$ & $24$ & $57$ & $22$ & $59$ & $20$ & $61$ & $18$ & $63$ & $15$ & $66$ & $12$ & $69$ & $9$ & $72$ & $6$ & $75$ & $5$ & $76$ & $4$ & \cellcolor[gray]{.8} $77$ & \cellcolor[gray]{.8} $3$ & \cellcolor[gray]{.8} $78$ & \cellcolor[gray]{.8} $2$ \\
$(4,4,5)$ & $13$ & $65$ & $16$ & $62$ & $19$ & $59$ & $22$ & $56$ & $25$ & $53$ & $28$ & $50$ & $31$ & $47$ & $34$ & $44$ & $37$ & $41$ & $40$ & $38$ & $43$ & $35$ & $46$ & $32$ & $49$ & $29$ & $52$ & $26$ & $55$ & $23$ & $58$ & $20$ & $61$ & $17$ & $64$ & $14$ & $67$ & $13$ & $68$ & $12$ & $69$ & $11$ & $70$ & $10$ & \cellcolor[gray]{.8} $71$ & \cellcolor[gray]{.8} $8$ & \cellcolor[gray]{.8} $73$ & \cellcolor[gray]{.8} $6$ & $75$ & $4$ & $77$ & $2$ \\
$(4,5,5)$ & $14$ & $64$ & $17$ & $61$ & $20$ & $58$ & $23$ & $55$ & $26$ & $52$ & $29$ & $49$ & $32$ & $46$ & $35$ & $43$ & $38$ & $40$ & $41$ & $37$ & $44$ & $34$ & $47$ & $31$ & $50$ & $28$ & $53$ & $25$ & $56$ & $22$ & $59$ & $19$ & $62$ & $16$ & $65$ & $13$ & $68$ & $11$ & $70$ & $9$ & $72$ & $7$ & $74$ & $5$ & $76$ & $4$ & $77$ & $3$ & $78$ & $2$ & \cellcolor[gray]{.8} $79$ & \cellcolor[gray]{.8} $1$ \\
$(5,5,5)$ & $15$ & $63$ & $18$ & $60$ & $21$ & $57$ & $24$ & $54$ & $27$ & $51$ & $30$ & $48$ & $33$ & $45$ & $36$ & $42$ & $39$ & $39$ & $42$ & $36$ & $45$ & $33$ & $48$ & $30$ & $51$ & $27$ & $54$ & $24$ & $57$ & $21$ & $60$ & $18$ & $63$ & $15$ & $66$ & $12$ & $69$ & $9$ & $72$ & $6$ & $75$ & $3$ & $78$ & $0$ & $81$ & $0$ & $81$ & $0$ & $81$ & $0$ & \cellcolor[gray]{.8} $81$ & \cellcolor[gray]{.8} $0$ \\
\hline
 &   &  $78$ & $3$ & $75$ & $6$ & $73$ & $8$ & $70$ & $11$ & $66$ & $15$ & $63$ & $18$ & $60$ & $21$ & $57$ & $24$ & $53$ & $28$ & $50$ & $31$ & $47$ & $34$ & $44$ & $37$ & $40$ & $41$ & $37$ & $44$ & $34$ & $47$ & $31$ & $50$ & $28$ & $53$ & $25$ & $56$ & $21$ & $60$ & $18$ & $63$ & $15$ & $66$ & $11$ & $70$ & $8$ & $73$ & $5$ & $76$ & $3$ & $78$ &  \\
\end{tabular}
}
\end{center}
The highlighted entries show $T_w(\vec{\delta})$ for
$$c=
(3,6,8,11,15,18,21,24,28,31,34,37,41,44,47,50,53,56,60,63,66,70,73,76,78).$$ 
As in the $r=4$ case, this is not unimaginative, but it is steady with
respect to $T_w(\vec{\delta})$.

Now, first observe that the first, fourth, $17$th, $24$th, $25$th and $26$th columns can have all their rows dropped using Lemma \ref{rulesii-v}(b),
 and the $18$th column can have its rows dropped with Lemma \ref{rulesii-v}(a).
  We can then drop the remaining row in the third column, and then the remaining row in the second column.
   Likewise, we can drop the two remaining rows in the $23$rd column.
    Next, in the $22$nd column we have the rows $(2,4,5)$, $(3,3,5)$ and $(4,4,4)$ remaining, with $\delta_{22}=4$. 
    Then by Lemma \ref{rulesii-v}(d)  with $j=5$ we can drop the $(4,4,4)$ row, and then by Lemma \ref{rulesii-v}(c) the two remaining rows.
We again use Lemma \ref{rulesii-v}(c) to drop the two remaining rows in the $21$st column, and then again in the $20$th column, and then the $19$th column
(recalling that we have already dropped the rows in the $18$th column).
Thus, it remains to consider the rows which only appear in the fifth through $16$th columns.

Next, note that in the sixth column, the rows appearing are $(0,1,4)$, $(0,1,5)$, $(0,2,3)$ and $(1,1,1)$. 
Since $\delta_6=1$, we can apply Lemma \ref{rulesii-v}(d) with $j=0$ to drop row $(1,1,1)$. 
But then the remaining rows in the seventh column are $(0,1,5)$, $(0,2,3)$ and $(1,1,2)$, and $\delta_7=1$, so using Lemma \ref{rulesii-v}(d) with $j=2$ we can drop the $(0,1,5)$ 
row, and this allows us to drop the two remaining rows in the seventh column, and subsequently in the sixth, fifth and eighth columns, by using Lemma \ref{rulesii-v}(c) repeatedly.
 Similarly, the rows appearing in the $10$th column are $(0,2,4)$, $(0,3,3)$, $(1,2,2)$ and $(1,2,3)$, and $\delta_{10}=1$,
so we can use Lemma \ref{rulesii-v}(d) with $j=2$ to drop row $(0,3,3)$. 
The remaining rows in the ninth column are then $(0,2,4)$, $(1,1,5)$ and $(1,2,2)$, and $\delta_9=1$, so we can again apply Lemma \ref{rulesii-v}(d) with $j=2$ to drop
$(1,1,5)$, and then Lemma \ref{rulesii-v}(c) to drop the remaining rows in the ninth, and subsequently the $10$th, $11$th and $12$th columns. 
We can apply the same procedure a third time to the $13$th and $14$th columns, first using Lemma \ref{rulesii-v}(d) in the $13$th column with $j=3$ to drop row $(2,2,4)$,
and then again using Lemma \ref{rulesii-v}(d) with $j=3$, but this time in the $14$th column, to drop row $(2,2,5)$. 
We can then use Lemma \ref{rulesii-v}(c) to drop the remaining rows in the $13$th and $14$th columns.

We are left with the rows supported only in the $15$th and $16$th columns, which are $(0,3,5)$, $(1,3,4)$ and $(2,3,3)$. 
We can finally drop these using Lemma \ref{rulesvi} and Lemma \ref{rulesii-v}(c) .
\end{ex}

Combining Examples \ref{ex:m3-r3}, \ref{ex:m3-r4} and \ref{ex:m3-r5} with Proposition \ref{prop:injective-extend}, we conclude:

\begin{cor}\label{cor:m3-ranges} The Maximal Rank Conjecture holds for
$m=3$, and
\begin{ilist}
\itm $r=3$ with $g \geq 7$;
\itm $r=4$ with $g \geq 16$;
\itm $r=5$ with $g \geq 26$;
\end{ilist}

Moreover, in these cases the locus of $\overline{\cM}_g$ consisting of chains of genus-$1$ curves is not in the closure of the locus in $\cM_g$
 where the appropriate maximal rank condition fails.
\end{cor}

Note that (subject to the hypothesis $r+g-d>0$) Corollary \ref{cor:surj} covers all $m=3$ cases with $r=3$ and $g \leq 6$, with $r=4$ and $g \leq 9$, 
and with $r=5$ and $g \leq 12$ (as well a number of additional cases).
Thus, there are no missing cases for $r=3$, and a relatively small number for $r=4$, but a rather significant number for $r=5$. 
We expect that any  given one of these cases can be handled as above, but do not see any simple way of handling them all simultaneously.

\begin{rem}\label{rem:m3-compare}
Comparing to previously known injectivity results for $m=3$, Larson \cite{la4} obtains injectivity for $r=3$ when $g \geq 9$, for $r=4$
when $g \geq 19$, and for $r=5$ when $g \geq 35$. 
Jensen and Payne \cite{j-p3} obtain all cases for $r=3$ and $r=4$, but in $r=5$ only treat the case $\rho=0$, which translates to $g \geq 30$.
\end{rem}

\begin{rem}\label{rem:r4-g14} It is interesting to note that for $r=4$ and $g=14$, all cases are injective; in fact, the smallest allowable $d$,
which is $d=16$, gives a case which is both injective and surjective.
However, the case $g=15$, $d=16$ is not injective, which is why we are forced to start with $g=16$ above. Indeed, the noninjective genus-$15$ case,
together with Proposition \ref{prop:injective-extend}, imply that we cannot treat the $g=14$, $d=16$ case with any extendable $(g,r,d)$-sequence
(however, it is not difficult to treat with a non-extendable sequence).
\end{rem}

\bibliographystyle{amsalpha}
\bibliography{gen}
\newcommand{\noopsort}[1]{} \newcommand{\printfirst}[2]{#1}
  \newcommand{\singleletter}[1]{#1} \newcommand{\switchargs}[2]{#2#1}
\providecommand{\bysame}{\leavevmode\hbox to3em{\hrulefill}\thinspace}
\providecommand{\MR}{\relax\ifhmode\unskip\space\fi MR }
\providecommand{\MRhref}[2]{%
  \href{http://www.ams.org/mathscinet-getitem?mr=#1}{#2}
}
\providecommand{\href}[2]{#2}

\end{document}